\documentclass[a4paper]{article}

\usepackage[english]{babel}
\usepackage[utf8x]{inputenc}
\usepackage[T1]{fontenc}
\usepackage{bm}
\usepackage{authblk}
\usepackage[a4paper,top=3cm,bottom=2cm,left=3cm,right=3cm,marginparwidth=1.75cm]{geometry}
\usepackage{bbm}
\usepackage[normalem]{ulem}
\usepackage{amsmath}
\usepackage{amsthm}
\usepackage[shortlabels]{enumitem} 
\usepackage{graphicx}
\usepackage[colorinlistoftodos]{todonotes}
\usepackage[colorlinks=true, allcolors=blue]{hyperref}
\newtheorem{theorem}{Theorem}[section]

\newtheorem{lemma}[theorem]{Lemma}
\newtheorem{definition}[theorem]{Definition}
\newtheorem{prop}[theorem]{Proposition}
\usepackage{amssymb}
\usepackage{mathrsfs}
\newtheorem*{assumption}{Assumption}
\newtheorem{remark}{Remark}[section]

\usepackage{dsfont}

\def\XXint#1#2#3{{\setbox0=\hbox{$#1{#2#3}{\int}$ }
		\vcenter{\hbox{$#2#3$ }}\kern-.6\wd0}}

\usepackage{mathtools}

\graphicspath{ {images/} }

\usepackage{xcolor}

\newcommand{\tr}[0]{\textup{tr}}

\newcommand{\h}{\hspace}

\newcommand{\p}{\partial}
\newcommand{\RN}[1]{%
  \textup{\uppercase\expandafter{\romannumeral#1}}%
}
\allowdisplaybreaks

\usepackage{hyperref}
\newcommand{\rcm}[1]{{\color{black}#1}}

\DeclareMathOperator*{\bbotimes}{\text{\raisebox{0.25ex}{\scalebox{0.75}{$\bigotimes$}}}}

\numberwithin{equation}{section}

\DeclareFontFamily{U}{mathx}{}
\DeclareFontShape{U}{mathx}{m}{n}{<-> mathx10}{}
\DeclareSymbolFont{mathx}{U}{mathx}{m}{n}
\DeclareMathAccent{\widehat}{0}{mathx}{"70}
\DeclareMathAccent{\widecheck}{0}{mathx}{"71}

\newcommand{\mycomment}[1]{}

\usepackage{scalerel,stackengine}
\newcommand\pig[1]{\scalerel*[5pt]{\big#1}{%
		\ensurestackMath{\addstackgap[1.5pt]{\big#1}}}}

\newcommand\pigr[1]{\mathclose{\pig{#1}}}
\title{Viscosity Solutions of a class of Second Order Hamilton-Jacobi-Bellman Equations in the Wasserstein Space}
\date{}
\author[,1]{\normalsize Hang Cheung\footnote{E-mail: hang.cheung@ucalgary.ca}}
\author[,2]{\normalsize Ho Man Tai\footnote{E-mail: homan.tai@dcu.ie}}
\author[,1]{\normalsize Jinniao Qiu\footnote{E-mail: jinniao.qiu@ucalgary.ca}}
\affil[1]{\small\it Department of Mathematics and Statistics, University of Calgary, Canada}
\affil[2]{\small\it School of Mathematical Sciences, Dublin City University, Ireland}
\begin{document}
	\maketitle
	\begin{abstract}

 This paper is devoted to solving a class of second order Hamilton-Jacobi-Bellman (HJB) equations in the Wasserstein space, associated with mean field control problems involving common noise. {\color{black}The well-posedness of viscosity solutions to the HJB equation under a new notion is established under general assumptions on the coefficients.} Our approach adopts the smooth metric developed by Bayraktar, Ekren, and Zhang [Proc. Amer. Math. Soc. (2023)] as our gauge function for the purpose of smooth variational principle used in the proof of comparison theorem. Further estimates and regularity of the metric, including a novel second order derivative estimate with respect to the measure variable, are derived in order to ensure the uniqueness and existence.
	\end{abstract}

	

	\section{Introduction}
{\color{black}The primary objective of this paper is to establish the well-posedness of viscosity solutions for second order Hamilton-Jacobi-Bellman (HJB) equations, which arise from control problems in the space of probability measures over $\mathbb{R}^d$,  under a newly proposed notion (see Definition \ref{def. of vis sol} and Remark \ref{techincal_gap_def} for the rationale behind this definition). This is achieved under mild assumptions on the coefficients, as detailed in Section 2.3. The main HJB equation of interest is given by:}
\rcm{
{\small		 \begin{align}
   \label{main_HJB}
			\left\{\begin{aligned}
				&-\partial_t u(t,\mu)
    = \dfrac{1}{2}\textup{tr}\Big[\sigma^0(t)[\sigma^0(t)]^\top\mathcal{H}u(t,\mu)\Big]\\
				&+\int_{\mathbb{R}^d}
				\sup_{a\in A} \Bigg\{f(t,x,\mu,a)+b(t,x,\mu,a)\cdot \partial_\mu u(t,\mu)(x) \\
				&\h{55pt}+ \dfrac{1}{2}\textup{tr}\Big(\sigma(t,x,a)\big[\sigma(t,x,a)\big]^\top\partial_x\partial_\mu u(t,\mu)(x)\Big)\Bigg\}\mu(dx)\,,
				 \h{4pt}\text{ $(t,\mu)\in [0,T)\times\mathcal{P}_2(\mathbb{R}^d)$};\\
				&u(T,\mu)=\int_{\mathbb{R}^d}g(x,\mu)\mu(dx) \,,\h{10pt}\text{$\mu\in \mathcal{P}_2(\mathbb{R}^d)$}.
\end{aligned}\right.
	\end{align}
	}}
	In the above, the operator $\mathcal{H}$ is defined as in \cite{bayraktar_comparison_2023}, that is, 
\begin{align}
\label{H_def_intro}
    \mathcal{H}u(\mu) := \frac{d^2}{dw^2}u\pig((I_d+w)_\sharp \mu\pig)\Big|_{w=0}\in\mathbb{R}^{d\times d},
\end{align}
where $(I_d+w)_\sharp \mu$ is the pushforward measure associated with map $x\mapsto x+w$; and $\frac{d}{dw}$ is the usual gradient operator on $\mathbb{R}^d$.  Indeed, for sufficiently smooth functions $u:\mathcal{P}_2(\mathbb{R}^d)\to\mathbb{R}$, it holds that (see Lemma \ref{regular_H}):
\begin{align*}
         \mathcal{H}u(\mu) = \int_{\mathbb{R}^d} \partial_{x}\partial_{\mu}u(\mu)(x)\mu(dx)+ \int_{\mathbb{R}^d}\int_{\mathbb{R}^d} \partial^2_{\mu}u(\mu)(x,\tilde{x})\mu(dx)\mu(d\tilde{x}),
     \end{align*}
and thus $\mathcal{H}$ may be viewed as the second order differential operator with respect to the measure variable in a ``weaker'' sense. We refer readers to Section 2 for notations used in the above. This HJB equation characterizes the value function for the following mean field type control problem with common noise:
		$$\sup_{\alpha\in\mathcal{A}_t}\mathbb{E}\Bigg[\int_t^T f\Big(s,X_s^{t,\xi,\alpha},\mathbb{P}^{W^0}_{X_s^{t,\xi,\alpha}},\alpha_s\Big)ds + g\Big(X_T^{t,\xi,\alpha},\mathbb{P}_{X_T^{t,\xi,\alpha}}^{W^0}\Big)\Bigg]$$
subject to the dynamics
		$$	X_s^{t,\xi,\alpha} = \xi + \int_{t}^s b\Big(r,X_r^{t,\xi,\alpha},\mathbb{P}_{X_r^{t,\xi,\alpha}}^{W^0},\alpha_r\Big)dr + \int_t^s \sigma(r,X_r^{t,\xi,\alpha},\alpha_r)dW_r + \int_t^s \sigma^0(r)dW^0_r\,,$$
where $W$ and $W^0$ are independent Wiener processes representing idiosyncratic and common noises respectively; and $\mathbb{P}_{X_r^{t,\xi,\alpha}}^{W^0}$ denotes the conditional distribution of $X_r^{t,\xi,\alpha}$ given $W^0$.
\\[4pt]
\noindent The theory of mean field type (McKean-Vlasov type) control problems and related mean field games, initiated independently by Lasry and Lions \cite{lasry_mean_2007} and Caines, Huang, and Malhame \cite{cis/1183728987}, have experienced significant growth and development over the past decade. Mean field type control problems introduce measure dependence in both the state processes and cost functionals, which allows the modeling of large-scale systems with agents interacting symmetrically. Readers are referred to the comprehensive monographs such as Carmona and Delaure \cite{carmona_probabilistic_2018_vol1,carmona_probabilistic_2018_vol2}, Bensoussan, Frehse and Yam \cite{bensoussan_mean_2013}, and the references therein for more exhaustive discussions. 
\\[4pt] 
In standard control theory, two primary approaches are prominent: the Pontryagin maximum principle and the {\color{black}HJB} equation. The Pontryagin maximum principle aims to establish necessary conditions that characterize the optimal control, often relying on solving a system of mean field forward-backward stochastic differential equations. Extensive research has been conducted in this area and interested readers can refer to Andersson and Djehiche \cite{andersson_maximum_2011}, Li \cite{li_stochastic_2012}, Buckdahn, Djehiche, and Li \cite{buckdahn_general_2011}, Carmona and Delaure \cite{Carmona_Delaure_AOP}, Bensoussan, Tai, and Yam \cite{bensoussan_mean_2023}, Bensoussan, Wong, Yam, and Yuan \cite{bensoussan_theory_2023}, Gangbo, Mészáros, Mou, and Zhang \cite{Gangbo_Mou_Zhang_AOP}, Mou and Zhang \cite{mou_wellposedness_2019}, Chassagneux, Crisan, and Delarue \cite{Dan_Delaure_book}, Bayraktar, Cosso, and Pham \cite{erhan_1}, Bayraktar, Alekos, and Prakash \cite{erhan_2}, among others. 
The latter approach links the solvability of the control problems to HJB equations. It is widely accepted that, apart from specific cases like standard linear quadratic scenarios, classical solutions to the HJB equations are generally hard to obtain. Because of this, Crandall, Ishii, and Lions \cite{user_guide}, Lions \cite{vis_lions} introduced the concept of viscosity solution, which allows solutions in a much broader sense. For viscosity solutions in the Euclidean space or Hilbert space, readers are referred to Crandall and Lions \cite{Crandall_Lions_83,crandall_lions_85},  Cannarsa and Soner \cite{soner_vis}, and so on.\\[4pt]
In the context of the Wasserstein space, there are several notable works focusing on viscosity solutions of HJB equations. Pham and Wei \cite{Pham_Wei_Dynamic_Programming} introduced a dynamic programming principle tailored for cases where the control is adapted to the filtration generated by the common noise only. Due to the lack of local compactness in the infinite-dimensional Wasserstein space, they lifted their resulting HJB equation to the Hilbert space of random variables to establish both the uniqueness and existence of the viscosity solution. To bypass this lifting procedure, Wu and Zhang \cite{mean_field_path} introduced a new viscosity solution concept, departing from the traditional Crandall-Lions' definition. Their approach required the maximum/minimum condition to hold on a compact subset of the Wasserstein space instead of just a local neighborhood, enabling the development of a comprehensive theory of viscosity solutions. In an attempt to align with the original Crandall-Lions' definition and address local compactness challenges in the Wasserstein space, Cosso, Gozzi, Kharroubi, Pham, and Rosestolato \cite{cosso_master_2022} utilized the Borwein-Preiss generalization of Ekeland’s variational principle. This principle addresses the local compactness concerns by showing that a small perturbation of the value function with a smooth gauge function can attain its maximum or minimum. {\color{black}Through this, they tried to extend the Crandall-Lions' definition of viscosity solutions to the Wasserstein space.} \\[4pt]
The above, excluding \cite{Pham_Wei_Dynamic_Programming}, pertain to first order HJB equations in the Wasserstein spaces. The presence of common noise induces the second order $L$-derivative in these equations, which has garnered significant attention in recent times. Bayraktar, Ekren, and Zhang \cite{bayraktar_comparison_2023} made use of the observation (see the operator $\mathcal{H}$ in \eqref{H_def_intro}) that the second order $L$-derivative appears in the form of integration and is actually finite-dimensional. Together with the Fourier-Wasserstein metric employed in \cite{soner_viscosity_2022}, they successfully established Ishii's lemma in the Wasserstein space. This accomplishment allowed a direct comparison of second order equations, rather than relying on finite-dimensional approximation methods as outlined in \cite{cosso_master_2022}. Consequently, it broadened the scope to include a wider array of second order equations beyond those induced solely from control problems like HJB equations. However, their result required stringent regularity assumptions imposed on the coefficients, because of the utilization of the Fourier-Wasserstein metric.  On top of the finite-dimensional observation for the second order derivative, Samuel, Joe and Benjamin \cite{daudin2023well} established the well-posedness of viscosity solutions for a class of semi-linear Hamilton-Jacobi equations (which might not arise from control problems) by using some delicate estimates for a sequence of finite-dimensional approximating PDEs and mollification techniques. Nonetheless, they were concerned only about the space of probability measures over the torus. See also Bayraktar, Ekren and Zhang \cite{bayraktar2024convergence} for the convergence rate. More recently, a new notion of viscosity solutions has been proposed by Touzi, Zhang and Zhou \cite{zhou2024viscosity} accompanied by the establishment of existence and comparison principles under the sole assumption of Lipschitz continuity. The main feature of their new notion is its incorporation of an additional singular component in the test function, enabling the treatment of second order derivatives of test functions without invoking Ishii's lemma. Their new notion allows them to handle second order equations with both drift and volatility controls. {\color{black}Nevertheless, this notion is different from the Crandall-Lions' definition in the sense that it allows a nonsmooth component in the test function.} In contrast to the literature referenced, the primary contribution of our article involves considering the HJB equation incorporating the second order measure derivative under mild assumptions. Furthermore, our approach to viscosity solutions is intrinsic \rcm{(in the sense that we do not have to lift the equations to Hilbert spaces)}, and we extend our analysis to allow the state space to encompass the entire space $\mathbb{R}^d$ rather than being confined to a torus.
\\[4pt]
\noindent To establish the uniqueness of the viscosity solution to \eqref{main_HJB}, a pivotal idea is to seek a suitable comparison function with adequate regularity and favorable estimates on its derivatives, for the purpose of the comparison theorem.  Drawing inspiration from \cite{cosso_master_2022}, we choose a gauge-type function to construct the comparison function via the Borwein-Preiss variational principle. However, the gauge function used in \cite{cosso_master_2022} encounters difficulties when dealing with the second order $L$-derivative, particularly in addressing mean field-type control problems involving common noise. To resolve this, we adopt the Gaussian regularized sliced Wasserstein distance proposed in \cite{bayraktar_smooth_2023} to compose our gauge function. This provides a metric over the space of high-dimensional probability distributions using their one-dimensional projections, offering the advantage of explicitly deriving the optimal transport map, which facilitates estimates of the derivatives. We observe that when the state and the law of the first order $L$-derivative of this metric move together in a deterministic direction (which is captured by the operator $\mathcal{H}$, see Lemma \ref{gauge_H}), the gauge function behaves linearly with respect to the perturbed direction, enabling the estimation of the second order $L$-derivative of this gauge function in a ``weaker'' sense. As this gauge function will be used as a test function, we have to define a customized function space which is ${PC}_1^{1,2}([0,T]\times \mathcal{P}_2(\mathbb{R}^d))$ (see Definition \ref{PC_121}) as our set of test functions. Then, we perturb the HJB equation and mollify the coefficients as in \cite{cosso_master_2022}, together with the favorable estimates derived for the gauge function to complete the comparison theorem, as presented in  Theorem \ref{thm compar}. By the specific form of the optimal transport map in $\mathbb{R}$, we provide further regularity results of this metric for the purpose of the relaxed It\^o formula in Theorem \ref{ito_for_less_regularity}, which is crucial in establishing the existence of viscosity solutions. Moreover, this gauge function lacks standard linear or quadratic growth, which requires a more subtle analysis in the proof of the existence of viscosity solution compared to standard test functions. {\color{black}We note that our proposed definition of the viscosity supersolution in Definition \ref{def. of vis sol} is stronger than that of Crandall-Lions. It aims to overcome technical difficulties arising from the supremum taken in equation \eqref{main_HJB}, which leads to scarce choice of test functions in the proof of the comparison for the supersolution. The cost functional with a fixed control is used to construct the test function for supersolutions as in  \cite{cosso_master_2022}, but with the domain of $[0,T] \times \mathcal{P}_2(\mathbb{R}^d\times A)$ rather than $[0,T] \times\mathcal{P}_2(\mathbb{R}^d)$. The strong dependency of the initial random variable and the control makes it hard to use only the Crandall-Lions' definition to draw the conclusion. See Remark \ref{techincal_gap_def} for more details. }
\\    [4pt]
This paper is organized as follows: Section 2 covers the foundational elements: the control problem set-up, some definitions of our tools, It\^o formulas and the standing assumptions. In Section 3, we explore fundamental properties of the value function and introduce the dynamic programming principle. Moving to Section 4, we provide some estimates of the Gaussian regularized sliced Wasserstein distance and state the Borwein-Preiss variational principle. In Section 5, we outline the finite-dimensional approximation scheme of the value function, drawing inspiration from \cite{cosso_master_2022}. Section 6 is dedicated to proving the existence and uniqueness of the viscosity solution, which stands as the value function of our control problem.

	\section{Preliminary}

	\subsection{Probabilistic Setting and Notations}
	
	\label{strong_form_1}
	Fix a probability space $(\Omega, \mathcal{F},\mathbb{P})$ of the form $(\Omega^0\times\Omega^1,\mathcal{F}^0\otimes \mathcal{F}^1,\mathbb{P}^0\otimes \mathbb{P}^1)$. The space $(\Omega^0,\mathcal{F}^0,\mathbb{P}^0)$ supports a $d$-dimensional Brownian motion $W^0$ which we regard as the common noise. For $(\Omega^1,\mathcal{F}^1,\mathbb{P}^1)$, it is of the form $(\tilde{\Omega}^1\times \hat{\Omega}^1,\mathcal{G}\otimes\hat{\mathcal{F}}^1,\tilde{\mathbb{P}}^1\otimes\hat{\mathbb{P}}^1)$. On $(\hat{\Omega}^1,\hat{\mathcal{F}}^1,\hat{\mathbb{P}}^1)$, there lives a $d$-dimensional Brownian motion $W$ which we regard as the idiosyncratic noise, whereas $(\tilde{\Omega}^1,\mathcal{G},\tilde{\mathbb{P}}^1)$ is where the initial random variables live. We assume that the probability space $(\tilde{\Omega}^1,\mathcal{G},\tilde{\mathbb{P}}^1)$ is rich enough to support all probability laws in $\mathbb{R}^d$, i.e., for any probability law $\mu$ in $\mathbb{R}^d$, there exists $X\in \tilde{\Omega}^1$ such that the law of $X$, denoted by $\mathcal{L}(X)$, is $\mu$.
	\\[4pt]
We write $\omega\in\Omega$ as $\omega = (\omega^0,\omega^1)$, and  regard the Brownian motions $W(\omega) = W(\omega^1)$, $W^0(\omega) = W^0(\omega^0)$. We denote by $\mathbb{E}$ the expectation under $\mathbb{P}$ and by $\mathbb{E}^0$ (resp. $\mathbb{E}^1$) the expectation under $\mathbb{P}^0$ (resp. $\mathbb{P}^1$). Also, we set $\mathbb{F} = (\mathcal{F}_s)_{s\geq 0}:=(\sigma(W^0_r)_{0\leq r\leq s}\vee\sigma(W_r)_{0\leq r\leq s}\vee\mathcal{G})_{s\geq 0}$, $\mathbb{F}^t = (\mathcal{F}_s^t)_{s\geq 0}:=(\sigma(W_r^0)_{0\leq r\leq s}\vee\sigma(W_{r\vee t}-W_t)_{0\leq r\leq s}\vee\mathcal{G})_{s\geq 0}$, $\mathbb{F}^{W^0} = (\mathcal{F}_s^{W^0})_{s\geq 0}:=(\sigma(W_r^0))_{0\leq r \leq s}$, and $\mathbb{F}^1 = (\mathcal{F}_s^1)_{s\geq 0}:= (\sigma(W_s)\vee\mathcal{G})_{{0\leq r\leq s}}$. Without loss of generality, we assume they are $\mathbb{P}$-complete.\\[4pt]
	\noindent Let $A$ be a compact subset  of the Euclidean space $\mathbb{R}^d$ equipped with the distance $d_A$. Let $t>0$ and denote by $\mathcal{A}$ (resp. $\mathcal{A}_t$) the set of $\mathbb{F}$-progressively measurable processes (resp. $\mathbb{F}^{t}$-progressively measurable processes) on $\Omega$ valued in $A$. Note that both $\mathcal{A}$ and $\mathcal{A}_t$ are separable metric spaces endowed with the Krylov distance $\Delta(\alpha,\beta) := \mathbb{E}\pig[\int_0^T d_A(\alpha_r,\beta_r)dr\pig]$. Denote by $\mathcal{B}_{\mathcal{A}}$ (resp. $\mathcal{B}_{\mathcal{A}_t}$) the Borel $\sigma$-algebra of $\mathcal{A}$ (resp. $\mathcal{A}_t$). We assume that $ (\Omega^0, \mathcal{F}^0, \mathbb{P}^0)$ is the canonical space, i.e., $\Omega^0 = C(\mathbb{R}_{+},\mathbb{R}^d)$, the set of continuous functions from $\mathbb{R}_{+}$ into $\mathbb{R}^d$. For any $\omega^0,\hat{\omega}^0\in\Omega^0$, $r \in [0,T]$, we set
	\begin{align*}
		\hat{\omega}^0\otimes_r\omega^0(s) := \hat{w}^0(s)\mathbbm{1}_{[0,r)}(s)
  +\big[\hat{w}^0(r)+\omega^0(s)-\omega^0(r)\big]
  \mathbbm{1}_{[r,T]}(s).
	\end{align*}
Let $r\in [0,T]$, for $\mathbb{P}^0$-a.s. $\tilde{\omega}^0$, $\alpha\in\mathcal{A}_t$ for some $t\in[0,T]$, we define:
\begin{align*}
    \alpha^{r,\tilde{\omega}^0}(\omega^0,\omega^1) :=\alpha(\tilde{\omega}^0\otimes_r\omega^0,\omega^1).
\end{align*}
For any $x\in \mathbb{R}^d$, we use $\left|x\right|$ to denote the Euclidean norm of $x$ in $\mathbb{R}^d$, $x_i$ (sometimes we also use $(x)_i$ or $[x]_i$) to denote the $i$-th component of $x$, $\langle\cdot, \cdot\rangle$ (or simply the $x\cdot y$ for $x,y \in \mathbb{R}^d$) to denote the standard scalar product on $\mathbb{R}^d$. Let $n \in \mathbb{N}$ and $x^1,x^2,\ldots,x^n \in \mathbb{R}^d$, we use $\overline{x}\in \mathbb{R}^{dn}$ to denote $\overline{x}:=(x^1,x^2,\ldots,x^n)$.  For any matrix $M \in \mathbb{R}^{d\times d}$, we use $\textup{tr}M:=\sum_{i=1}^dM_{ii}$ to denote its trace, $M^\top$ to denote its transpose, and $|M|:=\left[\operatorname{tr}\left(M M^{\top}\right)\right]^{1 / 2}$ $=\left(\sum_{i, j}^d\left|M_{ij}\right|^2\right)^{1 / 2}$ to denote the Frobenius norm of $M$. If $M^{0}\in \mathbb{R}^{d\times d}$ is another matrix, we use $M^{0;\top}$ to denote the transpose of $M^{0}$. We denote the identity matrix over $\mathbb{R}^d$ by $I_d \in \mathbb{R}^{d \times d}$. If $x\in \mathbb{R}$ is a scalar variable, the notation $\p_x h \in \mathbb{R}$ means the usual partial derivative of the scalar function $h$ with respect to $x$; if $x \in \mathbb{R}^d$ is a vector variable, then $\p_x h \in \mathbb{R}^d$ means the gradient vector of $h$.
{\color{black}\begin{remark}
    Here we explain a little bit about the definition of $\mathbb{F}^t$. It makes no technical difference whether we consider the control sets which is progressively measurable with respect to $\mathbb{F}^t = (\mathcal{F}_s^t)_{s\geq 0}:=(\sigma(W_r^0)_{0\leq r\leq s}\vee\sigma(W_{r\vee t}-W_t)_{0\leq r\leq s}\vee\mathcal{G})_{s\geq 0}$ or $\mathring{\mathbb{F}}^t = (\mathring{\mathcal{F}}_s^t)_{s\geq 0}:=(\sigma(W^0_{r\vee t}-W^0_t)_{0\leq r\leq s}\vee\sigma(W_{r\vee t}-W_t)_{0\leq r\leq s}\vee\mathcal{G})_{s\geq 0}$. All the arguments in this article remain valid with slight or no modification. In fact, denoting $\mathring{\mathcal{A}}_t$ to be the set of all $\mathring{\mathbb{F}}^t$-progressively measurable processes taking value in $A$, and the value function $\mathring{V}(t,\xi):= \sup_{\alpha\in\mathring{\mathcal{A}}_t}J(t,\xi,\alpha)$ (see \eqref{def_J} for the definition), we actually have
    \begin{align*}
        \mathring{V}(t,\xi) = V(t,\xi), 
    \end{align*}
   where $V$ is defined in \eqref{def. value function with xi}. This can be seen from the following: First of all, as $\mathring{\mathcal{A}}_t\subset \mathcal{A}_t$, it follows that $\mathring{V}(t,\xi) \leq V(t,\xi)$. To show the other direction of the inequality, let $\alpha\in\mathcal{A}_t$, then
    \begin{align*}
        J(t,\xi,\alpha) =&\, \mathbb{E}\Bigg[\int_t^T f(s,X_s^{t,\xi,\alpha},\mathbb{P}^{W^0}_{X_s^{t,\xi,\alpha}},\alpha_s)ds + g\Big(X_T^{t,\xi,\alpha},\mathbb{P}_{X_T^{t,\xi,\alpha}}^{W^0}\Big)\Bigg]\\
        =&\,\mathbb{E}^0\mathbb{E}\Bigg[\int_t^T f(s,X_s^{t,\xi,\alpha},\mathbb{P}^{W^0}_{X_s^{t,\xi,\alpha}},\alpha_s)ds + g\Big(X_T^{t,\xi,\alpha},\mathbb{P}_{X_T^{t,\xi,\alpha}}^{W^0}\Big)\Big|\mathcal{F}_t^0\Bigg]\\
        =&\,\mathbb{E}^0[J(t,\xi,\alpha^{t,\omega^0})]\\
        \leq &\,\mathbb{E}^0[\mathring{V}(t,\xi)] \\
        =&\, \mathring{V}(t,\xi).
    \end{align*}
    Taking supremum with respect to $\alpha\in\mathcal{A}_t$ we conclude the desired equality.\\
    \hfill\\
    The requirement that the controls are adapted to $\mathcal{G}$ for every $s\geq 0$ is a bit technical. The condition that the controls have to be adapted to $\mathcal{G}$ is used in the existence theorem. Because of the nonlocal nature of the HJB equations in the Wasserstein space, assuming that $\varphi$ is our test function, in the proof of existence of viscosity solution we can only conclude that
\begin{align}
\label{worse_form}
&\partial_t \varphi(t_0, \mu_0) + \sup_{\alpha' \in \mathcal{M}_t} \mathbb{E} \Big\{ f(t_0, \xi, \mu_0, \alpha') + \partial_\mu \varphi(t_0, \mu_0)(\xi) \cdot b(t_0, \xi, \mu_0, \alpha') \nonumber\\
&+ \frac{1}{2} \operatorname{tr} \left[ \partial_x \partial_\mu \varphi(t_0, \mu_0)(\xi) \sigma(t_0, \xi, \alpha') \sigma(t_0, \xi, \alpha')^\top \right] \Big\} + \frac{1}{2} \operatorname{tr} \left[ \mathcal{H} \varphi(t_0, \mu_0) \sigma^0(t_0) \sigma^0(t_0)^\top \right] \leq \text{(or } \geq) \ 0,
\end{align}
where $\mathcal{L}(\xi) = \mu_0$, and $\mathcal{M}_t$ is the set of $\mathcal{F}_t^t$-measurable random variables. The condition of requiring the control set to be adapted to $\mathcal{G}$ makes the controls rich enough such that we can choose controls to make the above equivalent to 
\begin{align}
\label{better_form}
&\partial_t \varphi(t_0, \mu_0) + \int_{\mathbb{R}^d} \sup_{a \in A} \Bigg\{ f(t_0, x, \mu_0, a) + b(t_0, x, \mu_0, a) \cdot \partial_\mu \varphi(t_0, \mu_0)(x) \nonumber\\
&+ \frac{1}{2} \operatorname{tr} \Big[ \partial_x \partial_\mu \varphi(t_0, \mu_0)(x) \sigma(t_0, x, a) [\sigma(t_0, x, a)]^\top \Big] \Bigg\} \mu_0(dx) 
+ \frac{1}{2} \operatorname{tr} \Big\{ \mathcal{H} \varphi(t_0, \mu_0) \sigma^0(t_0) [\sigma^0(t_0)]^\top \Big\} \leq \text{(or } \geq) \ 0,
\end{align}
see Theorem \ref{thm. existence of vis sol}. 
\end{remark}}
	
	\subsection{Sense of Differentiability in the Wasserstein Space}
We introduce over $\mathbb{R}^d$ the space of probability measures $\mathcal{P}(\mathbb{R}^d)$ and its subset $\mathcal{P}_p(\mathbb{R}^d)$ consisting of those with finite $p$-th moment for $p\geq 1$. The space $\mathcal{P}_p(\mathbb{R}^d)$ is equipped with the $p$-Wasserstein distance
		$$
		\mathcal{W}_p(\mu, \nu)=\inf _{\pi \in \Pi(\mu, \nu)}\Bigg(\int_{\mathbb{R}^d \times \mathbb{R}^d}|x-y|^p \pi(dx, dy)\Bigg)^{\frac{1}{p}}, \quad  \text{for $\mu, \nu \in \mathcal{P}_p(\mathbb{R}^d)$},
		$$
    		where $\Pi(\mu, \nu)$ is the set of probability measures on $\mathbb{R}^d\times\mathbb{R}^d$ such that for any $\pi \in \Pi(\mu,v)$, $\pi(\mathbb{R}^d\times\cdot) = \mu$ and $\pi(\cdot\times\mathbb{R}^d) = \nu$. We call $(\mathcal{P}_p(\mathbb{R}^d),\mathcal{W}_p)$ the $p$-Wasserstein space over $\mathbb{R}^d$, and it is a Polish space. Finally, we denote by $\operatorname{Supp}(\mu)$ the support of $\mu \in \mathcal{P}(\mathbb{R}^d)$. For a map $f:[0,T]\times\mathcal{P}_2(\mathbb{R}^d)\to\mathbb{R}$, we adopt the notion of $L$-derivative (see \cite{lions_annals} for instance) which is recalled as follows:
	\begin{definition}
		The function $f:[0,T]\times\mathcal{P}_2(\mathbb{R}^d)\to\mathbb{R}$ is said to be first order $L$-differentiable if its lifting $F:[0,T]\times L^2(\Omega,\mathcal{F},\mathbb{P};\mathbb{R}^d)\to\mathbb{R}$; $F(t,\xi) := f(t,\mathcal{L}(\xi)) $ admits a continuous Fr\'echet derivative $D_\xi F :[0,T] \times L^2(\Omega,\mathcal{F},\mathbb{P};\mathbb{R}^d) \to L^2(\Omega,\mathcal{F},\mathbb{P};\mathbb{R}^d)$. 
	\end{definition}
	
	\begin{remark}
		By \cite[Proposition 5.25]{carmona_probabilistic_2018_vol1}, if $f$ is first order $L$-differentiable, then it may be shown that there is a measurable function, denoted by $\partial_\mu f(t,\mu)(\cdot) :\mathbb{R}^d \to \mathbb{R}^d$, such that $D_\xi F(t,\xi)=\partial_\mu f(t,\mu)(\xi)$ for any $(t,\mu) \in [0,T] \times \mathcal{P}_2(\mathbb{R}^d)$ and $\xi \in L^2(\Omega,\mathcal{F},\mathbb{P};\mathbb{R}^d)$ with $\mathcal{L}(\xi)=\mu$. We say that $\partial_\mu f:[0,T]\times \mathcal{P}_2(\mathbb{R}^d)\times\mathbb{R}^d \to\mathbb{R}^d$ is the first order $L$-derivative of $f$.
		\label{169}
	\end{remark}

	\begin{definition}
		The function $f:[0,T]\times\mathcal{P}_2(\mathbb{R}^d)\to\mathbb{R}$ is said to be second order $L$-differentiable if $f$ is first order $L$-differentiable and for any $x \in \mathbb{R}^d$, the function $\mu\mapsto\partial_\mu f(t,\mu)(x)$ is $L$-differentiable, i.e., the lifting $F':L^2(\Omega,\mathcal{F},\mathbb{P};\mathbb{R}^d)\to\mathbb{R}^d$ of $\mu\mapsto\partial_\mu f (t,\mu)(x)$ admits a continuous Fr\'echet derivative $D_\xi F' :[0,T] \times L^2(\Omega,\mathcal{F},\mathbb{P};\mathbb{R}^d) \to L^2(\Omega,\mathcal{F},\mathbb{P};\mathbb{R}^{d\times d})$. 
	\end{definition}
	\begin{remark}
		Similar to Remark \ref{169}, if $f$ is second order $L$-differentiable, then there is a measurable function, denoted by $\partial_\mu^2 f(t,\mu)(x,\cdot) :\mathbb{R}^d \to \mathbb{R}^{d\times d}$, such that $D_\xi F'(t,\xi')=\partial_\mu^2 f(t,\mu)(x,\xi')$ for any $(t,\mu,x) \in [0,T] \times \mathcal{P}_2(\mathbb{R}^d)\times \mathbb{R}^d$ and $\xi' \in L^2(\Omega,\mathcal{F},\mathbb{P};\mathbb{R}^d)$ with $\mathcal{L}(\xi')=\mu$. We say that $\partial^2_\mu f:[0,T]\times \mathcal{P}_2(\mathbb{R}^d)\times\mathbb{R}^d\times\mathbb{R}^d \to\mathbb{R}^{d\times d}$ is the second order $L$-derivative of $f$.
	\end{remark}
        \noindent As in \cite{bayraktar_comparison_2023}, we define the set of fully second order $L$-differentiable functions and the set of partially second order $L$-differentiable functions.
	\begin{definition}
		The set $C_1^{1,2}([0,T]\times \mathcal{P}_2(\mathbb{R}^d))$ consists of all continuous functions $f:[0,T]\times \mathcal{P}_2(\mathbb{R}^d) \to \mathbb{R}$ satisfying the following:
		\begin{enumerate}[(1).]
			
			\item the derivatives $\partial_t f(t,\mu)$, $\partial_\mu f(t,\mu)(x)$, $\partial_x \partial_\mu f (t,\mu)(x)$, $\partial_\mu^2 f(t,\mu)(x,x')$ exist and are jointly continuous in the respective arguments;
			\item there is a constant $C_f\geq 0$ such that for any $(t,\mu,x,x') \in [0,T] \times \mathcal{P}_2(\mathbb{R}^d)\times \mathbb{R}^d\times \mathbb{R}^d$, we have
\begin{align*}			
     |\partial_\mu f(t,\mu)(x)|
			&\leq C_f\pig(1+|x|\pig);\\
   |\p_t f(t,\mu)|+|\partial_x\partial_\mu f(t,\mu)(x)|
			+\pig|\partial_\mu^2 f(t,\mu)(x,x')\pig|&\leq C_f.
\end{align*}
		\end{enumerate}
	\end{definition}
\noindent We define the operator $\mathcal{H}$:
 \begin{align*}
    \mathcal{H}u(\mu) := \frac{d^2}{dw^2}u((I_d+w)_\sharp \mu)\Big|_{w=0}\in\mathbb{R}^{d\times d},
 \end{align*}
where $(I_d+w)_\sharp \mu$ is the pushforward measure associated with map $I_d+w:\mathbb{R}^d\to\mathbb{R}^d$ defined by $x\mapsto x+w$ for $w \in \mathbb{R}^d$; and we are differentiating with respect to finite-dimensional $w\in\mathbb{R}^d$.
\begin{definition}
\label{PC_121}
		The set ${PC}_1^{1,2}([0,T]\times \mathcal{P}_2(\mathbb{R}^d))$ consists of all continuous functions $f:[0,T]\times \mathcal{P}_2(\mathbb{R}^d) \to \mathbb{R}$ satisfying the following:
		\begin{enumerate}[(1).]
			
			\item the derivatives $\partial_t f(t,\mu)$, $\partial_\mu f(t,\mu)(x)$, $\partial_x \partial_\mu f (t,\mu)(x)$, $\mathcal{H}f(\mu)$ exist and are jointly continuous in the respective arguments;
			\item there exists $C_{f} \geq 0$ such that for any $\rcm{(t,\mu)} \in [0,T]\times\mathcal{P}_2(\mathbb{R}^d)$, it holds that
   \begin{align*}
       \int_{\mathbb{R}^d}\left|\p_\mu f(t, \mu)(x)\right|^2 \mu(d x) &\leq C_{f}\left(1+\int_{\mathbb{R}^d}|x|^2 \mu(d x)\right);\\
              |\p_t f(t,\mu)|+|\mathcal{H}f(t,\mu)| + \int_{\mathbb{R}^d}\left|\p_x\p_{\mu} f(t, \mu)(x)\right|^2 \mu(d x)&\leq C_{f}.
   \end{align*}
		\end{enumerate}
	\end{definition}
\noindent We have the following lemma relating $PC_1^{1,2}([0,T]\times\mathcal{P}_2(\mathbb{R}^d))$ and $C_1^{1,2}([0,T]\times\mathcal{P}_2(\mathbb{R}^d))$.
\begin{lemma}
 \label{regular_H}
     If $u\in C_1^{1,2}([0,T]\times \mathcal{P}_2(\mathbb{R}^d))$, then we have
     \begin{align*}
         \mathcal{H}u(t,\mu) = \int_{\mathbb{R}^d} \partial_{x}\partial_{\mu}u(t,\mu)(x)\mu(dx)+ \int_{\mathbb{R}^d}\int_{\mathbb{R}^d} \partial^2_{\mu}u(t,\mu)(x,\tilde{x})\mu(dx)\mu(d\tilde{x}),
     \end{align*}
     and hence $u\in PC_1^{1,2}([0,T]\times\mathcal{P}_2(\mathbb{R}^d))$.
 \end{lemma}
 \begin{proof}
     Let $\mu\in\mathcal{P}_2(\mathbb{R}^d)$ and $X\in L^2(\Omega,\mathcal{F},\mathbb{P};\mathbb{R}^d)$ such that $\mathcal{L}(X) = \mu$. We denote the lifting of $u\in C_1^{1,2}([0,T]\times\mathcal{P}_2(\mathbb{R}^d))$ to Hilbert space by $u^*$, we have 
     \begin{align*}
         \frac{d}{dw}u^*(t,X+w) = \mathbb{E}\pig[\partial_\mu u(t,\mathcal{L}(X+w))(X+w)\pig].
     \end{align*}
        Besides, we consider
     \begin{align*}
         &\frac{d}{dw}\partial_\mu u(t,\mathcal{L}(X+w))(X+w)\\
         &=\partial_x\partial_\mu u(t,\mathcal{L}(X+w))(X+w)
         +\bar{\mathbb{E}}\pig[\partial_\mu^2 u(t,\mathcal{L}(X+w))(X+w,\bar{X}+w)\pig],
     \end{align*}
     where $\bar{X}$ is an independent copy of $X$ and $\bar{\mathbb{E}}$ is the expectation with respect to $\bar{X}$. We assume that $|w|\leq 1$ and use the fact that $u\in C_1^{1,2}([0,T]\times \mathcal{P}_2(\mathbb{R}^d))$ to yield that
     \begin{align*}
         \bigg|\frac{d}{dw}\partial_\mu u(t,\mathcal{L}(X+w))(X+w)\bigg| 
         \leq C_u,
     \end{align*}
     for some constant $C_u>0$ independent of $w$ and $X$. By the dominated convergence theorem we have
     \begin{align*}
         \frac{d}{dw}\frac{d}{dw}u^*(t,X+w)\Big|_{w=0} =\,& \frac{d}{dw}\mathbb{E}\pig[\partial_\mu u(t,\mathcal{L}(X+w))(X+w)\pig]\Big|_{w=0} \\
         =\,& \mathbb{E}\Big\{\partial_x\partial_\mu u(t,\mathcal{L}(X+w))(X+w)+\bar{\mathbb{E}}\pig[\partial_\mu^2 u(t,\mathcal{L}(X+w))(X+w,\bar{X}+w)\pig]\Big\}\bigg|_{w=0}\\
         =\,&\int_{\mathbb{R}^d} \partial_{x}\partial_{\mu}u(t,\mu)(x)\mu(dx)+ \int_{\mathbb{R}^d}\int_{\mathbb{R}^d} \partial^2_{\mu}u(t,\mu)(x,\tilde{x})\mu(dx)\mu(d\tilde{x}).
     \end{align*}
 \end{proof}
 \noindent We have the following extension of It\^o formula for functions in $PC_1^{1,2}([0,T]\times\mathcal{P}_2(\mathbb{R}^d))$. This is inspired by \cite{bayraktar_comparison_2023}.
 \begin{theorem}
	\label{ito_for_less_regularity}
		Let $b_t \in \mathbb{R}^d$ and $\sigma_t \in \mathbb{R}^{d \times d}$ be two $\mathbb{F}$-adapted processes such that there exists $L>0$ satisfying $|b_t|\vee|\sigma_t| \leq L$ $\mathbb{P}$-a.s. for all $t\in [0,T]$. Let $\sigma_t^0: [0,T]\to \mathbb{R}^{d\times d}$ be deterministic and $\xi \in L^2(\Omega,\mathcal{F},\mathbb{P})$. Consider the following $\mathbb{R}^d$-valued It\^o process:
		\begin{align*}
			X_t = \xi+ \int^t_0 b_s ds + \int^t_0\sigma_s dW_s+\int^t_0\sigma_s^0dW_s^0,\quad \text{for }t\in[0,T].
		\end{align*}
		We write $\mu_t^{W^0}$ for $\mathbb{P}_{X_t}^{W^0}$, then it holds $\mathbb{P}^0$-a.s. that for $\varphi\in PC_{1}^{1,2}([0,T]\times\mathcal{P}_2(\mathbb{R}^d))$,
        \begin{align*}
			\varphi(t,\mu_t^{W^0}) =\,& \varphi(0,\mu_0^{W^0})  + \int_0^t \partial_t \varphi(s,\mu_s^{W^0})ds + \int_0^t\mathbb{E}^1\big[\partial_\mu \varphi(s,\mu_s^{W^0})(X_s)\cdot b_s\big]ds\\
   &+\int_0^t \mathbb{E}^1\big[\sigma_s^{0;\top} \partial_\mu \varphi(s,\mu_s^{W^0})(X_s)\big]\cdot dW_s^0
			+\frac{1}{2}\int_0^t\mathbb{E}^1\Big\{\operatorname{tr}\pig[\partial_x\partial_\mu \varphi(s,\mu_s^{W^0})(X_s)\sigma_s\sigma_s^\top\pig]\Big\}ds\\
   &+\frac{1}{2}\int_0^t \operatorname{tr}\Big[\mathcal{H} \varphi(s,\mu_s^{W^0})\sigma_s^0\sigma_s^{0;\top}\Big]ds.
		\end{align*}
	\end{theorem}
	\begin{proof}
		For $\mathbb{P}^0$-a.s. $\omega^0\in\Omega^0$, we consider
		\begin{align*}
			X_t(\omega^0,\cdot) = \xi(\omega^0,\cdot) + \left(\int_0^t b_r dr \right)(\omega^0,\cdot)
   + \left(\int_0^t \sigma_r dW_r\right)(\omega^0,\cdot)
   +\left(\int_0^t \sigma_r^0dW_r^0\right)(\omega^0). 
		\end{align*}
		We define $Y_t^{\omega^0}(\cdot):= X_t(\omega^0,\cdot) - \pig(\int_0^t\sigma_r^0dW_r^0\pig)(\omega^0)$. Applying the standard mean field It\^o formula (see \cite[Proposition 5.102]{carmona_probabilistic_2018_vol1}) to $\varphi(t,\mathcal{L}(Y^{\omega^0}_t+y))$ for any  $y\in\mathbb{R}^d$ and $\varphi\in PC_{1}^{1,2}([0,T]\times\mathcal{P}_2(\mathbb{R}^d))$, we conclude that for $\mathbb{P}^0$-a.s.  $\omega^0\in\Omega^0$,
		\begin{align*}
			\varphi(t,\mathcal{L}(Y_t^{\omega^0}+y)) =\,& \varphi(0,\mathcal{L}(Y_0^{\omega^0}+y))+\int_0^t\partial_t \varphi(s,\mathcal{L}(Y_s^{\omega^0}+y))ds \\
   &+\int_{0}^{t} \mathbb{E}^1\Big[\partial_\mu \varphi(s,\mathcal{L}(Y_s^{\omega^0}+y))(Y_s^{\omega^0}+y)\cdot b_s\Big]ds \\
			&+ \frac{1}{2}\int_0^t \mathbb{E}^1\Big\{\operatorname{tr}\pig[\partial_x\partial_\mu \varphi (s,\mathcal{L}(Y_s^{\omega^0}+y))(Y_s^{\omega^0}+y)\sigma_s\sigma_s^{\top}\pig]\Big\}ds.
		\end{align*}
		We define $\Psi(s,y):= \varphi(s,\mathcal{L}(Y_s^{\omega^0}+y))$. By the finite-dimensional It\^o-Kunita-Wentzell formula \cite[Theorem 2.3]{dos2022ito}, we have
  \begin{align*}
      	&\h{-10pt}\Psi\left(t,\int_0^t\sigma^0_sdW_s^0\right)\\ 
       =\,& \Psi(0,0) +\int_0^t\partial_t \varphi\Big(s,\mathcal{L}\Big(Y_s^{\omega^0}+\int_0^s\sigma^0_rdW_r^0\Big)\Big)ds\\
       &+\int_{0}^{t} \mathbb{E}^1\Big[\partial_\mu \varphi\Big(s,\mathcal{L}\Big(Y_s^{\omega^0}+\int_0^s\sigma^0_rdW_r^0\Big)\Big)\Big(Y_s^{\omega^0}+\int_0^s\sigma^0_rdW_r^0\Big)\cdot b_s\Big]ds \\
			&+ \frac{1}{2}\int_0^t \mathbb{E}^1\Big[\operatorname{tr}\Big\{\partial_x\partial_\mu \varphi\Big(s,\mathcal{L}\Big(Y_s^{\omega^0}+\int_0^s\sigma^0_rdW_r^0\Big)\Big)\Big(Y_s^{\omega^0}+\int_0^s\sigma^0_rdW_r^0\Big)\sigma_s\sigma_s^{\top}\Big\}\Big]ds\\
   &+\int_0^t \mathbb{E}^1\Big[\sigma_s^{0;\top} \partial_\mu \varphi\Big(s,\mathcal{L}\Big(Y_s^{\omega^0}+\int_0^s\sigma^0_rdW_r^0\Big)\Big)\Big(Y_s^{\omega^0}+\int_0^s\sigma^0_rdW_r^0\Big)\Big]\cdot dW_s^0\\
   &+\frac{1}{2}\int_0^t \Big[\operatorname{tr}\Big\{\p_y^2\varphi\Big(s,\mathcal{L}\Big(Y_s^{\omega^0}+\int_0^t\sigma^0_sdW_s^0\Big)\Big)\sigma_s^0\sigma_s^{0;\top}\Big\}\Big]ds\\
   =\,&\varphi(0,\mu_0^{W^0}) + \int_0^t \partial_t \varphi(s,\mu_s^{W^0})ds + \int_0^t\mathbb{E}^1\big[\partial_\mu \varphi(s,\mu_s^{W^0})(X_s)\cdot b_s\big]ds
   \\
			&
			+\frac{1}{2}\int_0^t\mathbb{E}^1\big[\operatorname{tr}\{\partial_x\partial_\mu \varphi(s,\mu_s^{W^0})(X_s)\sigma_s\sigma_s^\top\}\big]ds
			+\int_0^t \mathbb{E}^1\big[\sigma_s^{0;\top} \partial_\mu \varphi(s,\mu_s^{W^0})(X_s)\big]\cdot dW_s^0
			\\
			&+\frac{1}{2}\int_0^t \Big[\operatorname{tr}\Big\{\mathcal{H} \varphi(s,\mu_s^{W^0})\sigma_s^0\sigma_s^{0;\top}\Big\}\Big]ds.
  \end{align*}
  \end{proof}
	\subsection{Assumptions and the Control Problem}
	\label{strong_form_2}
	
	Define the functions: 
	\begin{align*}
		b&:[0,T]\times\mathbb{R}^d\times\mathcal{P}_2(\mathbb{R}^d)\times A\to \mathbb{R}^d,&
		\sigma&:[0,T]\times\mathbb{R}^d\times A\to\mathbb{R}^{d\times d},&
		\sigma^0&:[0,T] \to\mathbb{R}^{d\times d},\\
		f&:[0,T]\times\mathbb{R}^d\times\mathcal{P}_2(\mathbb{R}^d)\times A\to\mathbb{R},&
		g&:\mathbb{R}^d\times\mathcal{P}_2(\mathbb{R}^d)\to\mathbb{R}.
	\end{align*}
	 Throughout this work, we use the following assumptions:
	\begin{assumption}
		\textbf{\textup{(A).}} The functions $b$, $\sigma$, $\sigma^0$, $f$ and $g$ satisfy the following:
		\begin{enumerate}[(1).]
			\item the functions $b$, $\sigma$, $\sigma^0$, $f$ and $g$ are continuous;
			\item there exist constants $K\geq 0$ and $\beta \in (0,1]$ such that for any $a\in A$, $(t,x,\mu)$, $(t',x',\mu')\in[0,T]\times\mathbb{R}^d\times\mathcal{P}_2(\mathbb{R}^d)$, it holds that
			\begin{align*}
				&\big|b(t,x,\mu,a)-b(t',x',\mu',a)\big|+\big|\sigma(t,x,a) - \sigma(t',x',a)\big|+\big|\sigma^0(t)-\sigma^0(t')\big|\\
				&+\big|f(t,x,\mu,a)-f(t',x',\mu',a)\big|+\big|g(x,\mu) - g(x',\mu')\big|
				\leq K\big[|x-x'|+|t-t'|^\beta+\mathcal{W}_2(\mu,\mu')\big],\\[5mm]
				&\big|b(t,x,\mu,a)\big|+\big|\sigma(t,x,a)\big|+\big|\sigma^0(t)\big|
				+\big|f(t,x,\mu,a)\big|+\big|g(x,\mu)\big|\leq K.
			\end{align*}
		\end{enumerate}
	\end{assumption}	
	\begin{assumption}
		\textbf{\textup{(B).}}   For any $a \in A$, the function $\sigma(\cdot, \cdot, a)$ belongs to $C^{1,2}\left([0, T] \times \mathbb{R}^d\right)$. Moreover, there exists a constant $K \geq 0$ such that
		\begin{align*}
			\left|\partial_t \sigma(t, x, a)\right|+\left|\partial_{x_i} \sigma(t, x, a)\right|+\big|\partial_{x_i x_j}^2 \sigma(t, x, a)\big|+\left|\partial_t \sigma^0(t)\right| \leq K,
		\end{align*}
		for all $(t, x, a) \in[0, T] \times \mathbb{R}^d \times A$ and any $i, j=1,2, \ldots, d$. 
	\end{assumption} 
 {\color{black}	\begin{assumption}
		\textbf{\textup{(A*).}} 
	Suppose that Assumption (A) holds and we assume further that there exist constants $K\geq 0$ and $\beta \in (0,1]$ such that for any $a\in A$, $(t,x,\mu)$, $(t',x',\mu')\in[0,T]\times\mathbb{R}^d\times\mathcal{P}_2(\mathbb{R}^d)$, it holds that
			\begin{align*}
				&\big|b(t,x,\mu,a)-b(t',x',\mu',a)\big|
    +\big|f(t,x,\mu,a)-f(t',x',\mu',a)\big|+\big|g(x,\mu) - g(x',\mu')\big|\\
				&\leq K\big[|x-x'|+|t-t'|^\beta+\mathcal{W}_1(\mu,\mu')\big].
			\end{align*}
	\end{assumption}	}
 {\color{black} \begin{remark}
    Assumption (A*) is a strengthening assumption of Assumption (A). We write  Assumption (A) and  Assumption (A*) separately to emphasise that some of the lemmas and theorems still hold without assuming the stronger condition in Assumption (A*).
\end{remark}
\begin{remark}
    Here, the coefficients $b$, $\sigma$, $f$ and $g$ are assumed to be $\mathcal{W}_1$-Lipschitz continuous in Assumption (A*), which is different from the $\mathcal{W}_2$-Lipschitz continuity assumption in \cite[Assumption (A)]{cosso_master_2022}. The main reason of adopting $\mathcal{W}_1$-Lipschitz continuity instead of $\mathcal{W}_2$ is to provide a better estimate when we establish the smooth approximation of the coefficients in Lemma \ref{lem estimate of b^i_n,m...}. We note that the inequalities \cite[(A.6) and (A.7)]{cosso_master_2022} are invalid if only $\mathcal{W}_2$-Lipschitz continuity is assumed as in \cite[Assumption (A)]{cosso_master_2022}. This leads to flaws in the proofs of \cite[Theorem 5.1, Theorem A.7]{cosso_master_2022}. See more discussions in Remark \ref{pham gap remark}.
\end{remark}}
	\noindent For every $t\in[0,T]$, $\xi\in L^2(\Omega^1,\mathcal{F}_t^1,\mathbb{P}^1;\mathbb{R}^d)$ and $\alpha \in \mathcal{A}_t$, we consider the solution $X^{t,\xi,\alpha}$ of the following state dynamics:
	\begin{align}
		\label{dynamics}
		X_s &= \xi + \int_{t}^s b(r,X_{r},\mathbb{P}_{X_{r}}^{W^0},\alpha_r)dr + \int_t^s \sigma(r,X_{r},\alpha_r)dW_r + \int_t^s \sigma^0(r)dW^0_r,\h{5pt} \text{ for $s \in [t,T]$,}
	\end{align}
	where $\mathbb{P}_{X_{r}}^{W^0}$ denotes the conditional distribution of $X_{r}$ given $W^0$. We are subject to the cost functional:
	\begin{align}
 \label{def_J}
		J(t,\xi,\alpha) := \mathbb{E}\Bigg[\int_t^T f\Big(s,X_s^{t,\xi,\alpha},\mathbb{P}^{W^0}_{X_s^{t,\xi,\alpha}},\alpha_s\Big)ds + g\Big(X_T^{t,\xi,\alpha},\mathbb{P}_{X_T^{t,\xi,\alpha}}^{W^0}\Big)\Bigg].
	\end{align}
	Finally, we define the value function $V$ to be
	\begin{align}
		V(t,\xi):= \sup_{\alpha\in \mathcal{A}_t}J(t,\xi,\alpha),\quad\text{for any }(t,\xi)\in [0,T]\times L^2(\Omega^1,\mathcal{F}_t^1,\mathbb{P}^1;\mathbb{R}^d).
		\label{def. value function with xi}
	\end{align}
	\section{Basic Properties}
	\noindent This section collects some basic properties. First we give some standard results concerning the regularity of the solution of the SDE in \eqref{dynamics} and the value function $V$. Proofs are standard and therefore omitted, and readers are referred to \cite{yong1999stochastic}.
		\begin{prop}
  \label{prop. property of X}
		Suppose that Assumption (A) holds. For every $t\in [0,T]$, $\xi\in L^2(\Omega^1,\mathcal{F}^1_t,\mathbb{P}^1;\mathbb{R}^d)$, $\alpha\in\mathcal{A}_t$, there exists a unique (up to $\mathbb{P}$-indistinguishability) continuous $\mathbb{F}$-progressively measurable solution $X^{t,\xi,\alpha} = (X^{t,\xi,\alpha}_s)_{s\in [t,T]}$ of equation (\ref{dynamics}). Moreover, for any $p\geq 2$, there is a constant $C$ depending only on $p$, $K$, $T$, $d$ such that
			\begin{align}
			&\mathbb{E}\Big[\sup_{s\in[t,T]}|X_s^{t,\xi,\alpha}|^p\Big]^{1/p}\leq C \pig(1+\mathbb{E}|\xi|^p\pigr)^{1/p};\\
			&\mathbb{E}\Big[\sup_{s\in [t,T]}|X_{s}^{t,\xi,\alpha}-X_{s}^{t,\xi',\alpha}|^p\Big]^{1/p}\leq C \mathbb{E}\pig(|\xi-\xi'|^p\pigr)^{1/p};\\
		    &\mathbb{E}\Big[\sup_{s\in[t,t+h]}|X^{t,\xi,\alpha}_s-\xi|^2\Big]\leq C h,
		\end{align}
for any $t\in [0,T]$, $h \in [0,T-t]$, $\xi,\xi'\in L^2(\Omega^1,\mathcal{F}^1_t,\mathbb{P}^1;\mathbb{R}^d)$ and $\alpha\in\mathcal{A}_t$.
	\end{prop}
	
	\begin{prop}
		\label{property_V}
		Suppose that Assumption $(A)$ holds. The function $V$ satisfies the following properties:
		\begin{enumerate}[(1).]
			\item $V$ is bounded and jointly continuous;
			\item there exists a constant $C\geq 0$ such that for any $t,t'\in[0,T]$, $\xi$, $\xi'\in L^2(\Omega,\mathcal{F}_t,\mathbb{P};\mathbb{R}^d)$,
			\begin{align*}
				|V(t,\xi)-V(t',\xi')| \leq C\Big[\mathbb{E}\pig(|\xi-\xi'|^2\pigr)^{1/2}
				+|t-t'|^{1/2}\Big].
			\end{align*}
			The constant $C$ depends on $d$, $K$, $T$ and independent of $t$, $t'$,$\xi$, $\xi'$, $\alpha$.
		\end{enumerate}
	\end{prop}
	\begin{prop}[Law Invariance]
		\label{Law Invariance_prop}
	 	Suppose that Assumption (A) holds. For every $t \in [0,T]$ and $\xi$, $\eta\in L^2({\Omega}^1,\mathcal{F}^1_t,\mathbb{P}^1;\mathbb{R}^d)$, with $\mathcal{L}(\xi) = \mathcal{L}(\eta)$, it holds that
$
			V(t,\xi) = V(t,\eta)$. 
	\end{prop}
	\begin{proof}
	This proof is modified from \cite[Theorem 3.6]{cosso2023optimal} to suit the case with common noise. We only highlight the differences, and interested readers are referred to \cite[Theorem 3.6]{cosso2023optimal}. Fix $t\in[0,T]$. Let $\xi, \eta\in L^2(\Omega^1,\mathcal{F}^1_t,\mathbb{P}^1;\mathbb{R}^d)$ with $\mathcal{L}(\xi) = \mathcal{L}(\eta)$. By Proposition \ref{property_V}, the function $\xi\mapsto V(t,\xi)$ is continuous. Thus, without loss of generality (see \cite[Theorem 3.6, Substep 2.2]{cosso2023optimal}), we consider the case where $\xi$, $\eta$ are discrete, i.e.,
		\begin{align*}
		\mathcal{L}(\xi) = \mathcal{L}(\eta) = \sum_{i=1}^{m}p_i\delta_{x^i},
		\end{align*}
		for some $\{x^1,x^2,\ldots,x^{m}\}\subset \mathbb{R}^d$ with $x^i$ being distinct. Here $\delta_{x^i}$ is the Dirac measure at $x^i$. For $i=1,2,\ldots,m$, the numbers $p_i$ satisfy $p_i >0$ and $\sum_{i=1}^{m} p_i = 1$. By \cite[Lemma B.3]{cosso2023optimal}, there exist two $\mathcal{F}^1_t$-measurable random variables $U_\xi$ and $U_\eta$ with uniform distribution on $[0,1]$, such that $\xi$ and $U_\xi$ are independent, as are $\eta$ and $U_\eta$. Let $\mathbb{F}^{W,W^0,t}$ be the $\mathbb{P}$-completion of the filtration generated by $(W_{s \vee t}-W_t)_{s \geq0}$ and $(W^0_{s})_{s \geq0}$ and let $Prog(\mathbb{F}^{W,W^0,t})$ be the $\sigma$-algebra of $[0,T]\times\Omega$ of all $\mathbb{F}^{W,W^0,t}$-progressively measurable sets. Fixing $\alpha\in \mathcal{A}_t$, we claim that there exists a measurable function:
		\begin{align*}
			a:\pig([0,T]\times\Omega\times\mathbb{R}^d\times[0,1],Prog(\mathbb{F}^{W,W^0,t})\otimes\mathcal{B}(\mathbb{R}^d)\otimes\mathcal{B}([0,1])\pig)\to (A,\mathcal{B}(A))
		\end{align*}
		such that
		\begin{align}
			\label{law_invariance}
			&\h{-10pt}\mathcal{L}\Big(\xi,(a_s(\xi,U_\xi))_{s\in[t,T]},(W_s^0)_{s\geq 0},(W_s-W_t)_{s\in[t,T]}\Big)\nonumber \\=& \mathcal{L}\Big(\xi,(\alpha_s)_{s\in[t,T]},(W_s^0)_{s\geq 0},(W_s-W_t)_{s\in[t,T]}\Big)\, .
		\end{align}
		This claim may be seen from the following: denote 
		\begin{align*}
			\hat{\Omega}&=[0,T]\times\Omega,&
			\quad \hat{\mathcal{F}}&=\mathcal{B}([0,T])\otimes\mathcal{F},&\quad
			\hat{\mathbb{P}} &= \lambda_T\otimes\mathbb{P},&\quad
			\bar{E} &= [0,T]\times\Omega,&\quad
			\bar{\mathscr{E}} &= Prog(\mathbb{F}^{W,W^0,t}),\\
			E &= \bar{E}\times\mathbb{R}^d,&\quad 
			\mathscr{E} &= \bar{\mathscr{E}}\otimes\mathcal{B}(\mathbb{R}^d),
		\end{align*} with $\lambda_T$ being the uniform distribution on $([0,T],\mathcal{B}([0,T]))$. Consider the canonical extension of $U_\xi$ to $\hat{\Omega}$, denoted by $\hat{U}_\xi$. Define $\mathcal{I}^{W,W^0,t}:(\hat{\Omega},\hat{\mathcal{F}})\to (\bar{E},\bar{\mathscr{E}})$ the identity map and $\Gamma := (\mathcal{I}^{W,W^0,t},\xi)$. Then $\Gamma : (\hat{\Omega},\hat{\mathcal{F}},\hat{\mathbb{P}})\to (E,\mathscr{E})$ is independent of $U_\xi$ as $U_\xi$ is $\mathcal{F}^1_t$ measurable. Applying \cite[Theorem 6.10]{kallenberg_foundations_2002}, we get a map $\bar{a}:[0,T]\times\Omega\times\mathbb{R}^d\times[0,1]\to (A,\mathcal{B}(A))$, being measurable with respect to the $\sigma$-algebra $Prog(\mathbb{F}^{W,W^0,t})\otimes\mathcal{B}(\mathbb{R}^d)\otimes\mathcal{B}([0,1])$ such that
		\begin{align*}
			\mathcal{L}(\Gamma,\bar{a}(\Gamma,\hat{U}_\xi))=\mathcal{L}(\Gamma,\alpha).
		\end{align*}
		From the definition of $\Gamma$ we conclude that
		\begin{align*}
			&\h{-10pt}\mathcal{L}\Big(\xi,(\bar{a}_s(\xi,U_\xi))_{s\in[0,T]},(W_s-W_t)_{s\in[t,T]},(W^0_s)_{s\geq 0}\Big) \\=& \mathcal{L}\Big(\xi,(\alpha_s)_{s\in[0,T]},(W_s-W_t)_{s\in[t,T]},(W^0_s)_{s\geq 0}\Big).
		\end{align*}
		Given an arbitrary $u_0\in A$, we can choose the required measurable function $a:[0,T]\times\Omega\times\mathbb{R}^d\times[0,1]\to A$ to be
		\begin{align*}
			a_s(\omega,x,u):= u_0 
   \mathbbm{1}_{[0,t)}(s)
+\bar{a}_s(\omega,x,u)
\mathbbm{1}_{[t,T]}(s)
,\quad\text{for any }(s,\omega,x,u)\in[0,T]\times\Omega\times\mathbb{R}^d\times A.
		\end{align*}
		Then $(a_s(\xi,U_\xi))_{s\in[0,T]}$ is $\mathbb{F}$-progressively measurable and (\ref{law_invariance}) is satisfied. The measurable function $a$ satisfies
		\begin{align*}
			&\h{-10pt}\mathcal{L}\Big(\xi,(\alpha_s)_{s\in[t,T]},(W_s-W_t)_{s\in[t,T]},(W^0_s)_{s\geq 0}\Big)\\
			=&\mathcal{L}\Big(\xi,(a_s(\xi,U_\xi))_{s\in[t,T]},(W_s-W_t)_{s\in[t,T]},(W^0_s)_{s\geq 0}\Big)\\
			=&\mathcal{L}\Big(\eta,(a_s(\eta,U_\eta))_{s\in[t,T]},(W_s-W_t)_{s\in[t,T]},(W^0_s)_{s\geq 0}\Big)\\
			=&\mathcal{L}\Big(\eta,(\beta_s
			)_{s\in[t,T]},(W_s-W_t)_{s\in[t,T]},(W^0_s)_{s\geq 0}\Big),
		\end{align*}
	where $\beta_s := a_s(\eta,U_\eta)$. It can be shown therefore that
		\begin{align*}
			\mathcal{L}\pig((X_s^{t,\xi,\alpha})_{s\in[t,T]},(\alpha_s)_{s\in[t,T]},(W^0_s)_{s\geq 0}\pig) = \mathcal{L}\pig((X_s^{t,\eta,\beta})_{s\in[t,T]},(\beta_s)_{s\in[t,T]},(W^0_s)_{s\geq 0}\pig),  
		\end{align*}
		and therefore $J(t,\xi,\alpha) = J(t,\eta,\beta)$, implying $V(t,\xi) = V(t,\eta)$.  
	\end{proof}
\begin{theorem}
		[Dynamic Programming Principle]
		\label{dpp_thm of V}
   Suppose that Assumption (A) holds. The value function $V$ satisfies the dynamic programming principle: for every $t\in[0,T]$ and $\xi \in L^2(\Omega^1,\mathcal{F}_t^1,\mathbb{P}^1;\mathbb{R}^d)$,  it holds that
		\begin{align*}
			V(t,\xi) = &\sup_{\alpha\in\mathcal{A}_t}\Bigg\{\inf_{s\in[t,T]}\mathbb{E}\Bigg[\int_t^s f\Big(r,X_r^{t,\xi,\alpha},\mathbb{P}^{W^0}_{X_r^{t,\xi,\alpha}},\alpha_r\Big)dr+V(s,X^{t,\xi,\alpha}_s)\Bigg]\Bigg\}\\
   = &\sup_{\alpha\in\mathcal{A}_t}\Bigg\{\sup_{s\in[t,T]} \mathbb{E}\Bigg[\int_t^s f\Big(r,X_r^{t,\xi,\alpha},\mathbb{P}^{W^0}_{X_r^{t,\xi,\alpha}},\alpha_r\Big)dr+V(s,X^{t,\xi,\alpha}_s)\Bigg]\Bigg\}.
		\end{align*}
	\end{theorem}
 
	\begin{proof}
		Put
		\begin{align*}
			\underline{\Lambda}(t,\xi):= \sup_{\alpha\in\mathcal{A}_t}\Bigg\{\inf_{s\in[t,T]} \mathbb{E}\Bigg[\int_t^s f\Big(r,X_r^{t,\xi,\alpha},\mathbb{P}_{X_r^{t,\xi,\alpha}}^{W^0},\alpha_r\Big)dr+V(s,X^{t,\xi,\alpha}_s)\Bigg]\Bigg\},\\
			\overline{\Lambda}(t,\xi):= \sup_{\alpha\in\mathcal{A}_t}\Bigg\{\sup_{s\in[t,T]} \mathbb{E}\Bigg[\int_t^s f\Big(r,X_r^{t,\xi,\alpha},\mathbb{P}_{X_r^{t,\xi,\alpha}}^{W^0},\alpha_r\Big)dr+V(s,X^{t,\xi,\alpha}_s)\Bigg]\Bigg\}.
		\end{align*}
  We shall show $\overline{\Lambda}(t,\xi)\leq V(t,\xi)\leq \underline{\Lambda}(t,\xi)$. We start with $V(t,\xi)\leq \underline{\Lambda}(t,\xi)$.
		Let $\alpha\in\mathcal{A}_t$ be arbitrary. Note that our value function $V(t,\xi)$ is in fact defined only for $\xi \in L^2(\Omega^1,\mathcal{F}_t^1,\mathbb{P}^1;\mathbb{R}^d)$, thus $V(s,X_s^{t,\xi,\alpha})$ is in fact random in the sense that $V(s,X_s^{t,\xi,\alpha})(\omega^0) = V(s,X_s^{t,\xi,\alpha}(\omega^0,\cdot))$. Note that by the uniqueness of the SDE, for $\mathbb{P}^0$-a.s. $\omega^0\in\Omega^0$, the process $\big(X_r^{t,\xi,\alpha}(\omega^0,\cdot)\big)_{r\in [s,T]}$ is the unique solution of 
		\begin{align*}
			X_r(\omega^0,\cdot) =& \Big(X_s^{t,\xi,\alpha}+\int_{s}^r b(k,X_{k},\mathbb{P}_{X_{k}}^{W^0},\alpha_k)dk + \int_s^r \sigma(k,X_{k},\alpha_k)dW_k + \int_s^r \sigma^0(k)dW^0_k\Big)(\omega^0,\cdot),
		\end{align*}
		thus for $\mathbb{P}^0\otimes\mathbb{P}^0$-a.s. $(\omega^0,\omega^{00})\in\Omega^0\times\Omega^0$, $\big(X_r^{t,\xi,\alpha}(\omega^0\otimes_s\omega^{00},\cdot)\big)_{r\in [s,T]}$ solves 
		\begin{align}
			\label{dy_1}
			X_r(\omega^0\otimes_s \omega^{00},\cdot) =\,& \Big(X_s^{t,\xi,\alpha}+\int_{s}^r b(k,X_{k},\mathbb{P}_{X_{k}}^{W^0},\alpha_k)dk + \int_s^r \sigma(k,X_{k},\alpha_k)dW_k \nonumber\\
   &+ \int_s^r \sigma^0(k)dW^0_k\Big)(\omega^0\otimes_s \omega^{00},\cdot).
		\end{align}
		Since $X_s^{t,\xi,\alpha}(\omega^0,\cdot)$ is  $\mathcal{F}_s^t$-measurable, we have $X_s^{t,\xi,\alpha}(\omega^0\otimes_s\omega^{00},\cdot) = X_s^{t,\xi,\alpha}(\omega^0,\cdot)$ and for $\mathbb{P}^0$-a.s. $\omega^0\in \Omega^0$, $(\ref{dy_1})$ is equivalent to \begin{align*}
			&X_r(\omega^0\otimes_s \omega^{00},\cdot) 
			\\
			&= X_s^{t,\xi,\alpha}(\omega^0,\cdot) 
   + \Big(\int_{s}^r b(k,X_{k},\mathbb{P}_{X_{k}}^{W^0},\alpha_k)dk+\int_s^r \sigma(k,X_{k},\alpha_k)dW_k + \int_s^r \sigma^0(k)dW^0_k\Big)(\omega^0\otimes_s \omega^{00},\cdot).
		\end{align*}
		For $\mathbb{P}^0$-a.s. $\omega^0\in \Omega^0$ and $\tilde{\omega} = (\tilde{\omega}^0,\tilde{\omega}^1) \in \Omega^0\times \Omega^1$, we define $\hat{\alpha}^{\omega^0}_r (\tilde{\omega}):= \alpha_r (\omega^0\otimes_s\tilde{\omega}^0,\tilde{\omega}^1)$. Then for $\mathbb{P}^0$-a.s. $\omega^0 \in \Omega^0$, $X_r^{t,\xi,\alpha}(\omega^0\otimes_s \omega^{00},\cdot)$ solves 
		\begin{align*}
			Y_r(\omega^{00},\omega^1) =\,& X_s^{t,\xi,\alpha}(\omega^0,\omega^1) \\&+ \Big(\int_{s}^r b(k,Y_{k},\mathbb{P}_{Y_{k}}^{W^0},\hat{\alpha}^{\omega^0}_k)dk+\int_s^r \sigma(k,Y_{k},\hat{\alpha}^{\omega^0}_k)dW_k + \int_s^r \sigma^0(k)dW^0_k\Big)(\omega^{00},\omega^1).
		\end{align*}
		Therefore
		\begin{align*}
			\mathbb{E}V(s,X_s^{t,\xi,\alpha}) 
			=\,&\mathbb{E}^0\pig[{V(s,X_s^{t,\xi,\alpha}(\omega^0,\cdot))}\pig]\\
			\geq\,&\mathbb{E}^0\Bigg\{\mathbb{E}^{00}\mathbb{E}^1\Bigg[\int_s^T f\left(r,X_r^{t,\xi,\alpha}(\omega^0\otimes_s\omega^{00},\omega^1),\mathbb{P}^{W^0}_{X_r^{t,\xi,\alpha}}(\omega^0\otimes_s\omega^{00}),\hat{\alpha}^{\omega^0}_r(\omega^{00},\omega^1)\right)dr\\
			&\h{60pt} +g\left(X_T^{t,\xi,\alpha}(\omega^0\otimes_s\omega^{00},\omega^1),\mathbb{P}_{X_T^{t,\xi,\alpha}}^{W^0}(\omega^0\otimes_s\omega^{00})\right)\Bigg]\Bigg\}\\
			=\,& \mathbb{E}\Bigg[\int_s^T f(r,X_r^{t,\xi,\alpha},\mathbb{P}^{W^0}_{X_r^{t,\xi,\alpha}},\alpha_r)dr + g\Big(X_T^{t,\xi,\alpha},\mathbb{P}_{X_T^{t,\xi,\alpha}}^{W^0}\Big)\Bigg].
		\end{align*}
Therefore
            \begin{align*}
                \mathbb{E}\Bigg[\int_t^s f\Big(r,X_r^{t,\xi,\alpha},\mathbb{P}_{X_r^{t,\xi,\alpha}}^{W^0},\alpha_r\Big)dr+V(s,X^{t,\xi,\alpha}_s)\Bigg]\geq J(t,\xi,\alpha),
            \end{align*}
            since $s\in [t,T]$ is arbitrary, we have
            \begin{align*}
                \inf_{s\in[t,T]}\mathbb{E}\Bigg[\int_t^s f\Big(r,X_r^{t,\xi,\alpha},\mathbb{P}_{X_r^{t,\xi,\alpha}}^{W^0},\alpha_r\Big)dr+V(s,X^{t,\xi,\alpha}_s)\Bigg]\geq J(t,\xi,\alpha).
            \end{align*}
            As $\alpha$ is arbitrary, so
		\begin{align*}
			\underline{\Lambda}(t,\xi) \geq V(t,\xi).
		\end{align*}
On the other hand, for any $\beta\in\mathcal{A}_t$, $s\in[t,T]$, by the continuity of $J(\cdot,\cdot,\cdot)$ and $V(\cdot,\cdot)$, we use the measurable selection theorem (see \cite{measurable_selection} for instance) to find a $\gamma:(\Omega^0,\mathcal{F}_s^{W^0},\mathbb{P}^0)\to(\mathcal{A}_s,\mathcal{B}_{\mathcal{A}_s})$ such that
		\begin{align*}
			V(s,X^{t,\xi,\beta}_s(\omega^0))
			\leq\,& \mathbb{E}\Bigg[\int_s^T f\left(r,X_r^{s,X^{t,\xi,\beta}_s(\omega^0),\gamma(\omega^0)},
   \mathbb{P}^{W^0}_{X_r^{s,X^{t,\xi,\beta}_s(\omega^0),\gamma(\omega^0)}},\gamma_r(\omega^0)\right)dr \\
   &\h{25pt}+ g\left(X_T^{s,X^{t,\xi,\beta}_s(\omega^0),\gamma(\omega^0)},\mathbb{P}_{X_T^{s,X^{t,\xi,\beta}_s(\omega^0),\gamma(\omega^0)}}^{W^0}\right)\Bigg]+\varepsilon.
		\end{align*}
		Put $\Gamma_r(\omega^0,\omega^1):= \beta_r(\omega^0,\omega^1)
   \mathbbm{1}_{[0,s)}(r)
   + \gamma_r(\omega^0)(\omega^0,\omega^1)
   \mathbbm{1}_{[s,T]}(r)$. Then by \cite[Lemma 2.1]{soner_dynamic_2002}, we have $\Gamma\in\mathcal{A}_t$ and
		\begin{align*}
			&\h{-10pt}\mathbb{E}\Bigg[\int_t^s f(r,X_r^{t,\xi,\beta},\mathbb{P}_{X_r^{t,\xi,\beta}}^{W^0},\beta_r)dr+V(s,X^{t,\xi,\beta}_s)\Bigg]\\
			\leq\, &\mathbb{E}\Bigg[\int_t^T f(r,X_r^{t,\xi,\Gamma},\mathbb{P}_{X_r^{t,\xi,\Gamma}}^{W^0},\Gamma_r)dr+g(X_T^{t,\xi,\Gamma},\mathbb{P}_{X_T^{t,\xi,\Gamma}}^{W^0})\Bigg]+\varepsilon\\
			\leq\, &V(t,\xi)+\varepsilon.
		\end{align*}
		Taking supremum in $s\in [t,T]$ on both sides gives
		\begin{align*}
			\sup_{s\in[t,T]}\mathbb{E}\Bigg[\int_t^s f(r,X_r^{t,\xi,\beta},\mathbb{P}_{X_r^{t,\xi,\beta}}^{W^0},\beta_r)dr+V(s,X^{t,\xi,\beta}_s)\Bigg]\leq V(t,\xi)+\varepsilon
		\end{align*}
		and by the arbitrariness of $\varepsilon>0$ and $\beta\in\mathcal{A}_t$, we have $\overline{\Lambda}(t,\xi)\leq V(t,\xi)$.
	\end{proof}
	\noindent From Proposition \ref{Law Invariance_prop}, we can define a function $v(t,\mu):[0,T]\times\mathcal{P}_2(\mathbb{R}^d) \to \mathbb{R}$ such that for any $t\in [0,T]$ and $\mu \in \mathcal{P}_2(\mathbb{R}^d)$, it is given by
	\begin{align}
		\label{value_function_after_law_invariance}
		v(t,\mu) := V(t,\xi)
	\end{align}
     for any $\xi\in L^2(\Omega^1,\mathcal{F}_t^1,\mathbb{P}^1;\mathbb{R}^d)$ such that $\mathcal{L}(\xi)=\mu$, where $V(t,\xi)$ is defined in \eqref{def. value function with xi}. From Theorem \ref{dpp_thm of V} and \eqref{value_function_after_law_invariance}, the dynamic programming principle may therefore be recast as
	\begin{theorem}
		[Dynamic Programming Principle]
		\label{dpp_thm}  
			Suppose that Assumption (A) hold. The value function $v$ satisfies the dynamic programming principle: for every $t\in[0,T]$ and $\mu \in \mathcal{P}_2(\mathbb{R}^d)$, it holds that
		\begin{align*}
			v(t,\mu) = & \sup_{\alpha\in\mathcal{A}_t}\sup_{s\in[t,T]}\Bigg\{\mathbb{E}\Bigg[\int_t^s f\Big(r,X_r^{t,\xi,\alpha},\mathbb{P}^{W^0}_{X_r^{t,\xi,\alpha}},\alpha_r\Big)dr+v\pig(s,\mathbb{P}_{X^{t,\xi,\alpha}_s}^{W^0}\pig)\Bigg]\Bigg\}\\
   = & \sup_{\alpha\in\mathcal{A}_t}\inf_{s\in[t,T]}\Bigg\{\mathbb{E}\Bigg[\int_t^s f\Big(r,X_r^{t,\xi,\alpha},\mathbb{P}^{W^0}_{X_r^{t,\xi,\alpha}},\alpha_r\Big)dr+v\pig(s,\mathbb{P}_{X^{t,\xi,\alpha}_s}^{W^0}\pig)\Bigg]\Bigg\},
		\end{align*}
		for any $\xi\in L^2(\Omega^1,\mathcal{F}_t^1,\mathbb{P}^1;\mathbb{R}^d)$ such that $\mathcal{L}(\xi) = \mu$.
	\end{theorem}

\section{Smooth Variational Principle}
 This section is devoted to the study of the regularity of our chosen gauge function for the comparison theorem of viscosity solutions. As outlined in \cite{cosso_master_2022}, the key to the comparison theorem of viscosity subsolutions/supersolutions is the existence of a smooth gauge function, so that when the value function is slightly perturbed by the gauge function, it can attain its maximum/minimum, as indicated by the Borwein-Preiss variational principle. To this purpose, we shall make use of the Gaussian regularized sliced Wasserstein distance defined in \cite[Section 2]{bayraktar_smooth_2023} as our gauge function. By a gauge function, we mean the following:
		\begin{definition}
			Let $d_2$ be a metric on $\mathcal{P}_2(\mathbb{R}^d)$ such that $(\mathcal{P}_2(\mathbb{R}^d), d_2)$ is complete. Consider the set $[0, T] \times \mathcal{P}_2(\mathbb{R}^d)$ endowed with the metric $((t, \mu),(s, \nu)) \mapsto|t-s|+d_2(\mu, \nu)$. A map $\rho:([0, T] \times \mathcal{P}_2(\mathbb{R}^d))^2 \rightarrow[0,+\infty)$ is said to be a gauge-type function if the following hold:
			\begin{enumerate}[(1).]
				\item $\rho((t, \mu),(t, \mu))=0$, for every $(t, \mu) \in[0, T] \times \mathcal{P}_2(\mathbb{R}^d)$;
				\item $\rho$ is continuous on $([0, T] \times \mathcal{P}_2(\mathbb{R}^d))^2$;
				\item for any $\varepsilon>0$, there exists $\eta>0$ such that, for any $(t, \mu),(s, \nu) \in[0, T] \times \mathcal{P}_2(\mathbb{R}^d)$, the inequality $\rho((t, \mu),(s, \nu)) \leq \eta$ implies $|t-s|+d_2(\mu, \nu) \leq \varepsilon$.
			\end{enumerate}
   \label{def. gauge}
		\end{definition}
\subsection{Properties of the Smooth Gauge Function}
		
	We recall the definition of the Gaussian regularized sliced Wasserstein distance in  \cite[Section 2]{bayraktar_smooth_2023}.

	\begin{definition}
		[Gaussian regularized sliced Wasserstein distance] For any $\mu \in \mathcal{P}_2(\mathbb{R}^d)$ and $\theta \in \mathcal{S}^{d-1}:=\{x\in \mathbb{R}^d:|x|=1\}$, we define the mapping $P_\theta: \mathbb{R}^d \rightarrow \mathbb{R}$ by $P_\theta(x):=\theta^\top x$ and the pushforward measure $\mu_\theta:=P_\theta \sharp \mu \in \mathcal{P}_2(\mathbb{R})$. For any $\mu, \nu \in \mathcal{P}_2(\mathbb{R}^d)$, the sliced Wasserstein distance is defined by
		$$
		S W_2(\mu, \nu):=\left\{\int_{\mathcal{S}^{d-1}} \pig[\mathcal{W}_2(\mu_\theta, \nu_\theta)\pigr]^2 d \theta\right\}^{1/2},
		$$
		where the integration is taken with respect to the standard spherical measure on $\mathcal{S}^{d-1}$. Moreover, the Gaussian regularized sliced Wasserstein distance is defined by
		$$
		S W_2^\sigma(\mu, \nu):=S W_2(\mu^\sigma, \nu^\sigma),
		$$
		where $\mu^\sigma:=\mu * \mathcal{N}_\sigma$ and $\mathcal{N}_\sigma \in \mathcal{P}_2(\mathbb{R}^d)$ is the normal distribution with variance $\sigma^2 I_d$ for $\sigma \in(0, \infty)$. By abuse of notation, $\mathcal{N}_\sigma(y)dy$ also denotes the $1$-dimensional normal distribution with mean 0 and variance $\sigma^2$. We have that $\mu^\sigma_\theta:=(\mu^\sigma)_\theta=$ $(\mu * \mathcal{N}_\sigma)_\theta=\mu_\theta * \mathcal{N}_\sigma=(\mu_\theta)^\sigma$, where the first Gaussian is $d$-dimensional and the second one is $1$-dimensional. 
	\end{definition}
	\noindent Denoting the cumulative distribution function of $\mu$ by $F_\mu$, it is well known that (for example, see \cite[Theorem 2.18]{villani_topics_2003}) in the one-dimensional case, the optimal transport map from $\mu_\theta^\sigma$ to $\nu_\theta^\sigma$ is given by 
	\begin{align*}
		T_{\theta,\mu}^\sigma(x):=F_{\nu_\theta^\sigma}^{-1}\pig(F_{\mu_\theta^\sigma}(x)\pig),
	\end{align*}
	which satisfies $\big[\mathcal{W}_2(\mu_\theta^\sigma, \nu_\theta^\sigma)\bigr]^2=\int_{\mathbb{R}} \frac{1}{2}\big|x-T_{\theta,\mu}^\sigma(x)\bigr|^2 \mu_\theta^\sigma(d x)$. Note that because it is the optimal transport map from $\mu_\theta^\sigma$ to $\nu_\theta^\sigma$, for any $X\in L^2(\Omega,\mathcal{F},\mathbb{P};\mathbb{R}^d)$ such that $\mathcal{L}(X) = \mu$, 
	\begin{align}
		\label{optimal_transport_map}
		\mathcal{L}\pig(T_{\theta,\mu}^\sigma(\theta^\top(X+N_\sigma))\pig) = \nu_\theta^\sigma,
	\end{align}
	where $N_\sigma \in \mathbb{R}^d$ is a normal random variable independent of $X$, following the distribution $\mathcal{N}_\sigma$.\\ \hfill\\
The following lemma is taken from \cite[Lemmas 2.1]{bayraktar_smooth_2023}.
	\begin{lemma}
		\label{smooth_var_lemma_1}
		For any {\color{black}$\sigma > 0$}, $(\mathcal{P}_2(\mathbb{R}^d),SW_2^\sigma)$ is a complete metric space, and it is equal to $(\mathcal{P}_2(\mathbb{R}^d),\mathcal{W}_2)$ as a topological space.
	\end{lemma}
\noindent  
In order to establish the regularity of the derivatives of this metric, we require an estimate on the optimal transport map as outlined below.
 \begin{lemma}
    \label{dominator}
    Fix $\sigma>0$. Let $\mathcal{V}$, $\mathcal{U}\subset \mathcal{P}_2(\mathbb{R}^d)$ be any compact sets such that all measures in $\mathcal{V}$, $\mathcal{U}$ are of the form $\mu * \mathcal{N}_\sigma$ for some $\mu\in \mathcal{P}_2(\mathbb{R}^d)$. Let $\mathcal{X}\subset\mathbb{R}^d$ be any compact set. There is a positive constant $C_{\mathcal{U},\mathcal{V},\mathcal{X},\sigma}$ depending only on $\mathcal{U}$, $\mathcal{V}$, $\mathcal{X}$ and $\sigma$ such that 
    \begin{align*}
        |T_{\theta,\mu}^\sigma(z)|\exp\left(\dfrac{-(z-\theta^\top x)^2}{2\sigma^2}\right)\leq 
        C_{\mathcal{U},\mathcal{V},\mathcal{X},\sigma}\left\{ 
\exp\left(\dfrac{-z^2}{4\sigma^2}\right)
\vee\left[ \pig(|z|^{1/2}+1\pig)
\exp\left(\frac{-z^2}{8\sigma^2}\right)\right]\right\},
    \end{align*}
     for all $\nu^\sigma\in\mathcal{V}$, $\mu^\sigma\in\mathcal{U}$, $x\in\mathcal{X}$, $z\in\mathbb{R}$ and $\theta\in\mathcal{S}^{d-1}$.
    \end{lemma}
    \begin{proof}
        The values of the constants in this proof may change from line to line in this proof, but they still depend on the same set of parameters as mentioned. First of all, we show that the map $\mathcal{U}\times\mathcal{S}^{d-1}\ni(\mu^\sigma,\theta)\mapsto \mu_\theta^\sigma\in\mathcal{P}_2(\mathbb{R})$ is continuous. Let $(\mu_n^\sigma,\theta_n)_{n\in \mathbb{N}}$ be a sequence converging to $(\mu^\sigma,\theta)$. For a $L^2$ random variable $X$ such that $\mathcal{L}(X)=\mu^\sigma$, it yields that
        \begin{align*}
            \mathcal{W}_2(\mu_{n,\theta_n}^\sigma,\mu_{\theta}^\sigma)&\leq
\mathcal{W}_2(\mu_{n,\theta_n}^\sigma,\mu_{\theta_n}^\sigma)+\mathcal{W}_2(\mu_{\theta_n}^\sigma,\mu_{\theta}^\sigma)\\
            &\leq\mathcal{W}_2(\mu_{n}^\sigma,\mu^\sigma)+\sqrt{\mathbb{E}|\theta_n^\top X-\theta^\top X|^2},
        \end{align*}
        where $\mu_{n,\theta_n}^\sigma:=(\mu_n)_{\theta_n}^\sigma$ and $\mu_{n}^\sigma:=(\mu_n)^\sigma$. Taking $n\to\infty$ and using the dominated convergence theorem, we conclude our desired continuity. It also implies the continuity of the map $\mathcal{V}\times\mathcal{S}^{d-1}\ni(\nu^\sigma,\theta)\mapsto \int_{\mathbb{R}}|y|^2\nu^\sigma_\theta(dy)$ and hence $C_{\mathcal{V},\sigma}:=\displaystyle\sup_{\nu^\sigma\in\mathcal{V},\theta\in\mathcal{S}^{d-1}}$$\int_{\mathbb{R}}|y|^2\nu^\sigma_\theta(dy) < \infty$ by the compactness of $\mathcal{V}$. By Markov's inequality, we have
        \begin{align*}
            F_{\nu^\sigma_\theta}(z)\leq C_{\mathcal{V},\sigma}/z^2, \text{ for any $z<0$},
        \end{align*}        
        and hence {\color{black}for $z=F^{-1}_{\nu^\sigma_\theta}(y)<0$ with $0<y<F_{\nu^\sigma_\theta}(0)$, we have}
\begin{equation}
	|F^{-1}_{\nu^\sigma_\theta}(y)|\le\sqrt{C_{\mathcal{V},\sigma}/y}, 
 \text{ for $0<y<F_{\nu^\sigma_\theta}(0)$}.\label{40}
\end{equation}
Let $m_{\sigma,\theta,\mu}$ be the median of $\mu_\theta^\sigma$, i.e., $m_{\sigma,\theta,\mu}:= F^{-1}_{\mu^\sigma_\theta}(1/2)$. We have that $C_{\mathcal{U},\sigma}:= \h{-3pt}\displaystyle\sup_{\mu^\sigma\in\mathcal{U},\theta\in\mathcal{S}^{d-1}}
\h{-3pt}\pig|F^{-1}_{\mu_\theta^\sigma}(1/2)\pig|<\infty$ due to the continuity of the map $\mathcal{U} \times\mathcal{S}^{d-1}\ni(\mu^\sigma,\theta)\mapsto F^{-1}_{\mu_\theta^\sigma}(1/2)$ (see \cite[Lemma 21.2]{van2000asymptotic}). Therefore, there is a large enough positive constant $L_{\mathcal{U},\sigma}>0$ depending only on $\sigma$, $\mathcal{U}$ such that for any $z<-L_{\mathcal{U},\sigma}$, $\mu^\sigma\in\mathcal{U}$ and $\theta\in\mathcal{S}^{d-1}$, we have 
\begin{align}
\label{F_nu_estimate}
F_{\mu^\sigma_\theta}(z)\nonumber
&=\int^{\infty}_{z-m_{\sigma,\theta,\mu}}
 F_{\mu_\theta}(z-y)\mathcal{N}_\sigma(y)dy+\int_{-\infty}^{z-m_{\sigma,\theta,\mu}}F_{\mu_\theta}(z-y)\mathcal{N}_\sigma(y)dy\nonumber\\
&\geq \int_{-\infty}^{z-m_{\sigma,\theta,\mu}}F_{\mu_\theta}(m_{\sigma,\theta,\mu})\mathcal{N}_\sigma(y)dy\nonumber\\
&=\frac{1}{2}\int_{-\infty}^{z-m_{\sigma,\theta,\mu}} \mathcal{N}_\sigma(y)dy\nonumber\\
&\geq\dfrac{\sigma^2}{2|z-m_{\sigma,\theta,\mu}|}
\left(1-\dfrac{\sigma^2}{|z-m_{\sigma,\theta,\mu}|^2}\right)
\exp\left(-\frac{|z-m_{\sigma,\theta,\mu}|^2}{2\sigma^2}\right)\nonumber\\
&\geq\dfrac{\sigma^2}{4|z-m_{\sigma,\theta,\mu}|}
\exp\left(-\frac{|z-m_{\sigma,\theta,\mu}|^2}{2\sigma^2}\right),
\end{align}
where the second last inequality comes from, for instance, \cite[Theorem 1.2.6]{durrett2019probability}. Note that $\displaystyle\inf_{\nu^\sigma\in\mathcal{V},\theta\in\mathcal{S}^{d-1}} \h{-3pt} F_{\nu_\theta^\sigma}(0)>0$ since every measures $\nu^\sigma$ is Gaussian regularized, and the infimum can actually be attained due to compactness. There is a $L_{\mathcal{U},\mathcal{V},\sigma}\geq L_{\mathcal{U},\sigma}$ such that $0<F_{\mu_\theta^\sigma}(z)<F_{\nu^\sigma_\theta}(0)$ for any $z<-L_{\mathcal{U},\mathcal{V},\sigma}$. Thus, by \eqref{40}, 
$$
\pig|T_{\theta,\mu}^\sigma(z)\pig|
=\pig|F^{-1}_{\nu^\sigma_\theta}(F_{\mu_\theta^\sigma}(z))\pig|
\le \dfrac{2C_{\mathcal{V},\sigma}^{1/2}|z-m_{\sigma,\theta,\mu}|^{1/2}}{\sigma}
\exp\left(\frac{|z-m_{\sigma,\theta,\mu}|^2}{4\sigma^2}\right) 
$$
for any $z<-L_{\mathcal{U},\mathcal{V},\sigma}$. The case for the right tail is similar. By Markov's inequality, we have 
\begin{align}
    |F_{\nu_\theta^\sigma}^{-1}(y)|\leq \sqrt{C_{\mathcal{V},\sigma}/(1-y)}, \text{ for }F_{\nu_\theta^\sigma}(0)<y<1.
    \label{633}
\end{align}
Taking $L_{\mathcal{U},\mathcal{V},\sigma}>0$ large enough, we see that for any $z>L_{\mathcal{U},\mathcal{V},\sigma}$,
    \begin{align*}
F_{\mu^\sigma_\theta}(z)
&=\int^{\infty}_{z-m_{\sigma,\theta,\mu}}
 F_{\mu_\theta}(z-y)\mathcal{N}_\sigma(y)dy+\int_{-\infty}^{z-m_{\sigma,\theta,\mu}}F_{\mu_\theta}(z-y)\mathcal{N}_\sigma(y)dy\\
&\leq \int_{z-m_{\sigma,\theta,\mu}}^{\infty}\frac{1}{2}\mathcal{N}_\sigma(y)dy+\int_{-\infty}^{z-m_{\sigma,\theta,\mu}}\mathcal{N}_\sigma(y)dy\\
&=1-\int_{z-m_{\sigma,\theta,\mu}}^{\infty}\frac{1}{2}\mathcal{N}_\sigma(y)dy,
\end{align*}
and therefore
\begin{align*}
    1-F_{\mu^\sigma_\theta}(z)\geq \frac{1}{2}\int_{z-m_{\sigma,\theta,\mu}}^{\infty}\mathcal{N}_\sigma(y)dy.
\end{align*}
By the symmetry of the distribution of $\mathcal{N}_\sigma$, a similar estimate as \eqref{F_nu_estimate} can be derived. Therefore, together with \eqref{40} and \eqref{633}, we have that for any $|z|>L_{\mathcal{U},\mathcal{V},\sigma}$,
\begin{align*}
\pig|T_{\theta,\mu}^\sigma(z)\pig|
 =\pig|F^{-1}_{\nu^\sigma_\theta}(F_{\mu_\theta^\sigma}(z))\pig|
	\le \dfrac{2C_{\mathcal{V},\sigma}^{1/2}|z-m_{\sigma,\theta,\mu}|^{1/2}}{\sigma}
\exp\left(\frac{|z-m_{\sigma,\theta,\mu}|^2}{4\sigma^2}\right).
\end{align*}
Since both the distribution and quantile functions of probability measures in $\mathcal{U}$, $\mathcal{V}$ are continuous and strictly increasing on the spatial variable, $T_{\theta,\mu}^\sigma(z)$ is strictly increasing. For the region $|z|\leq L_{\mathcal{U},\mathcal{V},\sigma}$, as the map $T_{\theta,\mu}^\sigma(z)$ is continuous on $\mu$ and $\theta$, $K:= \displaystyle\sup_{\mu^\sigma\in\mathcal{U},\theta\in\mathcal{S}^{d-1}}|T_{\theta,\mu}^\sigma(L_{\mathcal{U},\mathcal{V},\sigma})|\vee |T_{\theta,\mu}^\sigma(-L_{\mathcal{U},\mathcal{V},\sigma})| < \infty$. Since $T_{\theta,\mu}^\sigma(z)$ is strictly increasing in $z$,  we have \\$C_{\mathcal{U},\mathcal{V},\sigma}:= \displaystyle\sup_{\mu^\sigma\in\mathcal{U},\theta\in\mathcal{S}^{d-1}}
\displaystyle\sup_{ |z|\leq L_{\mathcal{U},\mathcal{V},\sigma}}|T_{\theta,\mu}^\sigma(z)|<\infty$. Therefore, for any $x\in \mathbb{R}^d$ and $z\in \mathbb{R}$, we have
\begin{align*}
&\h{-10pt}|T_{\theta,\mu}^\sigma(z)|\exp\left(-\dfrac{|z-\theta^\top x|^2}{2\sigma^2}\right)\nonumber\\
\leq &\left[ C_{\mathcal{U},\mathcal{V},\sigma}\exp\left(-\dfrac{|z-\theta^\top x|^2}{2\sigma^2}\right)\right]
\vee
\left[ \dfrac{2C_{\mathcal{V},\sigma}^{1/2}| z-m_{\sigma,\theta,\mu}|^{1/2}}{\sigma}
\exp\left(\frac{|z-m_{\sigma,\theta,\mu}|^2-2|z-\theta^\top x|^2}{4\sigma^2}\right)\right].
\end{align*}
Recall that $|m_{\sigma,\theta,\mu}| =|F_{\mu_\theta^\sigma}^{-1}(1/2)|\leq C_{\mathcal{U},\sigma}$. Thus, we can use Young's inequality to obtain that
\begin{align*}
&\h{-10pt}|T_{\theta,\mu}^\sigma(z)|\exp\left(-\dfrac{|z-\theta^\top x|^2}{2\sigma^2}\right)\nonumber\\
\leq &\left[ C_{\mathcal{U},\mathcal{V},\sigma}
\exp\left(\dfrac{-z^2/2+|\theta^\top x|^2}{2\sigma^2}\right)\right]\vee\left[ \dfrac{2C_{\mathcal{V},\sigma}^{1/2}\pig(|z|^{1/2}+C_{\mathcal{U},\sigma}^{1/2}\pig)}{\sigma}
\exp\left(\frac{-z^2/2+14|\theta^\top x|^2+5C_{\mathcal{U},\sigma}^2}{4\sigma^2}\right)\right].
\end{align*}
As $x$ is in the compact set $\mathcal{X}$, we conclude the required result.
\end{proof}
\noindent We have the following lemmas for the derivatives of $\pig[S W_2^\sigma(\mu, \nu)\pigr]^2$, when fixing $\nu\in\mathcal{P}_2(\mathbb{R}^d)$ and $\sigma>0$:
\begin{lemma}
		\label{estimate_smooth}
		Let $\nu \in \mathcal{P}_2(\mathbb{R}^d)$ and $\sigma > 0$ be fixed. It holds that
  \begin{enumerate}[(1).]
      \item the mapping $\mathcal{P}_2(\mathbb{R}^d)\ni \mu \mapsto \pig[S W_2^\sigma(\mu, \nu)\pigr]^2
		$
		is L-differentiable with 
		$$
		\partial_\mu \pig[S W_2^\sigma(\mu, \nu)\pigr]^2(x)=\int_{\mathcal{S}^{d-1}} \theta\Big\{\theta^{\top} x-\mathbb{E}\pig[T_{\theta,\mu}^\sigma(\theta^{\top}(x+N_\sigma))\pig]\Big\} d \theta;
		$$
  \item there is a constant $C_d>0$ depending only on $d$ such that
		$$
\int_{\mathbb{R}^d}\left|\partial_\mu \pig[S W_2^\sigma(\mu, \nu)\pigr]^2(x)\right|^2 \mu(d x) \leq C_d\Big(\int_{\mathbb{R}^d}|x|^2 \mu(d x)+\int_{\mathbb{R}^d}|y|^2 \nu^\sigma(d y)\Big);
		$$    
  \item the mapping $\mathbb{R}^d \times \mathcal{P}_2(\mathbb{R}^d) \ni (x,\mu) \mapsto \partial_\mu \pig[S W_2^\sigma(\mu, \nu)\pigr]^2(x)$ is jointly continuous.
  \end{enumerate}
		\end{lemma}
 \begin{proof}
    We follow \cite[Lemma 2.2]{bayraktar_smooth_2023} to obtain items (1) and (2). To prove (3), we let $\{x_n\}_{n\in \mathbb{N}}\subset \mathbb{R}^d$ and $\{\mu_n\}_{n\in \mathbb{N}} \subset \mathcal{P}_2(\mathbb{R}^d)$ such that $|x_n -x|\to 0$ and $\mathcal{W}_2(\mu_n, \mu) \to 0$ as $n \to \infty$. It is sufficient to consider
\begin{align*}
    &\h{-10pt} \int_{\mathcal{S}^{d-1}} \theta\Big\{\mathbb{E}\pig[T_{\theta,\mu_n}^\sigma\pig(\theta^{\top}(x_n+N_\sigma)\pig)
    -T_{\theta,\mu}^\sigma\pig(\theta^{\top}(x+N_\sigma)\pig)\pig]\Big\} d \theta \nonumber\\
    =\,&\int_{\mathcal{S}^{d-1}} \theta\Big\{\mathbb{E}\pig[T_{\theta,\mu_n}^\sigma\pig(\theta^{\top}x_n+N_\sigma\pig)
    -T_{\theta,\mu}^\sigma\pig(\theta^{\top}x+N_\sigma\pig)\pig]\Big\} d \theta \nonumber\\
    =\,& \int_{\mathcal{S}^{d-1}}
    \theta\int_{\mathbb{R}}
\Big[ T_{\theta,\mu_n}^\sigma\pig(\theta^{\top}x_n+y\pig)
    -T_{\theta,\mu}^\sigma\pig(\theta^{\top}x+y\pig) \Big] \dfrac{1}{(2\pi \sigma^2)^{1/2} }\exp\left(-\dfrac{y^2}{2\sigma^2}\right)dyd \theta\nonumber\\
    =\,& \dfrac{1}{(2\pi \sigma^2)^{1/2} }\int_{\mathcal{S}^{d-1}}\int_{\mathbb{R}}
\theta T_{\theta,\mu_n}^\sigma(y)
\exp\left(-\dfrac{|y- \theta^\top x_n|^2}{2\sigma^2}\right)
- \theta T_{\theta,\mu}^\sigma(y)\exp\left(-\dfrac{|y-\theta^\top x|^2}{2\sigma^2}\right)dyd \theta,
    \end{align*}
where the normal random variable in the first line is $d$-dimensional and that in the second line is $1$-dimensional. By Lemma \ref{dominator}, the pointwise convergence of $T^\sigma_{\theta,\mu_n}(z)$ and the Lebesgue dominated convergence theorem, we conclude that the above term converges to zero as $n \to \infty$, which implies the desired continuity.
    
 \end{proof}

	\begin{lemma}
		\label{estimate_smooth 2}
		Let $\nu \in \mathcal{P}_2(\mathbb{R}^d)$ and $\sigma > 0$ be fixed. It holds that
  \begin{enumerate}[(1).]
      \item for each $\mu \in \mathcal{P}_2(\mathbb{R}^d)$, the mapping $
		\mathbb{R}^d\ni x \mapsto \p_\mu\pig[S W_2^\sigma(\mu, \nu)\pigr]^2(x)
		$
		is differentiable with respect to $x$ with 
	\begin{align*}
			\partial_x\partial_\mu \pig[S W_2^\sigma(\mu, \nu)\pigr]^2(x)&=\mathbb{E} \left[\int_{\mathcal{S}^{d-1}} \theta \theta^{\top}\left(1-\dfrac{d}{dz}T_{\theta,\mu}^\sigma(z)\Big|_{z=\theta^{\top}\left(x+N_\sigma\right)}\right) d \theta \right]\\
   &=\int_{\mathcal{S}^{d-1}}\theta\theta^\top d\theta + \int_{\mathbb{R}}\int_{\mathcal{S}^{d-1}}\left(\theta\theta^\top\frac{(\theta^\top x-y)}{\sigma^2}T_{\theta,\mu}^\sigma(y)\mathcal{N}_\sigma(\theta^\top x-y)\right)d\theta dy;
		\end{align*}
  \item there is a constant $C_d>0$ depending only on $d$ such that
  		$$
		\int_{\mathbb{R}^d}\left|\partial_x\partial_\mu \pig[S W_2^\sigma(\mu, \nu)\pigr]^2(x)\right|^2\mu(dx) \leq C_d\left(1+\frac{1}{\sigma^2}\int_{\mathbb{R}^d}|x|^2\nu^\sigma (dx) \right);
		$$ 
  \item  the mapping $\mathbb{R}^d \times \mathcal{P}_2(\mathbb{R}^d) \ni (x,\mu) \mapsto \partial_x\partial_\mu \pig[S W_2^\sigma(\mu, \nu)\pigr]^2(x)$ is jointly continuous.
  \end{enumerate}

	\end{lemma}

\begin{proof}
     Item (1) is obtained by directly differentiating (1) in Lemma \ref{estimate_smooth} and the integration by part. The interchange of the differentiation and the integration is justified in the proof of \cite[Lemma 2.4]{bayraktar_smooth_2023}. To prove (2), we use \eqref{optimal_transport_map} to obtain that
     \begin{align*}
         &\h{-10pt}\int_{\mathbb{R}^d}\left|\partial_x\partial_\mu \pig[S W_2^\sigma(\mu, \nu)\pigr]^2(x)\right|^2\mu(dx) \\
         =\,& \int_{\mathbb{R}^d}\left|\int_{\mathcal{S}^{d-1}}\theta\theta^\top d\theta +\int_{\mathbb{R}}\int_{\mathcal{S}^{d-1}}\left(\theta\theta^\top\frac{(\theta^\top x-y)}{\sigma^2}T_{\theta,\mu}^\sigma(y)\mathcal{N}_\sigma(\theta^\top x-y)\right)d\theta dy\right|^2\mu(dx)\\
         \leq\, &2\left|\int_{\mathcal{S}^{d-1}}\theta\theta^\top d\theta \right|^2
         +2 \int_{\mathbb{R}^d}\left|\int_{\mathbb{R}}\int_{\mathcal{S}^{d-1}}\left(\theta\theta^\top\frac{(\theta^\top x-y)}{\sigma^2}T_{\theta,\mu}^\sigma(y)\mathcal{N}_\sigma(\theta^\top x-y)\right)d\theta dy\right|^2\mu(dx)\\
         \leq\,& C_d + \frac{2}{\sigma^4} \int_{\mathbb{R}^d}\left(\int_{\mathbb{R}\times\mathcal{S}^{d-1}}\!\! (\theta^\top x-y)^2 \mathcal{N}_\sigma(\theta^\top x-y)d\theta dy\right) \!\!\left(\int_{\mathbb{R}\times \mathcal{S}^{d-1}}\!\! \pig[T_{\theta,\mu}^\sigma(y)\pigr]^2\mathcal{N}_\sigma(\theta^\top x-y)d\theta dy\right)\mu(dx)
         \\
         \leq\,&C_d\left(1+\frac{1}{\sigma^2}\int_{\mathbb{R}^d}|x|^2\nu^\sigma (dx) \right).
     \end{align*}
    To prove (3), we let $\{x_n\}_{n\in \mathbb{N}}\subset \mathbb{R}^d$ and $\{\mu_n\}_{n\in \mathbb{N}} \subset \mathcal{P}_2(\mathbb{R}^d)$ such that $|x_n -x|\to 0$ and $\mathcal{W}_2(\mu_n, \mu) \to 0$ as $n \to \infty$. It is sufficient to consider
    \begin{align*}
    &\h{-10pt}\int_{\mathbb{R}}\int_{\mathcal{S}^{d-1}}
    \theta\theta^\top\left(\frac{(\theta^\top x_n-y)}{\sigma^2}T_{\theta,\mu_n}^\sigma(y)\mathcal{N}_\sigma(\theta^\top x_n-y)
    -\frac{(\theta^\top x-y)}{\sigma^2}T_{\theta,\mu}^\sigma(y)\mathcal{N}_\sigma(\theta^\top x-y)\right)d\theta dy\\
    =\,& \int_{\mathbb{R}}\int_{\mathcal{S}^{d-1}}
    \theta\theta^\top
    \frac{(\theta^\top x_n-\theta^\top x)}{\sigma^2}
    T_{\theta,\mu_n}^\sigma(y)
    \mathcal{N}_\sigma(\theta^\top x_n-y)d\theta dy\\
    &+\int_{\mathbb{R}}\int_{\mathcal{S}^{d-1}}
    \theta\theta^\top \frac{(\theta^\top x-y)}{\sigma^2}\pig[T_{\theta,\mu_n}^\sigma(y)
    -T_{\theta,\mu}^\sigma(y)\pig]
    \mathcal{N}_\sigma(\theta^\top x_n-y)d\theta dy\\
    &+\int_{\mathbb{R}}\int_{\mathcal{S}^{d-1}}
    \theta\theta^\top\frac{(\theta^\top x-y)}{\sigma^2}
    T_{\theta,\mu}^\sigma(y)
    \pig[\mathcal{N}_\sigma(\theta^\top x_n-y)
    -\mathcal{N}_\sigma(\theta^\top x-y)
    \pig]d\theta dy.
    \end{align*}
    By Lemma \ref{dominator}, the pointwise convergence of $T^\sigma_{\theta,\mu_n}(y)$ and the Lebesgue dominated convergence theorem, we conclude that the above term converges to zero as $n \to \infty$, which implies the desired continuity.
 \end{proof}

\noindent We have the following estimate when acting $\mathcal{H}$ on the metric:
\begin{lemma}
\label{gauge_H}
    Let $\nu \in \mathcal{P}_2(\mathbb{R}^d)$ and $\sigma > 0$ be fixed. We have 
    \begin{align*}
        \Big|\mathcal{H}\pig[S W_2^\sigma(\mu, \nu)\pigr]^2\Big|\leq C_d.
    \end{align*}
\end{lemma}
\begin{proof}
    For ease of notation, let $h$ denote the map:
	\begin{align*}
		h:\mathcal{P}_2(\mathbb{R}^d)\to \mathbb{R}\,;\,
		\mu \longmapsto h(\mu) := \pig[S W_2^\sigma(\mu, \nu)\pigr]^2.
	\end{align*}
 For each $X\in L^2(\Omega,\mathcal{F},\mathbb{P};\mathbb{R}^d)$ such that $\mu=\mathcal{L}(X)$, we denote $T_{\theta,X}^\sigma$ to be the map $T_{\theta,X}^\sigma(x):=F_{\nu_\theta^\sigma}^{-1}\pig(F_{\mu_\theta^\sigma}(x)\pig)$. For $\mu\in\mathcal{P}_2(\mathbb{R}^d)$ such that $\mu=\mathcal{L}(X)$, we also denote $h^*$ to be the lifting of $h$ to Hilbert space, then $h^*(X+w) = h((I_d+w)_\sharp \mu)$, where $I_d+w:\mathbb{R}^d\to\mathbb{R}^d$ is the map $x\mapsto x+w$ for $w \in \mathbb{R}^d$, and
    \begin{align*}
       \frac{d}{dw} h^*(X+w) = \mathbb{E}[\partial_\mu h(\mathcal{L}(X+w))(X+w)],
    \end{align*}
    where $\dfrac{d}{dw}$ is the usual gradient operator on $\mathbb{R}^d$. Therefore, as the laws of $T^{\sigma}_{\theta,X+w}\pig(\theta^\top(X+w+\bar{N}_\sigma)\pig)$  and $T^{\sigma}_{\theta,X}\pig(\theta^\top(X+\bar{N}_\sigma)\pig)$ are the same by \eqref{optimal_transport_map}, we have
    \begin{align*}
        &\h{-10pt}\mathbb{E}[\partial_\mu h (\mathcal{L}(X+w))(X+w)] - \mathbb{E}[\partial_\mu h (\mathcal{L}(X))(X)]\\
        =&\mathbb{E}\int_{\mathcal{S}^{d-1}}\theta\left\{\theta^\top (X+w) - \bar{\mathbb{E}}\Big[T^{\sigma}_{\theta,X+w}\pig(\theta^\top(X+w+\bar{N}_\sigma)\pig)\Big]\right\}d\theta\\
     &-\mathbb{E}\int_{\mathcal{S}^{d-1}}\theta\left\{\theta^\top X - \bar{\mathbb{E}}\Big[T^{\sigma}_{\theta,X}\pig(\theta^\top(X+\bar{N}_\sigma)\pig)\Big]\right\}d\theta\\
     =&\int_{\mathcal{S}^{d-1}}\theta\theta^\top w d\theta,
    \end{align*}
    and thus 
    \begin{align*}
        \mathcal{H}h = \int_{\mathcal{S}^{d-1}}\theta\theta^\top d\theta,
    \end{align*}
    from which $|\mathcal{H}h|\leq C_d$ is obvious.
\end{proof}

 \noindent For each $\sigma >0$, we define $\rho_\sigma:([0,T]\times\mathcal{P}_2(\mathbb{R}^d))^2 \to \mathbb{R}$ by
	\begin{align}
		\rho_\sigma((s,\mu),(t,\nu)) := |t-s|^2 + \pig[SW_2^\sigma(\mu,\nu)\pigr]^2.
		\label{def. rho_1/d}
	\end{align}
	It is a gauge type function on ${\color{black}([0,T]\times\mathcal{P}_2(\mathbb{R}^d),(\rho_\sigma)^{1/2})}$ by Definition \ref{def. gauge}. We conclude the following smooth variational principle: 
	\begin{theorem}
        \label{regularity_metric}
		Fix $\delta>0$ and let $G:[0, T] \times \mathcal{P}_2(\mathbb{R}^d) \rightarrow \mathbb{R}$ be upper semicontinuous and bounded from above. If there exists $\lambda>0$ and $\left(t_0, \mu_0\right) \in[0, T] \times \mathcal{P}_2(\mathbb{R}^d)$ such that 
		\begin{align*}
			\displaystyle\sup _{(t, \mu) \in[0, T] \times \mathcal{P}_2(\mathbb{R}^d)} G(t, \mu)-\lambda \leq G\left(t_0, \mu_0\right),
		\end{align*} 
		then there exists $(\tilde{t}, \tilde{\mu}) \in[0, T] \times \mathcal{P}_2(\mathbb{R}^d)$ and a sequence $\left\{\left(t_n, \mu_n\right)\right\}_{n \in \mathbb{N}} \subset[0, T] \times \mathcal{P}_2(\mathbb{R}^d)$ such that:
		\begin{enumerate}[(1).]
			\item  $\rho_{1 / \delta}\left((\tilde{t}, \tilde{\mu}),\left(t_n, \mu_n\right)\right) \leq \dfrac{\lambda}{2^n \delta^2}$ for any $n=0,1,2,\ldots$;
			\item $G\left(t_0, \mu_0\right) \leq G(\tilde{t}, \tilde{\mu})-\delta^2 \varphi_\delta(\tilde{t}, \tilde{\mu})$ with $\varphi_\delta:[0, T] \times \mathcal{P}_2(\mathbb{R}^d) \rightarrow[0,+\infty)$ given by
			\begin{align}
				\label{varphi_and_rho}
				\varphi_\delta(t, \mu):=\sum_{n=0}^{+\infty} \frac{1}{2^n} \rho_{1 / \delta}\left((t, \mu),\left(t_n, \mu_n\right)\right), \quad \text{for any } (t, \mu) \in[0, T] \times \mathcal{P}_2(\mathbb{R}^d);
			\end{align}
			\item $G(t, \mu)-\delta^2 \varphi_\delta(t, \mu)<G(\tilde{t}, \tilde{\mu})-\delta^2 \varphi_\delta(\tilde{t}, \tilde{\mu})$, for every $(t, \mu) \in\left([0, T] \times \mathcal{P}_2(\mathbb{R}^d)\right) \backslash\{(\tilde{t}, \tilde{\mu})\}$;
   \item \label{bound_varphi} $\varphi_\delta\in PC_{1}^{1,2}([0,T]\times\mathcal{P}_2(\mathbb{R}^d))$ and there exists a constant $C_d>0$ depending only on $d$ such that for any $(t, \mu) \in[0, T] \times \mathcal{P}_2(\mathbb{R}^d)$, we have
		\begin{align*}
			|\p_t\varphi_\delta(t,\mu)| &\leq 4 T;
			\\
			\int_{\mathbb{R}^d}\left|\p_\mu \varphi_\delta(t, \mu)(x)\right|^2 \mu(d x) &\leq C_d\left(\frac{1+\lambda}{\delta^2}+\int_{\mathbb{R}^d}|x|^2 \mu(d x)
			+\int_{\mathbb{R}^d}|x|^2 \mu_0(d x)
			\right) ;
			\\
			\int_{\mathbb{R}^d}\left|\p_x\p_{\mu} \varphi_\delta(t, \mu)(x)\right|^2 \mu(d x)&\leq C_d\delta^2\Big(\frac{1+\lambda}{\delta^2}+\int_{\mathbb{R}^d}|x|^2\mu_0 (dx)\Big);\\
   	|\mathcal{H}\varphi_\delta(t,\mu)|&\leq C_d.
		\end{align*}
		\end{enumerate}

		\label{thm. bdd of varphi_d}
	\end{theorem}
	\begin{proof}
For items (1) to (3), we refer readers to \cite[Theorem 2.5.2]{technique_variational_principle}. As there is a constant $\kappa_d>0$ depending only on $d$ such that $\int_{\mathcal{S}^{d-1}}\theta\theta^\top d\theta=\kappa_d I_d$, by a simple application of the triangle inequality we conclude that 
	\begin{align}
			\kappa_d\int_{\mathbb{R}^d} |x|^2 \nu(dx) 
   = \int_{\mathcal{S}^{d-1}}  \int_{\mathbb{R}} |y|^2 \nu_\theta(dy) d\theta
   \leq C_d\left(\int_{\mathbb{R}^d} |x|^2 \mu(dx) + \pig[SW_2(\mu,\nu)\pigr]^2\right).
   \label{878}
		\end{align}
The constant $C_d$ changes from line to line in this proof, but it still depends only on $d$. Replacing $\nu$ and $\mu$ with $\mu_n^{1/\delta}$ and $\mu^{1/\delta}_0$ respectively in the above, we have
		\begin{align}
			\int_{\mathbb{R}^d} |x|^2 \mu_n^{1/\delta}(dx)
   &\leq C_d\left(\int_{\mathbb{R}^d} |x|^2 \mu^{1/\delta}_0(dx) 
   +\pig[SW_2(\mu^{1/\delta}_n,\tilde{\mu}^{1/\delta})\pigr]^2
   +\pig[SW_2(\tilde{\mu}^{1/\delta},\mu^{1/\delta}_0)\pigr]^2
   \right)\nonumber\\
   &\leq C_d\left(\int_{\mathbb{R}^d} |x|^2 \mu_0(dx)
   +\int_{\mathbb{R}^d} |y|^2 \mathcal{N}_{1/\delta}(dy)
   + \frac{\lambda}{2^n\delta^2}
   + \frac{\lambda}{\delta^2}\right)\nonumber\\
   &\leq C_d\left(\int_{\mathbb{R}^d} |x|^2 \mu_0(dx) 
   + \frac{1+\lambda}{\delta^2}
   +\frac{\lambda}{2^n\delta^2}\right).
   \label{870}
		\end{align}
  From Lemmas \ref{estimate_smooth}-\ref{gauge_H}, we see that the derivatives $\p_t \rho_{1/\delta}\left((t,\mu),(t_n,\mu_n)\right)$, $\p_\mu \rho_{1/\delta}\left((t,\mu),(t_n,\mu_n)\right)(x)$, $\p_x\p_\mu \rho_{1/\delta}\left((t,\mu),(t_n,\mu_n)\right)(x)$ and $\mathcal{H}  \rho_{1/\delta}\left((t,\mu),(t_n,\mu_n)\right)$ exist for each $n=0,1,2,\ldots$ and also
  \begin{align}
      |\p_t \rho_{1/\delta}((t,\mu),(t_n,\mu_n))|=2t \h{5pt} \text{and}\h{5pt} \left|\mathcal{H}\pig(\rho_{1/\delta}((t,\mu),(t_n,\mu_n))\pig)\right|
			\leq C_{d}.
   \label{887}
  \end{align}
  The interchange of the differentiation and the infinite sum for $\p_t \varphi_{1/\delta}(t,\mu)$ and $\mathcal{H}\varphi_{1/\delta}(t,\mu)$ are obvious by the above bounds and Lemma \ref{dominator}. Therefore, $\p_t \varphi_{1/\delta}(t,\mu)$ and $\mathcal{H}\varphi_{1/\delta}(t,\mu)$ exist and have the desired upper bounds by using the definition of $\varphi_{1/\delta}(t,\mu)$ and \eqref{887}. Moreover, we are safe to interchange the measure differentiation and the infinite sum for $\p_\mu \varphi_{1/\delta}(t,\mu)(x)$ by the proof of \cite[Proposition 2.5]{bayraktar_smooth_2023}, thus the derivative $\p_\mu \varphi_{1/\delta}(t,\mu)(x)$ also exists. According to Lemmas \ref{dominator} and \ref{estimate_smooth 2}, we see that the derivative $\p_x\p_\mu \varphi_{1/\delta}(t,\mu)(x)$ also exists by interchanging the differentiation and the infinite sum. The joint continuity of these derivatives follows from the dominating function constructed in Lemma \ref{dominator}, (3) in Lemma \ref{estimate_smooth}, (3) in Lemma \ref{estimate_smooth 2}, Lemma \ref{gauge_H}, the definition of $\varphi_{1/\delta}(t,\mu)$ and the dominated convergence theorem. Finally, from Lemmas \ref{estimate_smooth}-\ref{gauge_H} and \eqref{870}, we have 
		\begin{align*}
			\int_{\mathbb{R}^d} |\partial_\mu\rho_{1/\delta}((t,\mu),(t_n,\mu_n))(x)|^2\mu(dx)\leq\,
   &C_d\Big(\frac{1+\lambda}{\delta^2}+\frac{\lambda}{2^n\delta^2}+\int_{\mathbb{R}^d}|x|^2 \mu(d x)+\int_{\mathbb{R}^d} |x|^2\mu_0(dx)\Big),\\
			\int_{\mathbb{R}^d} |\partial_x\partial_\mu\rho_{1/\delta}((t,\mu),(t_n,\mu_n))(x)|^2\mu(dx)\leq\, &C_d\Big(1+\lambda+\dfrac{\lambda}{2^n}+\delta^2\int_{\mathbb{R}^d}|x|^2\mu_0(dx)\Big).
		\end{align*}
		By the above bounds, \eqref{887} and the definition of $\varphi_{1/\delta}(t,\mu)$, we conclude the desired bounds.
	\end{proof}

 \section{Smooth Finite-dimensional Approximations of the Value
		Function}
This section is to carry out the smooth finite-dimensional approximations of the value function through the $n$-particles approximation of the measure argument and the mollification of the coefficient functions $b$, $f$ and $g$. 
	\subsection{Infinite-dimensional Approximation}
	Fix a complete probability space $(\check{\Omega},\check{\mathcal{F}},\check{\mathbb{P}})$, also of the form $(\check{\Omega}^0\times\check{\Omega}^1,\check{\mathcal{F}}^0\otimes \check{\mathcal{F}}^1,\check{\mathbb{P}}^0\otimes \check{\mathbb{P}}^1)$. The space $(\check{\Omega}^0,\check{\mathcal{F}}^0,\check{\mathbb{P}}^0)$ supports a $d$-dimensional Brownian motion $\check{W}^0$. For $(\check{\Omega}^1,\check{\mathcal{F}}^1,\check{\mathbb{P}}^1)$, it is of the form $(\check{\tilde{\Omega}}^1\times \check{\hat{\Omega}}^1,\check{\mathcal{G}}\otimes\check{\hat{\mathcal{F}}}^1,\check{\tilde{\mathbb{P}}}^1\otimes\check{\hat{\mathbb{P}}}^1)$. On $(\check{\hat{\Omega}}^1,\check{\hat{\mathcal{F}}}^1,\check{\hat{\mathbb{P}}}^1)$, there lives $d$-dimensional Brownian motions $\check{W}$ and $\check{B}$. The space $(\check{\tilde{\Omega}}^1,\check{\mathcal{G}},\check{\tilde{\mathbb{P}}}^1)$ is where the initial random variables live. We assume that $(\check{\tilde{\Omega}}^1,\check{\mathcal{G}},\check{\tilde{\mathbb{P}}}^1)$ is rich enough to support all probability laws in $\mathbb{R}^d$, i.e., for any probability law $\mu$ in $\mathbb{R}^d$, there exists $X\in\check{\tilde{\Omega}}^1$ such that $\mathcal{L}(X) = \mu$. The expectation $\mathbb{E}$ is taken with respect to $\check{\mathbb{P}}^0\otimes \check{\mathbb{P}}^1$.\\
	\hfill\\
	Set $\check{\mathbb{F}} = (\check{\mathcal{F}}_s)_{s\geq 0}:=\left(\sigma(\check{W}^0_{r})_{0\leq r\leq s}\vee\sigma(\check{W}_{r})_{0\leq r\leq s}\vee\sigma(\check{B}_{r})_{0\leq r\leq s}\vee\check{\mathcal{G}}\right)_{s\geq 0}$, $\check{\mathbb{F}}^t = (\mathcal{F}_s^t)_{s\geq 0}:=\big(\sigma(\check{W}_r^0)_{0\leq r\leq s}\vee\sigma(\check{W}_{r\vee t}-\check{W}_t)_{0\leq r\leq s}\vee\sigma(\check{B}_{r\vee t}-\check{B}_t)_{0\leq r\leq s}\vee\check{\mathcal{G}}\big)_{s\geq 0}$. Without loss of generality, we assume that they are $\check{\mathbb{P}}$-complete. Let $t>0$, denote $\check{\mathcal{A}}$ (resp. $\check{\mathcal{A}}_t$) the set of $\check{\mathbb{F}}$-progressively measurable process (resp. $\check{\mathbb{F}}^{t}$-progressively measurable process) $\check{\alpha}$ valued in $A$.\\
	\hfill\\
	Letting $\varepsilon>0$, $t\in [0,T)$, $\check{\xi}\in L^2(\check{\Omega}^1,\check{\mathcal{F}}^1,\check{\mathbb{P}}^1;\mathbb{R}^d)$ and $\check{\alpha}\in \check{\mathcal{A}}$, we consider the unique solution $\check{X}^{\varepsilon,t,\check{\xi},\check{\alpha}}=(\check{X}^{\varepsilon,t,\check{\xi},\check{\alpha}}_s)_{s\in [t,T]}$ of the system of the perturbed equation:
	\begin{align}
		X_s =\,& \check{\xi} + \int_{t}^s b(r,X_{r},\mathbb{P}_{X_{r}}^{\check{W}^0},\check{\alpha}_r)dr + \int_t^s \sigma(r,X_{r},\check{\alpha}_r)d\check{W}_r + \int_t^s \sigma^0(r)d\check{W}^0_r\nonumber\\
		&+\varepsilon(\check{B}_s - \check{B}_t).
		\label{eq. perturbed, mfc}
	\end{align}
	For any $t \in [0,T]$ and $\check{\xi}\in L^2(\check{\Omega}^1,\check{\mathcal{F}}^1,\check{\mathbb{P}}^1;\mathbb{R}^d)$, we  consider the value function:
	\begin{align}
		V_{\varepsilon}(t,\check{\xi})
		=\sup_{\check{\alpha}\in \check{\mathcal{A}}_t}J_{\varepsilon}(t,\check{\xi},\check{\alpha})
		:=\sup_{\check{\alpha}\in \check{\mathcal{A}}_t}
		\mathbb{E}\Bigg[\int_t^T f\Big(s,\check{X}^{\varepsilon,t,\check{\xi},\check{\alpha}}_s,
		\mathbb{P}^{\check{W}^0}_{\check{X}^{\varepsilon,t,\check{\xi},\check{\alpha}}_s},\check{\alpha}_s\Big)ds + g\Big(\check{X}^{\varepsilon,t,\check{\xi},\check{\alpha}}_T,
		\mathbb{P}^{\check{W}^0}_{\check{X}^{\varepsilon,t,\check{\xi},\check{\alpha}}_T}\Big)\Bigg].
		\label{1368}
	\end{align}
	By the law invariance property, we can define a function $v_{\varepsilon}(t,\mu):[0,T]\times\mathcal{P}_2(\mathbb{R}^d) \to \mathbb{R}$ such that
	\begin{align}
		v_{\varepsilon}(t,\mu) := V_{\varepsilon}(t,\check{\xi}),
		\label{def. v_e=V_e mfc}
	\end{align}
	for any $\check{\xi}\in L^2(\check{\Omega}^1,\check{\mathcal{F}}^1,\check{\mathbb{P}}^1;\mathbb{R}^d)$ such that $\mathcal{L}(\check{\xi})=\mu$.
	\begin{lemma}
	Suppose that Assumption (A) holds. There exists a constant $C_5=C_5(d,K,T)$ such that for any $\varepsilon \geq 0$ and $(t,\mu) \in [0,T] \times \mathcal{P}_2(\mathbb{R}^d)$, it holds that $|v_{\varepsilon}(t,\mu)-v_0(t,\mu)| \leq C_5\varepsilon.$
		\label{lem. |v_e-v_0|<C_5 e}
	\end{lemma}
	\noindent The proof is standard by using Assumption (A), equation \eqref{eq. perturbed, mfc} and the definition in \eqref{1368}. We omit it here.
	\subsection{Finite-dimensional Approximation}\label{sec. fin-d app}
	In this section, we will illustrate the finite-dimensional approximation. Consider a complete probability space $(\overline{\Omega},\overline{\mathcal{F}},\overline{\mathbb{P}})$, it is also of the form $(\overline{\Omega}^0\times\overline{\Omega}^1,\overline{\mathcal{F}}^0\otimes \overline{\mathcal{F}}^1,\overline{\mathbb{P}}^0\otimes \overline{\mathbb{P}}^1)$. The space $(\overline{\Omega}^0,\overline{\mathcal{F}}^0,\overline{\mathbb{P}}^0)$ supports a $d$-dimensional Brownian motion $\overline{W}^0$. For $(\overline{\Omega}^1,\overline{\mathcal{F}}^1,\overline{\mathbb{P}}^1)$, it is of the form $(\overline{\tilde{\Omega}}^1\times \overline{\hat{\Omega}}^1,\overline{\mathcal{G}}\otimes\overline{\hat{\mathcal{F}}}^1,\overline{\tilde{\mathbb{P}}}^1\otimes\overline{\hat{\mathbb{P}}}^1)$. Let $n \in \mathbb{N}$, there lives $d$-dimensional Brownian motions $\overline{W}^1,\ldots,\overline{W}^n$ and $\overline{B}^1,\ldots,\overline{B}^n$ on $(\overline{\hat{\Omega}}^1,\overline{\hat{\mathcal{F}}}^1,\overline{\hat{\mathbb{P}}}^1)$. We require $\overline{W}^{1},\ldots,\overline{W}^n,\overline{B}^1,\ldots,\overline{B}^n$ to be mutually independent. The space $(\overline{\tilde{\Omega}}^1,\overline{\mathcal{G}},\overline{\tilde{\mathbb{P}}}^1)$ is where the initial random variables live. We assume that $(\overline{\tilde{\Omega}}^1,\overline{\mathcal{G}},\overline{\tilde{\mathbb{P}}}^1)$ is rich enough to support all probability laws in $\mathbb{R}^d$, i.e., for any probability law $\mu$ in $\mathbb{R}^d$, there exists $X\in\overline{\tilde{\Omega}}^1$ such that $\mathcal{L}(X) = \mu$. \\
	\hfill\\
	Put $\overline{\mathbb{F}} = (\overline{\mathcal{F}}_s)_{s\geq 0}:=\left(\sigma(\overline{W}^0_{r})_{0\leq r\leq s}\vee\sigma(\overline{W}_{r}^i)_{0\leq r\leq s,i=1,\ldots,n}\vee\sigma(\overline{B}_{r}^i)_{0\leq r\leq s,i=1,\ldots,n}\vee\overline{\mathcal{G}}\right)_{s\geq 0}$, $\overline{\mathbb{F}}^t = (\mathcal{F}_s^t)_{s\geq 0}:=\left(\sigma(\overline{W}_r^0)_{0\leq r\leq s}\vee\sigma(\overline{W}_{r\vee t}^i-\overline{W}^i_t)_{0\leq r\leq s,i=1,\ldots,n}\vee\sigma(\overline{B}^i_{r\vee t}-\overline{B}^i_t)_{0\leq r\leq s,i=1,\ldots,n}\vee\overline{\mathcal{G}}\right)_{s\geq 0}$. Without loss of generality, we assume that they are $\overline{\mathbb{P}}$-complete. Let $t>0$ and denote $\overline{\mathcal{A}}^n$ (resp. $\overline{\mathcal{A}}_t^n$) the set of $\overline{\mathbb{F}}$-progressively measurable process (resp. $\overline{\mathbb{F}}^{t}$-progressively measurable process) $\overline{\alpha} = (\overline{\alpha}^1,\ldots,\overline{\alpha}^n)$ valued in $A^n$.\\
	\hfill\\
	 Let $\overline{\xi}=(\overline{\xi}^1,\ldots,\overline{\xi}^n) \in \mathbb{R}^{dn}$, we consider the solution $\overline{X}^{\varepsilon,t,\overline{\xi},\overline{\alpha}}_s=(\overline{X}^{1,\varepsilon,t,\overline{\xi},\overline{\alpha}}_s,\ldots,\overline{X}^{n,\varepsilon,t,\overline{\xi},\overline{\alpha}}_s) \in \mathbb{R}^{dn}$ of the system of the perturbed equation:
	\begin{align}
		X_s^i =\,& \overline{\xi}^i + \int_{t}^s b(r,X_{r}^i,\widehat{\mu}^n_r,\overline{\alpha}_r^i)dr + \int_t^s \sigma(r,X_{r}^i,\overline{\alpha}_r^i)d\overline{W}_r^i + \int_t^s \sigma^0(r)d\overline{W}^0_r\nonumber\\
		&+\varepsilon(\overline{B}^i_s - \overline{B}^i_t),
		\label{eq. state perturbed by e BM}
	\end{align}
	for each $i=1,2,\ldots,n$ and $X_s=(X^1_s,\ldots,X^n_s)$ with 
	\begin{align*}
		\widehat{\mu}^n_r=\dfrac{1}{n}\sum^n_{i=1}\delta_{\overline{X}^i_{r}}.
	\end{align*}
	Denoting $\widehat{\mu}^{n,\varepsilon,t,\overline{\xi},\overline{\alpha}}_s:=\dfrac{1}{n}\displaystyle\sum^n_{i=1}
	\delta_{\overline{X}^{i,\varepsilon,t,\overline{\xi},\overline{\alpha}}_s}$, we consider the value function:
	\begin{align*}
		\h{-10pt}\widetilde{v}_{\varepsilon,n}(t,\overline{\mu})&:=\sup_{\overline{\alpha}\in \overline{\mathcal{A}}^n_t}J_{\varepsilon,n}(t,\overline{\xi},\overline{\alpha})
		\nonumber\\&:=
		\sup_{\overline{\alpha}\in \overline{\mathcal{A}}^n_t} \dfrac{1}{n}\sum^n_{i=1}\mathbb{E}\Bigg[\int_t^T f(s,\overline{X}^{i,\varepsilon,t,\overline{\xi},\overline{\alpha}}_s,
		\widehat{\mu}^{\,n,\varepsilon,t,\overline{\xi},\overline{\alpha}}_s,\overline{\alpha}^i_s)ds + g(\overline{X}^{i,\varepsilon,t,\overline{\xi},\overline{\alpha}}_T,\widehat{\mu}^{\,n,\varepsilon,t,\overline{\xi},\overline{\alpha}}_T)\Bigg],
	\end{align*}
	for any $t \in [0,T]$ and $\overline{\mu} \in \mathcal{P}_2(\mathbb{R}^{dn})$ such that $\mathcal{L}({\overline{\xi}})=\overline{\mu}$.\\
	\hfill\\
	Now we introduce the smooth approximation of the coefficients. For $n,m \in \mathbb{N}$ and  $i=1,\ldots,n$, we define $b^i_{n,m}:[0,T]\times\mathbb{R}^{dn}\times A \to \mathbb{R}^d$, 
	$f^i_{n,m}:[0,T]\times\mathbb{R}^{dn}\times A \to \mathbb{R}^d$, 
	$g^i_{n,m}:\mathbb{R}^{dn}\to \mathbb{R}^d$ be the smooth approximation of $b$,$f$ and $g$ respectively such that
	\begin{align}
		b^i_{n,m}(t,\overline{x},a)&:=m^{dn+1}\int_{\mathbb{R}^{dn+1}}
		b\left( (t-s)^+\wedge T,x^i-y^i,\dfrac{1}{n}\sum^n_{j=1}\delta_{x^j-y^j},a\right)\phi(ms)\prod^n_{k=1}\Phi(my^k)dy^kds;\label{def. approx of b}\\
		f^i_{n,m}(t,\overline{x},a)&:=m^{dn+1}\int_{\mathbb{R}^{dn+1}}
		f\left( (t-s)^+\wedge T,x^i-y^i,\dfrac{1}{n}\sum^n_{j=1}\delta_{x^j-y^j},a\right)\phi(ms)\prod^n_{k=1}\Phi(my^k)dy^kds;\label{def. approx of f}\\
		g^i_{n,m}(\overline{x})&:=m^{dn}\int_{\mathbb{R}^{dn}}
		g\left( x^i-y^i,\dfrac{1}{n}\sum^n_{j=1}\delta_{x^j-y^j}\right)\prod^n_{k=1}\Phi(my^k)dy^k,
  \label{def. approx of g}
	\end{align}
	where $t \in [0,T]$, $a \in A$ and $\overline{x}=(x^1,\ldots,x^n)$. The functions $\phi:\mathbb{R}\to\mathbb{R}^+$, $\Phi:\mathbb{R}^d \to \mathbb{R}^+$ are two compactly supported smooth functions satisfying $\int_{\mathbb{R}}\phi(s)ds=1$ and $\int_{\mathbb{R}^{d}}\Phi(y)dy=1$. First, we establish some basic properties of these approximation functions. 
	\begin{lemma}
		Suppose that {\color{black}Assumption (A*)} holds and define
		\begin{align*}
		\widehat{\mu}^{n,\overline{x}}:=\dfrac{1}{n}\displaystyle\sum^n_{j=1} \delta_{x^j}.
		\end{align*}
		For any $i=1,2,\ldots,n$, we have the following 
		\begin{enumerate}[(1).]
			\item {\bf Upper Bound:}
			for any $n,m \in \mathbb{N}$ and $(t,\overline{x},a) \in [0,T] \times \mathbb{R}^{dn}\times A$, we have
			\begin{align*}
				|b^i_{n,m}(t,\overline{x},a)|
				\vee
				|f^i_{n,m}(t,\overline{x},a)|
				\vee|g^i_{n,m}(\overline{x})|
				\leq K;
			\end{align*}
			\item {\bf Convergence:} it holds that
			\begin{align*}
				&\lim_{m\to\infty} b^i_{n,m}(t,\overline{x},a)
				=b(t,x^i,\widehat{\mu}^{n,\overline{x}},a),\h{10pt}
				\lim_{m\to\infty} f^i_{n,m}(t,\overline{x},a)
				=f(t,x^i,\widehat{\mu}^{n,\overline{x}},a)\nonumber\\
				&\lim_{m\to\infty} g^i_{n,m}(\overline{x})
				=g(x^i,\widehat{\mu}^{n,\overline{x}}),
			\end{align*}
			uniformly in $(t,\overline{x},a)\in [0,T] \times \mathbb{R}^{dn} \times A$;
			\item {\bf Convergence Rate:} for any $n,m \in \mathbb{N}$, $(t,\overline{x},a) \in [0,T] \times \mathbb{R}^{dn}\times A $, we have the estimates
			\begin{enumerate}[(i).]
				\item for $b$ and $f$:
				\begin{align*}
					&\h{-10pt}|b^i_{n,m}(t,\overline{x},a)
					-b(t,x^i,\widehat{\mu}^{n,\overline{x}},a)|\vee
					|f^i_{n,m}(t,\overline{x},a)
					-f(t,x^i,\widehat{\mu}^{n,\overline{x}},a)|\nonumber\\
					\leq\, & Km\int_{\mathbb{R}}\left|t-\pig[T\wedge(t-s)^+\pig]\right|^\beta \phi(ms)ds
					+Km^{dn}\int_{\mathbb{R}^{dn}}
					\left(|y^i|+\dfrac{1}{n}\sum^n_{j=1}|y^j|\right) \prod^n_{k=1}\Phi(my^k)dy^k;
				\end{align*}
				\item for $g$: 
				\begin{align*}
					|g^i_{n,m}(\overline{x})
					-g(x^i,\widehat{\mu}^{n,\overline{x}})|
					&\leq Km^{dn}\int_{\mathbb{R}^{dn}}
					\left(|y^i|+\dfrac{1}{n}\sum^n_{j=1}|y^j|\right) \prod^n_{k=1}\Phi(my^k)dy^k;
				\end{align*}
			\end{enumerate}
			\item {\bf Continuity:} for any $n,m \in \mathbb{N}$, $(t,\overline{x},\overline{x}',a) \in [0,T] \times \mathbb{R}^{dn}\times \mathbb{R}^{dn}\times A$, we have the estimate
			\begin{align}
				&|b^i_{n,m}(t,\overline{x},a)
				-b^i_{n,m}(t,\overline{x}',a)|
				\vee
				|f^i_{n,m}(t,\overline{x},a)
				-f^i_{n,m}(t,\overline{x}',a)|
				\vee
				|g^i_{n,m}(\overline{x})
				-g^i_{n,m}(\overline{x}')|\nonumber\\
				&\leq K\left[|x^i-x'^i|+\dfrac{1}{n}\sum^n_{j=1}\left|x^j
				-x'^j\right|\right].
    \label{ineq. lip cts of b^i_n,m...}
			\end{align}
		\end{enumerate}
		\label{lem estimate of b^i_n,m...}
	\end{lemma}

{\color{black}
	\begin{proof} Assertion (1) is obvious by the definitions in \eqref{def. approx of b}-\eqref{def. approx of g} and Assumption (A*). We prove assertion (3) for $g$ by considering
\begin{align*}
\pig|g(x^i, \widehat{\mu}^{n,\overline{x}}) - g^i_{n,m} (\overline{x})\pig|
\leq m^{dn} \int_{\mathbb{R}^{dn}} \bigg| g(x^i, \widehat{\mu}^{n,\overline{x}}) 
- g \bigg( x^i -y^i, \frac{1}{n} \sum_{j=1}^{n} \delta_{x^j-y^j} \bigg) \bigg| \prod^n_{k=1}\Phi(my^k) dy^k.
\end{align*}
Using the Lipschitz continuity of $g$ in Assumption (A*) and the fact that
\begin{align}
\mathcal{W}_1 \bigg( \widehat{\mu}^{n,\overline{x}},  \frac{1}{n} \sum_{j=1}^{n} \delta_{x^j-y^j} \bigg) 
&= \mathcal{W}_1  \bigg( \frac{1}{n} \sum_{j=1}^{n} \delta_{x^j}, \frac{1}{n} \sum_{j=1}^{n} \delta_{x^j - y^j} \bigg) \nonumber\\
&\leq \int_{\mathbb{R}^d\times \mathbb{R}^d} |x-y|
\left[\frac{1}{n} \sum_{j=1}^{n} \delta_{(x^j,x^j - y^j)}(dx,dy)\right] \nonumber\\
&= \frac{1}{n} \sum_{j=1}^{n} |y^j|,
\label{1701}
\end{align}
we obtain that
\begin{align*}
\pig|g(x^i, \widehat{\mu}^{n,\overline{x}}) - g^i_{n,m} (\overline{x})\pig|
\leq K m^{dn} \int_{\mathbb{R}^{dn}} \bigg( |y^i| + \frac{1}{n} \sum_{j=1}^{n} |y^j| \bigg) \prod_{k=1}^{n} \Phi(my^k) dy^k.
\end{align*}
We prove assertion (3) for $b$ by considering (the proof for $f$ is exactly the same)
\begin{align*}
&|b(t, x^i, \widehat{\mu}^{n,\overline{x}}, a) - b_{n,m}^i (t, \overline{x}, a)| \\
&\leq m^{dn+1} \int_{\mathbb{R}^{dn+1}} \left| b(t, x^i, \widehat{\mu}^{n,\overline{x}}, a) - b \bigg( T \wedge (t-s)^+, x^i - y^i, \frac{1}{n} \sum_{j=1}^{n} \delta_{x^j - y^j}, a \bigg) \right| \phi(ms) \prod_{k=1}^{n} \Phi(my^k) dy^k ds\\
&\leq m \int_{\mathbb{R}} \left| b(t, x^i, \widehat{\mu}^{n,\overline{x}}, a) 
- b \left( T \wedge (t-s)^+, x^i, \widehat{\mu}^{n,\overline{x}}, a \right) \right| \phi(ms) ds \\
&\h{10pt}+ m^{dn+1} \int_{\mathbb{R}^{dn+1}} \left| b \left( T \wedge (t-s)^+, x^i, \widehat{\mu}^{n,\overline{x}}, a \right) - b \bigg( T \wedge (t-s)^+, x^i - y^i, \frac{1}{n} \sum_{j=1}^{n} \delta_{x^j - y^j}, a \bigg) \right|\cdot\\
&\h{335pt}\phi(ms) \prod_{j=1}^{n} \Phi(my^j) dy^j ds.
\end{align*}
The inequality in \eqref{1701} and Assumption (A*) implies that
\begin{align*}
&|b(t, x^i, \widehat{\mu}^{n,\overline{x}}, a) - b_{n,m}^i (t, \overline{x}, a)| \\
&\leq Km \int_{\mathbb{R}} \left|t- (T \wedge (t-s)^+) ) \right|^\beta \phi(ms) ds 
+ K m^{dn} \int_{\mathbb{R}^{dn}} \bigg( |y^i| + \frac{1}{n} \sum_{j=1}^{n} |y^j| \bigg) \prod_{k=1}^{n} \Phi(my^k) dy^k.
\end{align*}
Assertion (2) follows immediately from assertion (3). For the proof of assertion (4) for $g$ (the proofs for $f$ and $b$ are exactly the same), we let $\overline{x}$, $\overline{z} \in \mathbb{R}^{dn}$ and estimate
\begin{align*}
&|g_{n,m}^i (\overline{x}) - g_{n,m}^i (\overline{z})| \\
&\leq m^{dn} \int_{\mathbb{R}^{dn}} \bigg| g \bigg( x^i-y^i, \frac{1}{n} \sum_{j=1}^{n} \delta_{x^j-y^j} \bigg) 
- g \bigg( z^i-y^i, \frac{1}{n} \sum_{j=1}^{n} \delta_{z^j-y^j} \bigg) \bigg|  \prod_{k=1}^{n} \Phi(my^k) dy^k.
\end{align*}
Then the inequality in \eqref{1701}  yields that
\begin{align*}
|g_{n,m}^i (\overline{x}) - g_{n,m}^i (\overline{z})| 
&\leq K m^{dn} \int_{\mathbb{R}^{dn}}\bigg[ |x^i - z^i| + \frac{1}{n} \sum_{j=1}^{n} |x^j - z^j| \bigg]\prod_{k=1}^{n} \Phi(my^k) dy^k\\
&= K \bigg[ |x^i - z^i| + \frac{1}{n} \sum_{j=1}^{n} |x^j - z^j| \bigg].
\end{align*}
	\end{proof}
  \begin{remark}
  It is worth noting that the coefficients $b$, $f$ and $g$ are assumed to be $\mathcal{W}_1$-Lipschitz continuous in Assumption (A*), which is stronger than the $\mathcal{W}_2$-Lipschitz continuity assumed in \cite[Assumption (A)]{cosso_master_2022}. The main reason is to provide a better Lipschitz estimate when we establish the smooth approximation of the coefficients in Lemma \ref{lem estimate of b^i_n,m...}, which is critical in the proof of Lemma \ref{lem. classical sol. of smooth approx.} and Theorem \ref{thm compar}. We note that there are some gaps in the proofs of \cite[Theorem 5.1, Theorem A.7]{cosso_master_2022} if only $\mathcal{W}_2$-Lipschitz continuity is assumed in \cite[Assumption (A)]{cosso_master_2022}. More precisely, if we assume the coefficients are only  $\mathcal{W}_2$-Lipschitz continuous, then we only have 
\begin{align}
\mathcal{W}_2 \bigg( \widehat{\mu}^{n,\overline{x}},  \frac{1}{n} \sum_{j=1}^{n} \delta_{x^j-y^j} \bigg) 
= \mathcal{W}_2  \bigg( \frac{1}{n} \sum_{j=1}^{n} \delta_{x^j}, \frac{1}{n} \sum_{j=1}^{n} \delta_{x^j - y^j} \bigg) 
\leq \frac{1}{n^{1/2}} \sum_{j=1}^{n} |y^j|,
\label{1741}
\end{align}
instead of the estimate in \eqref{1701}. It seems that the corresponding estimate of $\mathcal{W}_2 \Big( \widehat{\mu}^{n,\overline{x}},  \frac{1}{n} \sum_{j=1}^{n} \delta_{x^j-y^j} \Big) $ in the proof of \cite[Lemma A.3]{cosso_master_2022} is not true, see the counterexample at the end of this remark and the estimate in \eqref{1787}. Hence, adopting the $\mathcal{W}_2$-Lipschitz continuity and the estimate in \eqref{1741}, assertion (3) of Lemma \ref{lem estimate of b^i_n,m...} becomes \begin{align}
				&|b^i_{n,m}(t,\overline{x},a)
				-b^i_{n,m}(t,\overline{x}',a)|
				\vee
				|f^i_{n,m}(t,\overline{x},a)
				-f^i_{n,m}(t,\overline{x}',a)|
				\vee
				|g^i_{n,m}(\overline{x})
				-g^i_{n,m}(\overline{x}')|\nonumber\\
				&\leq K\left[|x^i-x'^i|+\dfrac{1}{n^{1/2}}\sum^n_{j=1}\left|x^j
				-x'^j\right|\right],
    \label{1754}
			\end{align}
see the corresponding estimate in \cite[(A.7)]{cosso_master_2022}. This causes a problem when we apply Lemma \ref{lem estimate of b^i_n,m...} to prove Lemma \ref{lem. classical sol. of smooth approx.} and the bound of the first order derivative in \eqref{ineq.|D v_e,n,m| and |D^2 v_e,n,m|} (see also the corresponding estimate in \cite[(A.12)]{cosso_master_2022}), which is critical in the proof of Theorem \ref{thm compar}. While \cite[Assumption (A)]{cosso_master_2022} adopts the $\mathcal{W}_2$-Lipschitz continuity assumption, the estimate in \cite[(A.7)]{cosso_master_2022} is invalid and we should have \eqref{1754} instead. However, by using the arguments in Lemma \ref{lem estimate of b^i_n,m...}, \ref{lem. classical sol. of smooth approx.}, we showed that the estimates \cite[(A.6), (A.7), (A.12)]{cosso_master_2022} and the proofs of \cite[ Theorem 5.1, Theorem A.7]{cosso_master_2022} should be valid if \cite[Assumption (A)]{cosso_master_2022} adopts to the $\mathcal{W}_1$-Lipschitz continuity. \\
\hfill\\
\noindent Here is a counterexample to the estimate of $\mathcal{W}_2 \Big( \widehat{\mu}^{n,\overline{x}},  \frac{1}{n} \sum_{j=1}^{n} \delta_{x^j-y^j} \Big) $ in the proof of \cite[Lemma A.3]{cosso_master_2022}) which claims that
\begin{align}
\mathcal{W}_2 \bigg( \widehat{\mu}^{n,\overline{x}},  \frac{1}{n} \sum_{j=1}^{n} \delta_{x^j-y^j} \bigg) 
\leq \frac{1}{n} \sum_{j=1}^{n} |y^j|.
\label{1787}
\end{align}
Suppose $d=1$, $\overline{x}=(x^1,x^2,\ldots,x^n)^\top=(0,\ldots,0)^\top$ and $\overline{y}=(y^1,y^2,\ldots,y^n)^\top=(y^1,0,\ldots,0)^\top$. Consider the probability measures $\widehat{\mu}^{n,\overline{x}}=  \delta_{0}$  and $\frac{1}{n} \sum_{j=1}^{n} \delta_{x^j-y^j} = \frac{1}{n}\sum_{j=1}^n\delta_{-y^j}$. As $\widehat{\mu}^{n,\overline{x}}=  \delta_{0}$, all mass should be transported from/into this point only (see \cite[(11) in Chapter 1]{villani_topics_2003}), therefore, there is only one joint distribution $\pi$ such that $\pi(A \times \mathbb{R})=\widehat{\mu}^{n,\overline{x}}(A)=\delta_0(A)$ and $\pi(\mathbb{R} \times B)= \frac{1}{n}\sum_{j=1}^n\delta_{-y^j}(B)$ for any measurable sets $A$ and $B$ in $\mathbb{R}$. It is given by $\pi(A \times B) = \frac{1}{n}\sum_{j=1}^n\delta_0(A)\delta_{-y^j}(B)$. Thus, for $n>1$, we have
\begin{align*}
\mathcal{W}_2 \bigg( \widehat{\mu}^{n,\overline{x}},  \frac{1}{n} \sum_{j=1}^{n} \delta_{x^j-y^j} \bigg) 
= \Bigg(\int_{\mathbb{R}}|y|^2\frac{1}{n}\sum_{j=1}^n\delta_{-y^j}(dy)\Bigg)^{1/2} 
= \Bigg(\frac{1}{n}\sum_{j=1}^n |y^{j}|^2\Bigg)^{1/2} 
=  \frac{1}{n^{1/2}}\sum_{j=1}^n |y^j|,  
\end{align*}
where we have used $(y^1,y^2,\ldots,y^n)^\top=(y^1,0,\ldots,0)^\top \in \mathbb{R}^n$ in the last equality. It contradicts \eqref{1787}.
     \label{pham gap remark}
 \end{remark}}
 \noindent Let $\overline{X}^{m,\varepsilon,t,\overline{\xi},\overline{\alpha}}_s
	=(\overline{X}^{1,m,\varepsilon,t,\overline{\xi},\overline{\alpha}}_s,\ldots,\overline{X}^{n,m,\varepsilon,t,\overline{\xi},\overline{\alpha}}_s) \in \mathbb{R}^{dn}$ be the solution of 
	\begin{align}
		X_s^i =\,& \overline{\xi}^i + \int_{t}^s b_{n,m}^i(r,X_{r},\overline{\alpha}_r^i)dr + \int_t^s \sigma(r,X_{r}^i,\overline{\alpha}_r^i)d\overline{W}_r^i + \int_t^s \sigma^0(r)d\overline{W}^0_r\nonumber\\
		&+\varepsilon(\overline{B}^i_s - \overline{B}^i_t),
		\label{eq. state perturbed by e BM and smoothing}
	\end{align}
	for $i = 1,\dots,n$, where $X_s^i$ is the $i$th-component of $X_s$. We define
	\begin{align}
		\widetilde{v}_{\varepsilon,n,m}(t,\overline{\mu})
		:=&\sup_{\overline{\alpha}\in \overline{\mathcal{A}}^n_t}J_{\varepsilon,n}^*(t,\overline{\xi},\overline{\alpha})\nonumber\\
		:=&\sup_{\overline{\alpha}\in \overline{\mathcal{A}}^n_t} \dfrac{1}{n}\sum^n_{i=1}\mathbb{E}\Bigg[\int_t^T f^i_{n,m}\left(s,\overline{X}^{1,m,\varepsilon,t,\overline{\xi},\overline{\alpha}}_s,\ldots,
		\overline{X}^{n,m,\varepsilon,t,\overline{\xi},\overline{\alpha}}_s,
		\overline{\alpha}^i_s\right)ds\nonumber \\
		&\h{110pt}+ g^i_{n,m}\left(\overline{X}^{1,m,\varepsilon,t,\overline{\xi},\overline{\alpha}}_T,\ldots,
		\overline{X}^{n,m,\varepsilon,t,\overline{\xi},\overline{\alpha}}_T\right)\Bigg],
		\label{def. tilde v_e,n,m}
	\end{align}
	for any $t \in [0,T]$ and $\overline{\mu} \in \mathcal{P}_2(\mathbb{R}^{dn})$ such that $\mathcal{L}(\overline{\xi})=\overline{\mu}$.
	We define $v_{\varepsilon,n,m}(t,\mu)$, $v_{\varepsilon,n}(t,\mu):[0,T] \times \mathcal{P}_2(\mathbb{R}^d)\to\mathbb{R}$ by
	\begin{align}
		v_{\varepsilon,n,m}(t,\mu) := \widetilde{v}_{\varepsilon,n,m}(t,\mu\otimes\ldots\otimes\mu)
		\hspace{10pt}\text{and}\hspace{10pt}
		v_{\varepsilon,n}(t,\mu) := \widetilde{v}_{\varepsilon,n}(t,\mu\otimes\ldots\otimes\mu).
		\label{def. v_e=V_e mfg}
	\end{align}
	

	\begin{lemma}
		Suppose that {\color{black}Assumption (A*)} holds. There exists a constant $C_6=C_6(K,T,d)$ such that for any $(t,\mu) \in [0,T] \times \mathcal{P}_2(\mathbb{R}^d)$, we have $|v_{\varepsilon,n,m}(t,\mu)-v_{0,n,m}(t,\mu)|\leq C_6 \varepsilon$.
		\label{lem. |v_e,n,m-v_0,n,m| < C_6e}
	\end{lemma}
	\begin{proof}
        For every $i = 1,\ldots,n$, $m\in\mathbb{N}$, $\varepsilon \geq 0$, $t\in [0,T]$, $\overline{\xi}\in L^2(\overline{\Omega},\overline{\mathcal{F}}_t,\overline{\mathbb{P}};\mathbb{R}^{nd})$, $\overline{\alpha}\in\overline{\mathcal{A}}_t^n$, as $\varepsilon$ is the only parameter that is varying, we write $\overline{X}_s^{\varepsilon}$ to stand for $(\overline{X}_s^{1;\varepsilon},\ldots,\overline{X}_s^{n;\varepsilon})=(\overline{X}^{1,m,\varepsilon,t,\overline{\xi},\overline{\alpha}}_s,\ldots,\overline{X}^{n,m,\varepsilon,t,\overline{\xi},\overline{\alpha}}_s)$ which solves \eqref{eq. state perturbed by e BM and smoothing}. Then with \eqref{ineq. lip cts of b^i_n,m...}, it is easily seen that there is some constant $C_{d,K,T}>0$ such that 
\begin{align*}
 \mathbb{E}\left[ |X^{i;\varepsilon}_s - X^{i;0}_s|^2\right]
\leq C_{d,K,T}\mathbb{E}\bigg[\int_t^s
|X^{i;\varepsilon}_r - X^{i;0}_r|^2 dr+ \frac{1}{n}\sum_{j=1}^n \int_t^s 
|X^{j;\varepsilon}_r - X^{j;0}_r|^2 dr\bigg]
+\varepsilon^2d(s-t).
\end{align*}
        Summing on $i = 1,\ldots, n$, we have
        \begin{align*}
            \mathbb{E}\left[\sum_{i=1}^n|X^{i;\varepsilon}_s - X^{i;0}_s|^2\right]
            \leq C_{d,K,T}\Bigg[\mathbb{E}\left(\int_t^s\sum_{i=1}^n |X^{i;\varepsilon}_r - X^{i;0}_r|^2 dr\right)
            +n\varepsilon^2(s-t)\Bigg], 
        \end{align*}
        and hence Gr\"onwall's inequality gives
        \begin{align*}
            \mathbb{E}\left[\sum_{i=1}^n
            |X^{i;\varepsilon}_s - X^{i;0}_s|^2\right]
            \leq C_{d,K,T}n\varepsilon^2(s-t)e^{C_{d,K,T}(s-t)}.
        \end{align*}
       Let $\mu \in \mathcal{P}_2(\mathbb{R}^d)$ and take $\overline{\xi}=(\xi,\ldots,\xi)\in\mathbb{R}^{dn}$ such that $\mathcal{L}(\xi)=\mu$, we have
        \begin{align*}
& \left|v_{\varepsilon, n, m}(t, \mu)-v_{0, n, m}(t, \mu)\right| \\
& \leq \sup _{\overline{\alpha} \in \overline{\mathcal{A}}^n_t} \frac{1}{n} \sum_{i=1}^n 
\mathbb{E}
\left[\int_t^T\left|f_{n, m}^i\left(s, \overline{X}_s^{\varepsilon}, \overline{\alpha}_s^i\right)-f_{n, m}^i\left(s, \overline{X}_s^{0}, \overline{\alpha}_s^i\right)\right| d s+\left|g_{n, m}^i\left(\overline{X}_T^{\varepsilon}\right)-g_{n, m}^i\left(\overline{X}_T^{\varepsilon}\right)\right|\right] \\
& \leq \frac{2K}{n} \sup _{\overline{\alpha} \in \overline{\mathcal{A}}^n_t}
\sum_{i=1}^n \left\{\int_t^T\mathbb{E}\left[\left|\overline{X}_s^{i;\varepsilon}-\overline{X}_s^{i;0}\right|\right] d s 
+\mathbb { E }\left[\left|\overline{X}_T^{i;\varepsilon}-\overline{X}_T^{i;0}\right|\right]\right\} \\
& \leq \frac{2\sqrt{2}K}{n^{1/2}} \sup _{\overline{\alpha} \in \overline{\mathcal{A}}^n_t}
\left\{\sum_{i=1}^n \int_t^T T\mathbb{E}\left[\left|\overline{X}_s^{i;\varepsilon}-\overline{X}_s^{i;0}\right|^2\right] d s 
+\mathbb { E }\left[\left|\overline{X}_T^{i;\varepsilon}-\overline{X}_T^{i;0}\right|^2\right]\right\}^{1/2} \\
&\leq  2\sqrt{2}K\varepsilon\pig(C_{d,K,T} (T^3+T)e^{C_{d,K,T}T}\pig)^{1/2}.
\end{align*}
	\end{proof}

	\begin{lemma}
		Suppose that Assumption (A) holds and let $(t,\mu) \in [0,T] \times \mathcal{P}_2(\mathbb{R}^d)$. If there exists $q>2$ such that $ \mu \in \mathcal{P}_q(\mathbb{R}^d)$, then $\displaystyle\lim_{n\to \infty}\lim_{m\to \infty} v_{\varepsilon,n,m}(t,\mu)=v_\varepsilon(t,\mu)$.
		\label{lem. conv. of v_e,n,m to v_e}
	\end{lemma}

	\begin{proof}
		The proof goes as the same line as \cite[Theorem A.6]{cosso_master_2022} with \cite[Theorem 3.1, Theorem 3.6]{limit_theory_tan} replacing the limit theory in \cite{cosso_master_2022}, because of the appearance of the common noise.
        \end{proof}		
	\begin{lemma}
		Suppose that {\color{black}Assumptions (A*) and (B)} hold. The function $\overline{v}_{\varepsilon,n,m} : [0,T] \times \mathbb{R}^{dn} \to \mathbb{R}$ defined by
		\begin{align}
			\overline{v}_{\varepsilon,n,m}(t,\overline{x})=\overline{v}_{\varepsilon,n,m}(t,x^1,\ldots,x^n)
			:=\widetilde{v}_{\varepsilon,n,m}(t,\delta_{x^1}\otimes\ldots\otimes\delta_{x^n})
			\label{def. bar of v = tilde of v}
		\end{align}
		is a classical solution of the Bellman equation
		\rcm{\begin{equation}
			\left\{
			\begin{aligned}
				&\partial_t u(t,\overline{x})
				+\sup_{\overline{a} \in A^n}\Bigg\{\dfrac{1}{n}\sum^n_{i=1}f^i_{n,m}(t,\overline{x},a^i)
				+\dfrac{1}{2}\sum^n_{i=1}\textup{tr}\left[\Big((\sigma\sigma^\top)(t,x^i,a^i)+(\sigma^0\sigma^{0;\top})(t)+\varepsilon^2I_d\Big)\p_{x^ix^i}^2
				u(t,\overline{x})\right]\\ 
				&\h{80pt}+\sum^n_{i=1}\left\langle b^i_{n,m}(t,\overline{x},a^i)
				,\p_{x^i}u(t,\overline{x})\right\rangle
				+\frac{1}{2}\sum_{i,j=1,i\neq j}^n\textup{tr}\Big[(\sigma^0\sigma^{0;\top})(t)\p_{x^ix^j}^2u(t,\overline{x})\Big]\Bigg\}\\
				&=0\h{5pt} \text{in $[0,T) \times\mathbb{R}^{dn}$};\\
				&u(T,\overline{x})
				=\dfrac{1}{n}\sum^n_{i=1}g^i_{n,m}(\overline{x})
				\h{5pt} \text{ in $\mathbb{R}^{dn}$},
			\end{aligned}
			\right.   
			\label{eq. bellman}
		\end{equation}}
		where $\overline{a}=(a^1,\ldots,a^n)$ with $a^i\in A $ for $i=1,2,\ldots,n$. Moreover, there are constants $C_4=C_4(K)$ and $C_{n,m}$ such that for any $\varepsilon>0$, $n,m \in \mathbb{N}$, $t,t'\in [0,T]$ and $\overline{x}, \overline{x}' \in \mathbb{R}^{dn}$, we have
		\begin{equation}
			\begin{aligned}
				&|\p_{x^i}\overline{v}_{\varepsilon,n,m}(t,\overline{x})| \leq \dfrac{C_4}{n}
				\h{5pt}\text{for every $i=1,2,\ldots,n$ and}\h{10pt}\\
				&-C_{n,m}\leq\p^2_{\overline{x}_j\overline{x}_k}\overline{v}_{\varepsilon,n,m}(t,\overline{x}) \leq \dfrac{C_{n,m}}{\varepsilon^2}
				\h{5pt}\text{for every $j,k=1,2,\ldots,dn$,}
			\end{aligned}
			\label{ineq.|D v_e,n,m| and |D^2 v_e,n,m|}
		\end{equation}
		where we denote by $\p_{\overline{x}_j\overline{x}_k}^2 v(t,\overline{x}) \in  \mathbb{R}$ the second order derivative with respect to $\overline{x}_j$ and $\overline{x}_k$ for $j,k=1,2,\ldots,dn$.
		\label{lem. classical sol. of smooth approx.}
	\end{lemma}

{\color{black}\begin{proof} We refer readers to \cite[Theorem A.7]{cosso_master_2022} for proving that $\overline{v}_{\varepsilon,n,m}$ is a classical solution of the Bellman equation to \eqref{eq. bellman}. From \cite[Chapter 4.7, Theorem 4]{K08}, we deduce the estimate of the second order derivative in \eqref{ineq.|D v_e,n,m| and |D^2 v_e,n,m|}. It remains to prove the estimate of the first order derivative in \eqref{ineq.|D v_e,n,m| and |D^2 v_e,n,m|}. We would like to point out that the proof presented below differs from that in \cite[Theorem A.7]{cosso_master_2022}. Specifically, we were unable to reproduce identity \cite[(A.21)]{cosso_master_2022} therein, which may be attributed to a possible typographical error. As $\overline{v}_{\varepsilon,n,m}(t,\overline{x})\in C^{1,2}([0,T]\times\mathbb{R}^{dn})$, we notice that it is enough to establish that
\[
\left| \overline{v}_{\varepsilon,n,m}(t, \overline{x}) - \overline{v}_{\varepsilon,n,m}(t, \overline{z}) \right| \leq \frac{C_4}{n} \left| \overline{x} - \overline{z} \right|,
\]
when the components of $\overline{x} = (x^1, \ldots, x^n)$ and $\overline{z} = (z^1, \ldots, z^n)$ are all equal, apart from one component $x^k \neq z^k$ for some $k=1,2,\ldots,n$. Recalling the relations in \eqref{def. tilde v_e,n,m} and \eqref{def. bar of v = tilde of v}, we use the continuity in (4) of Lemma \ref{lem estimate of b^i_n,m...} to yield that
\begin{align}
&\left| \overline{v}_{\varepsilon,n,m}(t, \overline{x}) - \overline{v}_{\varepsilon,n,m}(t, \overline{z}) \right| \nonumber\\
&\leq 2K \sup_{\overline{\alpha} \in \overline{\mathcal{A}}^n_t} \frac{1}{n} \sum_{i=1}^{n} \mathbb{E} \left[ \int_{t}^{T} \left| \overline{X}_{s}^{i,m,\varepsilon,t,\overline{x},\overline{\alpha}} - \overline{X}_{s}^{i,m,\varepsilon,t,\overline{z},\overline{\alpha}} \right| ds + \left| \overline{X}_{T}^{i,m,\varepsilon,t,\overline{x},\overline{\alpha}} - \overline{X}_{T}^{i,m,\varepsilon,t,\overline{z},\overline{\alpha}} \right| \right].
\label{1897}
\end{align}
Suppose that $\overline{x}$ and $\overline{z}$ differ only for the first component $x^1 \neq z^1$. Fix $i = 1,2, \ldots, n$, the process $\overline{X}^i := \overline{X}_{s}^{i,m,\varepsilon,t,\overline{x},\overline{\alpha}} \in \mathbb{R}^d$ solves the following equation on $[t, T]$:

\[
\overline{X}_{s}^i = x^i 
+ \int_{t}^{s} b_{n,m}^i (r, \overline{X}_{r}^1, \ldots, \overline{X}_{r}^n, \overline{\alpha}_{r}^i) dr 
+ \int_{t}^{s} \sigma (r, \overline{X}_{r}^i, \overline{\alpha}_{r}^i) d\overline{W}^i_{r} 
+ \int_{t}^{s} \sigma^0 (r) d\overline{W}^0_{r}
+ \varepsilon (\overline{B}^i_{s} - \overline{B}^i_{t}).
\]
As the coefficients of the above equation are smooth and having bounded derivatives by Lemma \ref{lem estimate of b^i_n,m...} and Assumption (B), \cite[Chapter V, Theorem 39]{P05} allows us to differentiate the above process with respect to $x^1_k \in \mathbb{R}$ for some fixed $k=1,2,\ldots,d$ to obtain that

\[
\partial_{x^1_k} \overline{X}_{s}^i = e_k\delta_{1i} + \sum_{j=1}^{n} \int_{t}^{s} \partial_{x^j} b_{n,m}^i (r, \overline{X}_{r}^1, \ldots, \overline{X}_{r}^n, \overline{\alpha}_{r}^i) \partial_{x^1_k} \overline{X}_{r}^j dr
+ \int_{t}^{s}\pig( \partial_{x^1_k} \overline{X}_{r}^i\cdot\partial_{x}\pig) \sigma (r, \overline{X}_{r}^i, \overline{\alpha}_{r}^i)  d\overline{W}^i_{r},
\]
where $\pig( \partial_{x^1_k} \overline{X}_{r}^i\cdot\partial_{x}\pig) \sigma (r, \overline{X}_{r}^i, \overline{\alpha}_{r}^i)
=\sum_{l=1}^d \pig(\partial_{x^1_k} \overline{X}_{r}^i \pigr)_l
\partial_{x_l} \sigma (r, \overline{X}_{r}^i, \overline{\alpha}_{r}^i) \in \mathbb{R}^{d\times d}$ and $e_k\in \mathbb{R}^d$ is the standard basis vector in $\mathbb{R}^d$ such that $(e_k)_j=\delta_{kj}$. The $\mathbb{R}^{dn}$-valued continuous process $\partial_{x^1_k} \overline{X} := (\partial_{x^1_k} \overline{X}^1, \ldots, \partial_{x^1_k} \overline{X}^n)^\top$ is the unique solution to the above system of linear stochastic equations such that $\mathbb{E}\left[\max_{s\in [t,T]}|\partial_{x^1_k} \overline{X}_s|^2\right]<\infty$. Next, we provide an estimate of $\max_{s\in [t,T]}\mathbb{E}\pig[\sum_{i=1}^n|\partial_{x^1_k} \overline{X}^i_s|\pig]$ with a method akin to the proof of Tanaka's formula. Letting $\vartheta \in \mathbb{R}$, we consider the function $u_\vartheta:\mathbb{R}^d \to \mathbb{R}$ defined by 
\begin{align*}
    u_\vartheta(y):=\sqrt{|y|^2+\vartheta^2}.
\end{align*}
Direct calculation gives 
\begin{align*}
    \p_y u_\vartheta(y) = \frac{y}{u_\vartheta(y)},\quad \p_{yy}^2 u_\vartheta(x) = \frac{1}{u_\vartheta(y)}I_d - \frac{1}{\big[u_\vartheta(y)\big]^3}yy^\top.
\end{align*}
Applying It\^o formula to $u_\vartheta\left(\partial_{x^1_k} \overline{X}^i_s\right)$ gives 
\begin{align*}
&\h{-10pt}du_\vartheta\left(\partial_{x^1_k} \overline{X}^i_s\right)\\
=& \pig\langle \p_y u_\vartheta(\partial_{x^1_k} \overline{X}^i_s), d\partial_{x^1_k} \overline{X}_{s}^i\pig\rangle \\
&+ \frac{1}{2}\tr\left\{\left[\pig( \partial_{x^1_k} \overline{X}_{s}^i\cdot\partial_{x}\pig) \sigma (s, \overline{X}_{s}^i, \overline{\alpha}_{s}^i)\right]^\top \p^2_{yy} u_\vartheta\left(\partial_{x^1_k} \overline{X}^i_s\right)
\left[\pig( \partial_{x^1_k} \overline{X}_{s}^i\cdot\partial_{x}\pig) \sigma (s, \overline{X}_{s}^i, \overline{\alpha}_{s}^i)\right]\right\}ds\\
=&\left\langle \frac{\partial_{x^1_k} \overline{X}_{s}^i}{u_\vartheta(\partial_{x^1_k} \overline{X}_{s}^i)},\sum_{j=1}^n\partial_{x^j} b_{n,m}^i (s, \overline{X}_{s}^1, \ldots, \overline{X}_{s}^n, \overline{\alpha}_{s}^i) \partial_{x^1_k} \overline{X}_{s}^j ds
+ \left[\pig( \partial_{x^1_k} \overline{X}_{s}^i\cdot\partial_{x}\pig) \sigma (s, \overline{X}_{s}^i, \overline{\alpha}_{s}^i)\right] d\overline{W}^i_{s}\right\rangle\\
&+\frac{1}{2}\tr\left\{\left[\pig( \partial_{x^1_k} \overline{X}_{s}^i\cdot\partial_{x}\pig) \sigma (s, \overline{X}_{s}^i, \overline{\alpha}_{s}^i)\right]^\top 
\left(\frac{1}{u_\vartheta(\partial_{x^1_k} \overline{X}_{s}^i)}I_d - \frac{\pig(\partial_{x^1_k} \overline{X}_{s}^i\pig) 
\pig(\partial_{x^1_k} \overline{X}_{s}^{i}\pigr)^\top}{\pig[u_\vartheta(\partial_{x^1_k} \overline{X}_{s}^i)\pigr]^3}\right)
\left[\pig( \partial_{x^1_k} \overline{X}_{s}^i\cdot\partial_{x}\pig) \sigma (s, \overline{X}_{s}^i, \overline{\alpha}_{s}^i)\right]\right\}ds.
\end{align*}
It is clear that
\begin{align*}
\frac{\partial_{x^1_k} \overline{X}_{s}^i}{u_\vartheta(\partial_{x^1_k} \overline{X}_{s}^i)}\longrightarrow \frac{\partial_{x^1_k} \overline{X}_{s}^i}{\pig| \partial_{x^1_k} \overline{X}_{s}^i\pigr|}\mathbbm{1}_{\big\{\partial_{x^1_k} \overline{X}_{s}^i\neq 0\big\}},\quad\mathbb{P}\text{-a.s. as $\vartheta\to 0$.} 
\end{align*}
Moreover,
\begin{align*}
    &\h{-10pt}\frac{1}{2}\tr\left\{\left[\pig( \partial_{x^1_k} \overline{X}_{s}^i\cdot\partial_{x}\pig) \sigma (s, \overline{X}_{s}^i, \overline{\alpha}_{s}^i)\right]^\top \left(\frac{1}{u_\vartheta(\partial_{x^1_k} \overline{X}_{s}^i)}I_d\right)
    \left[\pig( \partial_{x^1_k} \overline{X}_{s}^i\cdot\partial_{x}\pig) \sigma (s, \overline{X}_{s}^i, \overline{\alpha}_{s}^i)\right]\right\}\\
    \longrightarrow&\frac{1}{2}\left(\frac{1}{\pig| \partial_{x^1_k} \overline{X}_{s}^i\pigr|}\mathbbm{1}_{\big\{\partial_{x^1_k} \overline{X}_{s}^i\neq 0\big\}}\right)
    \tr\left\{\left[\pig( \partial_{x^1_k} \overline{X}_{s}^i\cdot\partial_{x}\pig) \sigma (s, \overline{X}_{s}^i, \overline{\alpha}_{s}^i)\right]^\top 
    \left[\pig( \partial_{x^1_k} \overline{X}_{s}^i\cdot\partial_{x}\pig) \sigma (s, \overline{X}_{s}^i, \overline{\alpha}_{s}^i)\right]
    \right\}
\end{align*}
$\mathbb{P}$-a.s. as $\vartheta\to 0$. Similarly,
\begin{align*}
    &\h{-10pt}\frac{1}{2}\tr\left\{\left[\pig( \partial_{x^1_k} \overline{X}_{s}^i\cdot\partial_{x}\pig) \sigma (s, \overline{X}_{s}^i, \overline{\alpha}_{s}^i)\right]^\top 
    \left( \frac{\partial_{x^1_k} \overline{X}_{s}^i \pig(\partial_{x^1_k} \overline{X}_{s}^i\pig)^\top}{\pig[u_\vartheta(\partial_{x^1_k} \overline{X}_{s}^i)\pigr]^3}\right)
    \left[\pig( \partial_{x^1_k} \overline{X}_{s}^i\cdot\partial_{x}\pig) \sigma (s, \overline{X}_{s}^i, \overline{\alpha}_{s}^i)\right]\right\}\\
    \longrightarrow &\frac{1}{2}\left(\frac{1}{\pig| \partial_{x^1_k} \overline{X}_{s}^i\pigr|^3}\mathbbm{1}_{\big\{\partial_{x^1_k} \overline{X}_{s}^i\neq 0\big\}}\right)
    \tr\left\{\left[\pig( \partial_{x^1_k} \overline{X}_{s}^i\cdot\partial_{x}\pig) \sigma (s, \overline{X}_{s}^i, \overline{\alpha}_{s}^i)\right]^\top  \left(\partial_{x^1_k} \overline{X}_{s}^i \pig(\partial_{x^1_k} \overline{X}_{s}^i\pig)^\top\right)
    \left[\pig( \partial_{x^1_k} \overline{X}_{s}^i\cdot\partial_{x}\pig) \sigma (s, \overline{X}_{s}^i, \overline{\alpha}_{s}^i)\right]\right\}
\end{align*}
$\mathbb{P}$-a.s. as $\vartheta\to 0$. Therefore, by taking expectation and then passing $\vartheta\to 0$, together with the dominated convergence theorem, we conclude that
\begin{align}
\label{dX_before_exp}
    &\h{-10pt}\mathbb{E}\pig[|\partial_{x^1_k} \overline{X}^i_s|\pig]
    -\delta_{1i}\nonumber\\  
    =&\mathbb{E}\left[\int^s_t\left\langle \frac{\partial_{x^1_k} \overline{X}_{r}^i}{\pig| \partial_{x^1_k} \overline{X}_{r}^i\pigr|}
    ,\sum_{j=1}^n\partial_{x^j} b_{n,m}^i (r, \overline{X}_{r}^1, \ldots, \overline{X}_{r}^n, \overline{\alpha}_{r}^i) \partial_{x^1_k} \overline{X}_{r}^j \right\rangle 
    \mathbbm{1}_{\big\{\partial_{x^1_k} \overline{X}_{r}^i\neq 0\big\}}dr\right]\nonumber\\
&+\frac{1}{2}\mathbb{E}\left\{\int^s_t\frac{1}{\pig| \partial_{x^1_k} \overline{X}_{r}^i\pigr|}
    \tr\left\{\left[\pig( \partial_{x^1_k} \overline{X}_{r}^i\cdot\partial_{x}\pig) \sigma (r, \overline{X}_{r}^i, \overline{\alpha}_{r}^i)\right]^\top 
    \left[\pig( \partial_{x^1_k} \overline{X}_{r}^i\cdot\partial_{x}\pig) \sigma (r, \overline{X}_{r}^i, \overline{\alpha}_{r}^i)\right]
    \right\}
    \mathbbm{1}_{\big\{\partial_{x^1_k} \overline{X}_{r}^i\neq 0\big\}}dr\right\}\nonumber\\
&-\frac{1}{2}\mathbb{E}\Bigg\{\int^s_t\frac{1}{\pig| \partial_{x^1_k} \overline{X}_{r}^i\pigr|^3}
    \tr\left\{\left[\pig( \partial_{x^1_k} \overline{X}_{r}^i\cdot\partial_{x}\pig) \sigma (r, \overline{X}_{r}^i, \overline{\alpha}_{r}^i)\right]^\top  \left(\partial_{x^1_k} \overline{X}_{r}^i \pig(\partial_{x^1_k} \overline{X}_{r}^i\pig)^\top\right)
    \left[\pig( \partial_{x^1_k} \overline{X}_{r}^i\cdot\partial_{x}\pig) \sigma (r, \overline{X}_{r}^i, \overline{\alpha}_{r}^i)\right]\right\}\cdot\nonumber\\
&\h{380pt}\mathbbm{1}_{\big\{\partial_{x^1_k} \overline{X}_{r}^i\neq 0\big\}}dr\Bigg\}.
\end{align}
The term in the third line of \eqref{dX_before_exp} can be estimated by Cauchy–Schwarz inequality 
\begin{align*}
    &\h{-10pt}\frac{1}{\pig| \partial_{x^1_k} \overline{X}_{r}^i\pigr|}
    \tr\left\{\left[\pig( \partial_{x^1_k} \overline{X}_{r}^i\cdot\partial_{x}\pig) \sigma (r, \overline{X}_{r}^i, \overline{\alpha}_{r}^i)\right]^\top 
    \left[\pig( \partial_{x^1_k} \overline{X}_{r}^i\cdot\partial_{x}\pig) \sigma (r, \overline{X}_{r}^i, \overline{\alpha}_{r}^i)\right]
    \right\}
    \mathbbm{1}_{\big\{\partial_{x^1_k} \overline{X}_{r}^i\neq 0\big\}}\\
    =\,& \frac{1}{\pig| \partial_{x^1_k} \overline{X}_{r}^i\pigr|}
    \sum_{p,q=1}^d\left|\sum_{l=1}^d\pig( \partial_{x^1_k} \overline{X}_{r}^i\pigr)_l
    \pig(\partial_{x_l} \sigma (r, \overline{X}_{r}^i, \overline{\alpha}_{r}^i)\pigr)_{pq}\right|^2
    \mathbbm{1}_{\big\{\partial_{x^1_k} \overline{X}_{r}^i\neq 0\big\}}\\
    \leq\,&d K^2\pig| \partial_{x^1_k} \overline{X}_{r}^i\pigr|.
\end{align*}
Similarly, the term in the forth line of \eqref{dX_before_exp} can be estimated by
\begin{align*}
    & \h{-10pt}\frac{1}{\pig| \partial_{x^1_k} \overline{X}_{r}^i\pigr|^3}
    \tr\left\{\left[\pig( \partial_{x^1_k} \overline{X}_{r}^i\cdot\partial_{x}\pig) \sigma (r, \overline{X}_{r}^i, \overline{\alpha}_{r}^i)\right]^\top  \left(\partial_{x^1_k} \overline{X}_{r}^i \pig(\partial_{x^1_k} \overline{X}_{r}^i\pig)^\top\right)
    \left[\pig( \partial_{x^1_k} \overline{X}_{r}^i\cdot\partial_{x}\pig) \sigma (r, \overline{X}_{r}^i, \overline{\alpha}_{r}^i)\right]\right\}\mathbbm{1}_{\big\{\partial_{x^1_k} \overline{X}_{r}^i\neq 0\big\}}\\
    \leq &\frac{1}{\pig| \partial_{x^1_k} \overline{X}_{r}^i\pigr|^3} \pig| \partial_{x^1_k} \overline{X}_{r}^i\pigr|^4
    \left|\partial_{x} \sigma (s, \overline{X}_{r}^i, \overline{\alpha}_{r}^i)\right|^2
\mathbbm{1}_{\big\{\partial_{x^1_k} \overline{X}_{r}^i\neq 0\big\}}\\
    \leq & d K^2 \pig| \partial_{x^1_k} \overline{X}_{r}^i\pigr|.
\end{align*}
Therefore, \eqref{dX_before_exp} reduces to
\begin{align*}
  \mathbb{E}\pig[|\partial_{x^1_k} \overline{X}^i_s|\pig]
    &\leq \delta_{1i}+\mathbb{E}\left[\int^s_t\sum_{j=1}^n\pig|\partial_{x^j} b_{n,m}^i (r, \overline{X}_{r}^1, \ldots, \overline{X}_{r}^n, \overline{\alpha}_{r}^i) \partial_{x^1_k} \overline{X}_{r}^j\pig| 
    +dK^2\pig| \partial_{x^1_k} \overline{X}_{s}^i\pigr|dr\right]\nonumber\\
     &\leq \delta_{1i}+\mathbb{E}\left[\int^s_t\sum_{j=1}^n\pig|\partial_{x^j} b_{n,m}^i (r, \overline{X}_{r}^1, \ldots, \overline{X}_{r}^n, \overline{\alpha}_{r}^i)\pig|
     \pig| \partial_{x^1_k} \overline{X}_{r}^j\pig| 
    +dK^2\pig| \partial_{x^1_k} \overline{X}_{s}^i\pigr|dr\right].
\end{align*}
Summing over $i=1,2,\ldots,n$, we have
\begin{align*}
&\h{-10pt}\mathbb{E}\left[\sum^n_{i=1}\big|\partial_{x^1_k} \overline{X}_s^i\big| \right]\\
\leq \,&1+\int^s_t \mathbb{E}\left[\sum_{j=1}^{n}\left(\sum_{i=1}^{n}\big|\partial_{x^j} b_{n,m}^i (r, \overline{X}_{r}^1, \ldots, \overline{X}_{r}^n, \overline{\alpha}_{r}^i)\big|\right)
\big| \partial_{x^1_k} \overline{X}_{r}^j \big|
+ dK^2 \sum^n_{i=1}\pig| \partial_{x^1_k} \overline{X}_{r}^i\pigr|\right]dr\\
\leq \,&1+\int^s_t \mathbb{E}\left[\max_{1\leq \ell\leq n}\left(\sum_{i=1}^{n}\big|\partial_{x^\ell} b_{n,m}^i (r, \overline{X}_{r}^1, \ldots, \overline{X}_{r}^n, \overline{\alpha}_{r}^i)\big|\right)
\sum^n_{j=1}\big| \partial_{x^1_k} \overline{X}_{r}^j \big| 
+d K^2 \sum^n_{j=1}|\partial_{x^1_k} \overline{X}^j_{r}|\right]dr.
\end{align*}
 The Lipschitz continuity estimate for $b_{n,m}^i$ in Lemma \ref{lem estimate of b^i_n,m...} deduces that
\begin{align*}
    & \max_{1 \leq \ell \leq n} \sum_{i=1}^{n} \left| \partial_{x^\ell} b_{n,m}^i (s, \overline{X}_{s}^1, \ldots, \overline{X}_{s}^n, \overline{\alpha}_{s}^i) \right|\\
    &= \max_{1 \leq \ell \leq n} \left( \left| \partial_{x^\ell} b_{n,m}^\ell (s, \overline{X}_{s}^1, \ldots, \overline{X}_{s}^n, \overline{\alpha}_{s}^\ell) \right| 
    + \sum_{i=1, i \neq \ell}^{n} \left| \partial_{x^\ell} b_{n,m}^i (s, \overline{X}_{s}^1, \ldots, \overline{X}_{s}^n, \overline{\alpha}_{s}^i) \right| \right)\\
    &\leq \sqrt{d}\left[K \left( 1 + \frac{1}{n} \right) + K\frac{n-1}{n} \right]\\
    &= 2\sqrt{d}K.
\end{align*}
It gives us that $
    \mathbb{E}\left[\sum^n_{i=1}\big|\partial_{x^1_k} \overline{X}_s^i\big| \right] \leq 1 + (2\sqrt{d}K+ dK^2) \int_{t}^{s} \mathbb{E}\left[\sum^n_{i=1}\big|\partial_{x^1_k} \overline{X}_r^i\big| \right] \, dr
    $. Gr\"onwall's inequality yields
\begin{align*}
\mathbb{E}\left[\sum^n_{i=1}\big|\partial_{x^1_k} \overline{X}_s^i\big| \right] \leq e^{(2\sqrt{d}K+dK^2)T}, \quad \text{for every $s\in [t,T]$}.
\end{align*}
Furthermore, we have
\begin{align}
\mathbb{E}\left[\sum^n_{i=1}\big|\partial_{x^1} \overline{X}_s^i\big| \right]
=\mathbb{E}\left[\sum^n_{i=1}\left(\sum^d_{k=1}\big|\partial_{x^1_k} \overline{X}_s^i\big|^2\right)^{1/2} \right]
\leq\mathbb{E}\left[\sum^n_{i=1}\sum^d_{k=1}\big|\partial_{x^1_k} \overline{X}_s^i\big| \right]
\leq d e^{(2\sqrt{d}K+dK^2)T},
\label{1947}
\end{align}
for every $s\in [t,T]$. \\
\hfill\\
Finally, recalling that $\overline{x}$ and $\overline{z}$ differ only for the first component $x^1 \neq z^1$, we define 
\begin{align*}
    F_{s}^{i,x^1} := \overline{X}_{s}^{i,m,\varepsilon,t,(x^1,\ldots,x^n),\overline{\alpha}}.
\end{align*}
We obtain from \eqref{1897} and \eqref{1947} that
\begin{align*}
&\left| \overline{v}_{\varepsilon,n,m}(t, \overline{x}) - \overline{v}_{\varepsilon,n,m}(t, \overline{z}) \right| \nonumber\\
&\leq \frac{2K}{n}\sup_{\overline{\alpha} \in \overline{\mathcal{A}}^n_t}  \sum_{i=1}^{n} \mathbb{E} \left[ \int_{t}^{T} 
\left| F_{s}^{i,x^1+(z^1-x^1)} -F_{s}^{i,x^1}  \right| ds + \left|  F_{T}^{i,x^1+(z^1-x^1)} -F_{T}^{i,x^1} \right| \right]\\
&=\frac{2K}{n}\sup_{\overline{\alpha} \in \overline{\mathcal{A}}^n_t} \sum_{i=1}^{n} \mathbb{E} \left[ \int_{t}^{T} \left| \int_0^1 \p_{y^1}F^{i,y^1}_s(z^1-x^1) d\lambda\right| ds 
+ \left| \int_0^1 \p_{y^1}F^{i,y^1}_T(z^1-x^1)d\lambda\right| \right]\Bigg|_{y^1=x^1+\lambda(z^1-x^1)}\\
&\leq \frac{2K}{n}
\sup_{\overline{\alpha} \in \overline{\mathcal{A}}^n_t}
\sum_{i=1}^{n} \mathbb{E} \left[ \int_{t}^{T}  \int_0^1 \pig|\p_{y^1}F^{i,y^1}_s\pig|\Big|_{y^1=x^1+\lambda(z^1-x^1)}
 d\lambda ds 
+ \int_0^1 \pig|\p_{y^1}F^{i,y^1}_T\pig|\Big|_{y^1=x^1+\lambda(z^1-x^1)}
d\lambda\right]|z^1-x^1|\\
&\leq \frac{C_{d,K,T}}{n}|z^1-x^1|.
\end{align*}
\end{proof}}
\begin{theorem}
			Suppose that Assumptions (A)-(B) hold. For every $\varepsilon > 0$,
		$n,m \in \mathbb{N}$, the function $\overline{v}_{\varepsilon,n,m}$ defined in \eqref{def. bar of v = tilde of v} and the function $v_{\varepsilon,n,m}$ defined in \eqref{def. v_e=V_e mfg} satisfy the following:
		\begin{enumerate}[(1).]
			\item 
			for any $(t,\mu) \in [0,T] \times \mathcal{P}_2(\mathbb{R}^d)$, we have
			\begin{align}
				v_{\varepsilon,n,m}(t,\mu)
				=\int_{\mathbb{R}^{dn}}
				\overline{v}_{\varepsilon,n,m}
				(t,x^1,\ldots,x^n)
				\mu(dx^1)\otimes\ldots\otimes\mu(dx^n)\, ;
				\label{eq. v_e,n,m=int bar of v_e,n,m}
			\end{align}
			\item  $v_{\varepsilon,n,m}(t,\mu) \in C_1^{1,2}([0,T] \times \mathcal{P}_2(\mathbb{R}^d))$;
			\item there is a constant $\ell_2=\ell_2(K,T)$ such that for any $(t,\mu) \in [0,T] \times \mathcal{P}_2(\mathbb{R}^d)$, we have $|v_{\varepsilon,n,m}(t,\mu)|\leq \ell_2$;
			\item $v_{\varepsilon,n,m}(t,\mu)$ solves the following equation in the classical sense
			\begin{equation*}
				\left\{
				\begin{aligned}
					&\partial_t u(t,\mu)\\
					&+\int_{\mathbb{R}^{dn}}
					\sup_{\overline{a} \in A^n}\Bigg\{\dfrac{1}{n}\sum^n_{i=1}f^i_{n,m}(t,\overline{x},a^i)
					+\dfrac{1}{2}
					\sum^n_{i=1}\textup{tr}\left[\Big((\sigma\sigma^\top)(t,x^i,a^i)
					+(\sigma^0\sigma^{0;\top})(t)+\varepsilon^2I_d\Big)\p_{x^ix^i}^2
					\overline{v}_{\varepsilon,n,m}
					(t,\overline{x})\right]\\
					&\h{75pt}+\frac{1}{2}\sum_{i,j=1,i\neq j}^n\textup{tr}\Big[(\sigma^0\sigma^{0;\top})(t)\p_{x^ix^j}^2\overline{v}_{\varepsilon,n,m}(t,\overline{x})\Big]\\
					&\h{75pt}+\sum^n_{i=1}\left\langle b^i_{n,m}(t,\overline{x},a^i),\p_{x^i}\overline{v}_{\varepsilon,n,m}(t,\overline{x})\right\rangle\Bigg\}\mu(dx^1)\otimes\ldots\otimes \mu(dx^n)\\
					&=0\h{5pt} \text{for any $(t,\mu) \in [0,T) \times\mathcal{P}_2(\mathbb{R}^{d})$};\\
					& u(T,\mu)
					=\dfrac{1}{n}\sum^n_{i=1}\int_{\mathbb{R}^{dn}}g^i_{n,m}(\overline{x})\mu(dx^1)\otimes\ldots\otimes\mu(dx^n)
					\h{5pt} \text{for any $\mu \in \mathcal{P}_2(\mathbb{R}^{d})$}.
				\end{aligned}
				\right.    
			\end{equation*}
		\end{enumerate}
		\label{thm v_e,n,m}
	\end{theorem}
	\begin{proof} The proof follows from Lemma \ref{lem. classical sol. of smooth approx.} and \cite[Theorem A.7]{cosso_master_2022}.
 \end{proof}

\section{Viscosity Solution Theory}
Our main goal in this paper is to show that the value function defined by \eqref{value_function_after_law_invariance} is the unique viscosity solution to the following HJB equation:
 \rcm{\begin{align}
 \label{HJB_section_5}
			\left\{\begin{aligned}
				&\partial_t u(t,\mu)
				+\int_{\mathbb{R}^d}
				\sup_{a\in A} \Bigg\{f(t,x,\mu,a)+b(t,x,\mu,a)\cdot \partial_\mu u(t,\mu)(x) \\
				&+ \dfrac{1}{2}\textup{tr}\Big(\sigma(t,x,a)\big[\sigma(t,x,a)\big]^\top\partial_x\partial_\mu u(t,\mu)(x)\Big)\Bigg\}\mu(dx)
				+ \dfrac{1}{2}\textup{tr}\Big[\sigma^0(t)[\sigma^0(t)]^\top\mathcal{H}u(t,\mu)\Big]\\
			&=0\h{5pt} \text{for $(t,\mu) \in [0,T) \times\mathcal{P}_2(\mathbb{R}^{d})$};\\	&u(T,\mu)=\int_{\mathbb{R}^d}g(x,\mu)\mu(dx) \h{10pt}\text{for $\mu\in \mathcal{P}_2(\mathbb{R}^d)$}.
			\end{aligned}\right.
		\end{align}}
	We choose the set of test function to be $PC_{1}^{1,2}([0,T]\times\mathcal{P}_2(\mathbb{R}^d))$. Now, we are ready to define our notion of viscosity subsolution (resp. supersolution).
	\begin{definition}
		A continuous function $u:[0, T] \times \mathcal{P}_2(\mathbb{R}^d) \rightarrow \mathbb{R}$ is called a viscosity subsolution of equation \eqref{HJB_section_5} if
		\begin{enumerate}
			\item[(1a).] $u(T, \mu) \leq \displaystyle\int_{\mathbb{R}^d} g(x, \mu) \mu(d x)$, for every $\mu \in \mathcal{P}_2(\mathbb{R}^d) ;$
			\item[(1b).] for any $(t, \mu) \in[0, T) \times \mathcal{P}_2(\mathbb{R}^d)$ and $\varphi \in PC_{1}^{1,2}([0,T]\times\mathcal{P}_2(\mathbb{R}^d))$ such that $u-\varphi$ attains a maximum with a value of $0$ at $(t, \mu)$, then the first equation of \eqref{HJB_section_5} holds  with the inequality sign $\geq$ replacing the equality sign and with $\varphi$ replacing  $u$.
		\end{enumerate}
  {\color{black}For any function $h:\mathcal{P}_2(\mathbb{R}^d)\to \mathbb{R}$, it could be naturally interpreted as a function with domain $\mathcal{P}_2(\mathbb{R}^d\times A)$ by $h(\nu) := h(\mu)$, with $\mu$ being the marginal distribution of $\nu$ on $\mathbb{R}^d$ such that $\nu(\cdot \times A)=\mu(\cdot)$. We define the respective projections $\pi_d : \mathbb{R}^{2d} \to \mathbb{R}^{d}$ and $\pi_{d \times d} : \mathbb{R}^{2d \times 2d} \to \mathbb{R}^{d \times d}$ such that for any $y=(y_1,y_2,\ldots,y_{2d})^\top \in \mathbb{R}^{2d}$ and $M\in\mathbb{R}^{2d \times 2d}$, it holds that $\pi_d(y)=(y_1,y_2,\ldots,y_{d})^\top$ and $(\pi_{d \times d}(M))_{ij}=M_{ij}$. A continuous function $u:[0, T] \times \mathcal{P}_2(\mathbb{R}^d) \rightarrow \mathbb{R}$ is called a viscosity supersolution of equation \eqref{HJB_section_5} if
		\begin{enumerate}
			\item[(2a).] $u(T, \mu) \geq \displaystyle\int_{\mathbb{R}^d} g(x, \mu) \mu(d x)$, for every $\mu \in \mathcal{P}_2(\mathbb{R}^d) ;$
			\item[(2b).] for any $s_0\in [0,T]$, $(t,\nu)\in [0,s_0)\times\mathcal{P}_2(\mathbb{R}^d\times A)$ and $\varphi \in PC_{1}^{1,2}([0,s_0]\times\mathcal{P}_2(\mathbb{R}^d\times A))$ such that $u-\varphi$ attains a minimum with a value of $0$ at $(t,\nu)$ over $[0,s_0]\times\mathcal{P}_2(\mathbb{R}^d\times A)$, the following inequality holds:
\begin{align}
 \label{HJB_supersolution}
				&\partial_t \varphi(t,\nu)+\int_{\mathbb{R}^d\times A}
				 \Bigg\{f(t,x,\mu,a)+b(t,x,\mu,a)\cdot \partial_\mu \varphi(t,\nu)(x,a) \nonumber\\
				&+ \dfrac{1}{2}\textup{tr}\Big((\sigma(t,x,a)\big[\sigma(t,x,a)\big]^\top)\partial_x\partial_\mu \varphi(t,\nu)(x,a)\Big)\Bigg\}\nu(dx,da)+ \dfrac{1}{2}\textup{tr}\Big[\sigma^0(t)[\sigma^0(t)]^\top\mathcal{H}_{d\times d}\varphi(t,\nu)\Big] \nonumber\\
				&\leq 0, 
		\end{align}
    where $\mu$ is the marginal distribution of $\nu$ on $\mathbb{R}^d$, $\partial_\mu \varphi(t,\nu)(\cdot,\cdot):=\pi_d(\partial_\nu \varphi(t,\nu)(\cdot,\cdot))$, $\mathcal{H}_{d\times d}\varphi := \pi_{d\times d}(\mathcal{H}\varphi)$, recalling that $\partial_\nu \varphi(t,\nu)(\cdot,\cdot):\mathbb{R}^d\times A\to\mathbb{R}^d \times A$ and $\mathcal{H}$ maps functions on $\mathcal{P}_2(\mathbb{R}^d\times A)$ to $\mathbb{R}^{2d\times 2d}$. 
\end{enumerate}}
A continuous function $u:[0, T] \times \mathcal{P}_2(\mathbb{R}^d) \rightarrow \mathbb{R}$ is called a viscosity solution of \eqref{HJB_section_5} if it is both a viscosity subsolution and a viscosity supersolution.
		\label{def. of vis sol}
	\end{definition}
{\color{black}
\begin{remark}
\label{New_implies_CL}
We note that this definition of viscosity supersolution is stronger than Crandall-Lions' definition which is adopted by \cite{cosso_master_2022}. In fact, items (2a)-(2b) in Definition \ref{def. of vis sol} are sufficient conditions for the  usual Crandall-Lions' definition of supersolution. Assume $u$ is a viscosity supersolution satisfying items (2a)-(2b) in Definition \ref{def. of vis sol}. Let $\varphi\in  PC_{1}^{1,2}([0,T]\times\mathcal{P}_2(\mathbb{R}^d))$ such that $u-\varphi$ attains a minimum with a value of $0$ at $(t_0,\mu_0)$. Now, we interpret both $u$, $\varphi$ as functions on $[0,T]\times\mathcal{P}_2(\mathbb{R}^d\times A)$.  Let $\xi \in L^2(\Omega,\mathcal{F},\mathbb{P};\mathbb{R}^d)$, $\mathcal{L}(\xi) = \mu_0$, and $\alpha' \in \mathcal{M}_{t_0}$ be arbitrary, where $\mathcal{M}_{t}$ is the set of $\mathcal{F}_{t}^{t}$-measurable random variables. Then for $\nu_0:=\mathcal{L}(\xi,\alpha')\in\mathcal{P}_2(\mathbb{R}^d\times A)$, $u-\varphi$ attains a minimum at $(t_0,\nu_0)$, therefore from \eqref{HJB_supersolution}
\begin{align*}
&\partial_t\varphi(t_0,\mu_0)+\mathbb{E}\bigg\{f(t_0,\xi,\mu_0,\alpha') + \big[\partial_\mu \varphi(t_0,\mu_0)(\xi)\cdot b(t_0,\xi,\mu_0,\alpha')\big]\\
			&+\frac{1}{2}\operatorname{tr}\Big[\partial_x\partial_\mu \varphi(t_0,\mu_0)(\xi)\sigma(t_0,\xi,\alpha')[\sigma(t_0,\xi,\alpha')]^\top\Big]
   +\dfrac{1}{2}\textup{tr}\Big[\sigma^0(t_0)[\sigma^0(t_0)]^\top\mathcal{H}\varphi(t_0,\mu_0)\Big]\bigg\}\leq 0.
		\end{align*}
  Since $\alpha'\in\mathcal{M}_{t_0}$ is arbitrary, we have
		\begin{align*}
    			&\partial_t\varphi(t_0,\mu_0)+\sup_{\alpha'\in\mathcal{M}_{t_0}}\mathbb{E}\bigg\{f(t_0,\xi,\mu_0,\alpha') + \partial_\mu \varphi(t_0,\mu_0)(\xi)\cdot b(t_0,\xi,\mu_0,\alpha')\\
			&+\frac{1}{2}\operatorname{tr}\Big[\partial_x\partial_\mu \varphi(t_0,\mu_0)(\xi)\sigma(t_0,\xi,\alpha')[\sigma(t_0,\xi,\alpha')]^\top\Big]
   +\dfrac{1}{2}\textup{tr}\Big[\sigma^0(t_0)[\sigma^0(t_0)]^\top\mathcal{H}\varphi(t_0,\mu_0)\Big]\bigg\} \leq 0.
		\end{align*}
		We claim that the above is equivalent to the following:
		\begin{align}
			&\partial_t \varphi(t_0,\mu_0)
			+\int_{\mathbb{R}^d}
			\sup_{a\in A}\Bigg\{f(t_0,x,\mu_0,a)+b(t_0,x,\mu_0,a)\cdot \partial_\mu \varphi(t_0,\mu_0)(x) \nonumber\\
			&+ \dfrac{1}{2}\textup{tr}\Big[\partial_x\partial_\mu \varphi(t_0,\mu_0)(x)\sigma(t_0,x,a)[\sigma(t_0,x,a)]^\top
			\Big]\Bigg\}\mu_0(dx)+ \dfrac{1}{2}\textup{tr}\Big[\sigma^0(t_0)[\sigma^0(t_0)]^\top\mathcal{H}\varphi(t_0,\mu_0)\Big]\leq 0.
   \label{1176}
		\end{align}
        Define $H(t,x,\mu,a):= f(t,x,\mu,a) + \partial_\mu \varphi(t,\mu)(x)\cdot b(t,x,\mu,a)+\frac{1}{2}\operatorname{tr}\big[\partial_x\partial_\mu \varphi(t,\mu)(x)\sigma(t,x,a)\sigma(t,x,a)^\top\big]$, we aim to show that
	\begin{align*}
\sup_{\alpha'\in\mathcal{M}_{t_0}}\mathbb{E}H(t_0,\xi,\mu_0,\alpha') = \int_{\mathbb{R}^d}\sup_{a\in A}H(t_0,x,\mu_0,a)\mu_0(dx),\quad \text{ for   $\xi \in L^2(\Omega,\mathcal{F},\mathbb{P};\mathbb{R}^d)$ such that $\mathcal{L}(\xi) = \mu_0$}.
		\end{align*}
	To see this, we observe that for any $\alpha' \in \mathcal{M}_{t_0}$,
		\begin{align*}
			\mathbb{E}H(t_0,\xi,\mu_0,\alpha') \leq \int_{\mathbb{R}^d}\sup_{a\in A}H(t_0,x,\mu_0,a)\mu_0(dx),
		\end{align*}
    		and thus $\sup_{\alpha'\in\mathcal{M}_{t_0}}\mathbb{E}H(t_0,\xi,\mu_0,\alpha') \leq \int_{\mathbb{R}^d}\sup_{a\in A}H(t_0,x,\mu_0,a)\mu_0(dx)$. Invoking the regularity of $f$, $\varphi$, $b$ and $\sigma$, for any $\varepsilon>0$, we apply the measurable selection theorem (see for instance \cite{measurable_selection}) to find a $\alpha^\varepsilon\in\mathcal{M}_{t_0}$ such that 
		\begin{align*}
			\sup_{\alpha'\in\mathcal{M}_{t_0}}\mathbb{E}H(t_0,\xi,\mu_0,\alpha') \geq \mathbb{E}H(t_0,\xi,\mu_0,\alpha^\varepsilon)\geq \int_{\mathbb{R}^d}\sup_{a\in A}H(t_0,x,\mu_0,a)\mu_0(dx)-\varepsilon.
		\end{align*}
Since $\varepsilon$ is arbitrary, we conclude the claim in \eqref{1176}, which is the classical Crandall-Lions' definition.
\end{remark}
\begin{remark}\label{techincal_gap_def}
We impose this stronger definition because when proving the comparison theorem for the viscosity supersolution, we encounter technical difficulties that the Crandall-Lions' definition alone cannot resolve. 
We believe that a similar issue arises in Step II of the proof of \cite[Theorem 5.1]{cosso_master_2022}, resulting in certain gaps in the arguments presented therein.
Specifically, when we imitate the method used for the subsolution in the proof of Theorem \ref{thm compar}, it requires the inequality \eqref{1941} but with a reversed sign. Due to the supremum taken and the integration in \eqref{1941}, only one direction of the inequality is allowed. To bypass this issue, we prove the comparison theorem by showing that for every step control, i.e., controls of the form $\alpha = \sum_{i=0}^{n-1} \mathfrak{a}_i 1_{[t_i,t_{i+1})}$, with $0 =: t_0 < t_1 < \ldots < t_n := T$ and $\mathfrak{a}_i \in \mathcal{M}_{t_i}$ (where $\mathcal{M}_{t}$ is the set of $\mathcal{F}_{t}^{t}$-measurable random variables), the supersolution $u_2$ satisfies
\begin{align}
\label{u_2_geq_J}
u_2(t,\mu) \geq J(t,\xi,\alpha),
\end{align}
for every $(t,\xi) \in [0,T] \times L^2(\Omega,\mathcal{F},\mathbb{P};\mathbb{R}^d)$ such that $\mu=\mathcal{L}(\xi)$. Clearly this leads to $u_2 \geq v$ where $v$ is the value function defined in \eqref{value_function_after_law_invariance}. By using inductive arguments, it is sufficient to prove \eqref{u_2_geq_J} by proving that for every $\underline{t}, \overline{t} \in [0,T]$ with $\underline{t} < \overline{t}$, it holds that
\begin{align*}
    u_2(\underline{t},\mu) \geq \mathbb{E}\Bigg[\int_{\underline{t}}^{\overline{t}} f\left( r, X_r^{\underline{t}, \xi, \mathfrak{a}}, \mathbb{P}_{X_r^{\underline{t}, \xi, \mathfrak{a}}}, \mathfrak{a} \right)dr + u_2(\overline{t},\mathbb{P}^{W^0}_{X^{\underline{t},\xi,\mathfrak{a}}_{\overline{t}}})\Bigg], 
\end{align*}
for every $\xi \in L^2((\Omega,\mathcal{F},\mathbb{P});\mathbb{R}^d)$ with $\mathcal{L}(\xi)=\mu$, and  $\mathfrak{a} \in \mathcal{M}_{\underline{t}}$. It is not hard to see that the right-hand side is law invariant with respect to the joint law of the initial random variable and the control. Thus, $v^*(t,\nu) := \mathbb{E}\left[\int_{t}^{\overline{t}} f\left( r, X_r^{t, \xi, \mathfrak{a}}, \mathbb{P}_{X_r^{t, \xi, \mathfrak{a}}}, \mathfrak{a} \right)dr + u_2(\overline{t},\mathbb{P}^{W^0}_{X^{t,\xi,\mathfrak{a}}_{\overline{t}}})\right]$ is a real-valued function on $[0,\overline{t}]\times\mathcal{P}_2(\mathbb{R}^d \times A)$ with $\nu$ being the joint law of the initial random variable and the control. It seems that we must involve the function $v^*_{n,m}(t,\nu)$, the smooth and finite-dimensional approximation of $v^*(t,\nu)$, as a part of the test function used for comparison. The Crandall-Lions' definition cannot handle cases where the minimum point lies in $[0, T] \times \mathcal{P}_2(\mathbb{R}^d \times A)$. A natural approach is to modify the test function on $[0, T] \times \mathcal{P}_2(\mathbb{R}^d \times A)$ to be a function on $[0, T] \times \mathcal{P}_2(\mathbb{R}^d)$. However, finding such a modification that remains within our set of test functions and retains the necessary information (such as the $L$-derivatives) is challenging. For instance, we are unable to treat $v^*_{n,m}(t,\nu)$ as a function on $[0,T]\times\mathcal{P}_2(\mathbb{R}^d)$ for a fixed control $\mathfrak{a}$ directly, as $v^*_{n,m}(t,\nu)$ is not law invariant with respect to the marginal of $\nu$ on $\mathbb{R}^d$ (law of $\xi$). Furthermore, we may define $\underline{v}^* : [0, T] \times \mathcal{P}_2(\mathbb{R}^d) \to \mathbb{R}$ by
\begin{align*}
    \underline{v}^*(t, \mu) = \sup_{\substack{\text{$\nu \in\mathcal{P}_2(\mathbb{R}^d\times A)$ s.t.}\\\text{marginal of $\nu$ on $\mathbb{R}^d$ is $\mu$}}} v_{n,m}^*(t, \nu),
\end{align*}
However, there is no guarantee that $\underline{v}^*(t,\mu)$ is $L$-differentiable with respect to $\mu$. Moreover, even if it is $L$-differentiable, it is not certain that it could be related to $\partial_\nu v_{n,m}^*$. This issue seems to also arise in Step II of the proof of \cite[Theorem 5.1]{cosso_master_2022}. Finally, we remark that assuming non-degenerate volatility does not remedy this issue.
\end{remark}
}
 
 \subsection{Existence}
	\begin{theorem}
	Suppose that Assumption (A) holds. The value function $v$ defined in \eqref{value_function_after_law_invariance} is a viscosity solution of \eqref{HJB_section_5}.
		\label{thm. existence of vis sol}
	\end{theorem}
	\begin{proof}
  {\color{black}
  \textbf{Part 1. $v$ is a viscosity supersolution:} For any $s_0'\in [0,T]$ and $\varphi \in PC_{1}^{1,2}([0,s_0']\times\mathcal{P}_2(\mathbb{R}^d\times A))$, we assume that $v-\varphi$ attains a minimum at $(t_0,\nu_0)$ $\in [0,s_0']\times\mathcal{P}_2(\mathbb{R}^d\times A))$ with a value of $0$. Recalling  that $\mathcal{M}_{t}$ is the set of $\mathcal{F}_{t}^{t}$-measurable random variables, let $\alpha' \in \mathcal{M}_{t_0}$ and  $\xi \in L^2((\Omega^1,\mathcal{F}_t^1,\mathbb{P}^1);\mathbb{R}^d)$ such that $\mathcal{L}(\xi,\alpha') = \nu_0$. Letting $a\in A$ be arbitrary, we define $\alpha_s := a 
  \mathbbm{1}_{[0,t_0)}(s)
  +  \alpha' \mathbbm{1}_{[t_0,T]}(s)$, which belongs to $\mathcal{A}_{t_0}$. Let $(X_s^{t_0,\xi,\alpha})_{s\in [t_0, T]}$ be the solution of the dynamic \eqref{dynamics} with initial time $t_0$, initial data $\xi$ and control $\alpha$ defined in the above. Then for $h>0$ small enough, we use the dynamic programming in Theorem \ref{dpp_thm} and the It\^o formula in Theorem \ref{ito_for_less_regularity} to obtain that
		\begin{align*}
			0\geq\,&\frac{1}{h}\mathbb{E}\Big[(v-\varphi)(t_0,\nu_0) - (v-\varphi)(t_0+h,\mathbb{P}^{W^0}_{(X_{t_0+h}^{t_0,\xi,\alpha},\alpha_{t_0+h})})\Big]\\
			\geq\, &\frac{1}{h}\mathbb{E}\Bigg[\int_{t_0}^{t_0+h}f(s,X_s^{t_0,\xi,\alpha},\mathbb{P}^{W^0}_{X_s^{t_0,\xi,\alpha}},\alpha_s) + \partial_t\varphi(s,\mathbb{P}_{(X_s^{t_0,\xi,\alpha},\alpha_s)}^{W^0})\\
   &\h{30pt}+ \mathbb{E}^1\big[\partial_\mu \varphi(s,\mathbb{P}^{W^0}_{(X_s^{t_0,\xi,\alpha},\alpha_s)})(X_s^{t_0,\xi,\alpha},\alpha_s)\cdot b(s,X_s^{t_0,\xi,\alpha},\mathbb{P}^{W^0}_{X_s^{t_0,\xi,\alpha}},\alpha_s)\big]\\
			&\h{30pt}+\frac{1}{2}\mathbb{E}^1\Big\{\operatorname{tr}\pig[\partial_x\partial_\mu \varphi(s,\mathbb{P}^{W^0}_{(X_s^{t_0,\xi,\alpha},\alpha_s)})(X_s^{t_0,\xi,\alpha},\alpha_s)\sigma(s,X_s^{t_0,\xi,\alpha},\alpha_s)\sigma^\top(s,X_s^{t_0,\xi,\alpha},\alpha_s)\pig]\Big\}\\
&\h{30pt} +\frac{1}{2}\operatorname{tr}\pig[\mathcal{H}_{d\times d}\varphi(s,\mathbb{P}^{W^0}_{(X_s^{t_0,\xi,\alpha},\alpha_s)})
\sigma^0(s)[\sigma^0(s)]^\top\pig]ds\Bigg].
  \end{align*}
We proceed as follows: for any $\varepsilon >0$, 
\begin{align*}
\mathbb{P}^0\left(\sup_{t\in[s,s+h]}\mathcal{W}_2(\mathbb{P}^{W^0}_{X_{t}^{t_0,\xi,\alpha}},\mathbb{P}^{W^0}_{X_{s}^{t_0,\xi,\alpha}})> \varepsilon\right)
\leq \,&\frac{\mathbb{E}^0\displaystyle\sup_{t\in[s,s+h]}\pig[\mathcal{W}_2(\mathbb{P}^{W^0}_{X_{t}^{t_0,\xi,\alpha}},\mathbb{P}^{W^0}_{X_{s}^{t_0,\xi,\alpha}})\pigr]^2}{\varepsilon^2}\\
\leq\,&\frac{\mathbb{E}\displaystyle\sup_{t\in[s,s+h]}|X_{t}^{t_0,\xi,\alpha}-X_{s}^{t_0,\xi,\alpha}|^2}{\varepsilon^2}\\
\leq\,&\frac{Ch}{\varepsilon^2},
\end{align*}
where in the last inequality we have used Proposition \ref{prop. property of X}. The above term goes to $0$ as $h\to 0$, thus the flow of measure  $s\mapsto \mathbb{P}_{X_s^{t_0,\xi,\alpha}}^{W^0}$ is also continuous for $\mathbb{P}$-a.s. $\omega^0 \in \Omega^0$ and for any $s \in[t_0,T]$, so is $s\mapsto \mathbb{P}_{(X_s^{t_0,\xi,\alpha},\alpha_s)}^{W^0}$ for $\mathbb{P}$-a.s. $\omega^0 \in \Omega^0$ and any $s \in[t_0,T]$. From the regularity of the coefficients in Assumption (A), the regularity of the test function, the continuity of the process $s\mapsto X_s^{t_0,\xi,\alpha}$ for $\mathbb{P}$-a.s. $\omega \in \Omega$ and the dominated convergence theorem, we have
\begin{align}
\label{not_important_term}
    &\mathbb{E}\left\{\frac{1}{h}\int_{t_0}^{t_0+h}\Big[f(s,X_s^{t_0,\xi,\alpha},\mathbb{P}^{W^0}_{X_s^{t_0,\xi,\alpha}},\alpha_s) + \partial_t\varphi(s,\mathbb{P}_{(X_s^{t_0,\xi,\alpha},\alpha_s)}^{W^0})+\frac{1}{2}\operatorname{tr}\pig[\mathcal{H}_{d\times d}\varphi(s,\mathbb{P}^{W^0}_{(X_s^{t_0,\xi,\alpha},\alpha_s)})\sigma^0(s)[\sigma^0(s)]^\top\pig]\Big]ds\right\}\nonumber\\
    &\longrightarrow
    \mathbb{E}\bigg\{f(t_0,\xi,\mu_0,\alpha') + \partial_t\varphi(t_0,\nu_0)+\frac{1}{2}\operatorname{tr}\pig[\mathcal{H}_{d\times d}\varphi(t_0,\nu_0)\sigma^0(s)[\sigma^0(s)]^\top\pig]\bigg\}\quad \text{ as }h\to 0.
\end{align}
We claim that the map $[t_0,t_0+h] \ni s\mapsto \mathbb{E}\mathbb{E}^1\pig[\partial_\mu \varphi(s,\mathbb{P}^{W^0}_{(X_s^{t_0,\xi,\alpha},\alpha_s)})(X_s^{t_0,\xi,\alpha},\alpha_s)\cdot b(s,X_s^{t_0,\xi,\alpha},\mathbb{P}^{W^0}_{X_s^{t_0,\xi,\alpha}},\alpha_s)\pig]$ is continuous. Writing $\partial_\mu\varphi_s\cdot b_s: = \partial_\mu \varphi(s,\mathbb{P}^{W^0}_{(X_s^{t_0,\xi,\alpha},\alpha_s)})(X_s^{t_0,\xi,\alpha},\alpha_s)\cdot b(s,X_s^{t_0,\xi,\alpha},\mathbb{P}^{W^0}_{X_s^{t_0,\xi,\alpha}},\alpha_s)$ for ease of notation, the continuity of $\partial_\mu\varphi$, $b$, $X_s^{t_0,\xi,\alpha}$ and $\mathbb{P}^{W^0}_{(X_s^{t_0,\xi,\alpha},\alpha_s)}$ imply the continuity of $\partial_\mu\varphi_s\cdot b_s$ with respect to $s$ for $\mathbb{P}$-a.s. $\omega \in \Omega$. Furthermore, due to Markov's inequality, Definition \ref{PC_121} and Proposition \ref{prop. property of X}, we have
\begin{align}
    \sup_{s\in[t_0,t_0+h]}\mathbb{E}\mathbb{E}^1\Big[|\partial_\mu\varphi_s\cdot b_s| \mathbbm{1}_{\{|\partial_\mu\varphi_s\cdot b_s|\geq M\}}\Big]
    \leq\, &K\sup_{s\in[t_0,t_0+h]}\mathbb{E}\mathbb{E}^1\Big[|\partial_\mu\varphi_s| \mathbbm{1}_{\{|\partial_\mu\varphi_s\cdot b_s|\geq M\}}\Big]\nonumber\\
    \leq\,&K\sup_{s\in[t_0,t_0+h]}\mathbb{E}\mathbb{E}^1\Big[|\partial_\mu\varphi_s| \mathbbm{1}_{\{K|\partial_\mu\varphi_s|\geq M\}}\Big]\nonumber\\
    \leq\,&K\sup_{s\in[t_0,t_0+h]}\sqrt{\mathbb{E}\mathbb{E}^1|\partial_\mu\varphi_s|^2}\sqrt{\mathbb{E}\mathbb{E}^1\mathbbm{1}_{\{K|\partial_\mu\varphi_s|\geq M\}}}\nonumber\\
    \leq\, &\frac{K^2}{M} \sup_{s\in[t_0,t_0+h]}\mathbb{E}\mathbb{E}^1|\partial_\mu\varphi_s|^2\nonumber\\
    \leq\, &\frac{K^2}{M}C_\varphi(1+\mathbb{E}|\xi|^2
    +\mathbb{E}|\alpha'|^2)\to 0 \text{ as }M\to \infty.
    \label{cts_vitali_existence}
\end{align}
The continuity of the concerned map thus follows from Vitali convergence theorem. The continuity of the map $s\mapsto \mathbb{E}\mathbb{E}^1\operatorname{tr}\pig[\partial_x\partial_\mu \varphi(s,\mathbb{P}^{W^0}_{(X_s^{t_0,\xi,\alpha},\alpha_s)})(X_s^{t_0,\xi,\alpha},\alpha_s)\sigma(s,X_s^{t_0,\xi,\alpha},\alpha_s)\sigma^\top(s,X_s^{t_0,\xi,\alpha},\alpha_s)\pig]$ follows similarly. Hence, when $h\to 0$, together with \eqref{not_important_term} we conclude that 
\begin{align*}
&\partial_t\varphi(t_0,\nu_0)+\mathbb{E}\bigg\{f(t_0,\xi,\mu_0,\alpha') + \big[\partial_\mu \varphi(t_0,\nu_0)(\xi,\alpha')\cdot b(t_0,\xi,\mu_0,\alpha')\big]\\
			&+\frac{1}{2}\operatorname{tr}\Big[\partial_x\partial_\mu \varphi(t_0,\nu_0)(\xi,\alpha')\sigma(t_0,\xi,\alpha')[\sigma(t_0,\xi,\alpha')]^\top\Big]+\frac{1}{2}\operatorname{tr} \pig[\mathcal{H}_{d\times d} \varphi(t_0,\nu_0)
   \sigma^0(t_0)[\sigma^0(t_0)]^\top\pig]\bigg\}\leq 0.
		\end{align*}}
\textbf{Part 2. $v$ is a viscosity subsolution:} For any $\varphi \in {PC}_1^{1,2}([0,T]\times \mathcal{P}_2(\mathbb{R}^d))$, we assume that $v-\varphi$ attains a maximum at $(t_0,\mu_0)$ with value $0$. Let $\varepsilon>0$ and $\xi \in L^2(\Omega^1,\mathcal{F}^1,\mathbb{P}^1;\mathbb{R}^d)$ satisfy $\mathcal{L}(\xi) = \mu_0$. There is an $\alpha^{\varepsilon} \in \mathcal{A}_t$ such that 
\begin{align*}
			v(t_0,\mu_0)-\varepsilon <\,& \inf_{h\in[0,T-t_0]}\mathbb{E}\Bigg[\int_{t_0}^{t_0+h} f(r,X_r^{t_0,\xi,\alpha^{\varepsilon}},\mathbb{P}^{W^0}_{X_r^{t_0,\xi,\alpha^{\varepsilon}}},\alpha_r^{\varepsilon})dr+v(t_0+h,\mathbb{P}^{W^0}_{X^{t_0,\xi,\alpha^{\varepsilon}}_{t_0+h}})\Bigg]\\
            \leq\, &\mathbb{E}\Bigg[\int_{t_0}^{t_0+h} f(r,X_r^{t_0,\xi,\alpha^{\varepsilon}},\mathbb{P}^{W^0}_{X_r^{t_0,\xi,\alpha^{\varepsilon}}},\alpha_r^{\varepsilon})dr+v(t_0+h,\mathbb{P}^{W^0}_{X^{t_0,\xi,\alpha^{\varepsilon}}_{t_0+h}})\Bigg],\,\forall\, h\in[0,T-t_0],
		\end{align*}
		where $(X_s^{t_0,\xi,\alpha^{\varepsilon}})$ solves the dynamic \eqref{dynamics} with initial time $t_0$, initial data $\xi$ and control $\alpha^\varepsilon$. Then Theorem \ref{ito_for_less_regularity} implies
		\begin{align*}
			0\leq\,& \frac{1}{h}\mathbb{E}\Big[(v-\varphi)(t_0,\mu_0) - (v-\varphi)(t_0+h,\mathbb{P}^{W^0}_{X_{t_0+h}^{t_0,\xi,\alpha^{\varepsilon}}})\Big]\\
			<\, &\frac{1}{h}\mathbb{E}\Bigg\{\int_{t_0}^{t_0+h}f(s,X_s^{t_0,\xi,\alpha^{\varepsilon}},\mathbb{P}^{W^0}_{X_s^{t_0,\xi,\alpha^{\varepsilon}}},\alpha_s^{\varepsilon}) + \partial_t\varphi(s,\mathbb{P}_{X_s^{t_0,\xi,\alpha^{\varepsilon}}}^{W^0})\\
   &\h{30pt}+ \mathbb{E}^1\Big[\partial_\mu \varphi(s,\mathbb{P}^{W^0}_{X_s^{t_0,\xi,\alpha^{\varepsilon}}})(X_s^{t_0,\xi,\alpha^{\varepsilon}})\cdot b(s,X_s^{t_0,\xi,\alpha^{\varepsilon}},\mathbb{P}^{W^0}_{X_s^{t_0,\xi,\alpha^{\varepsilon}}},\alpha_s^{\varepsilon})\Big]\\
			&\h{30pt}+ \frac{1}{2}\mathbb{E}^1\Big\{\operatorname{tr}\pig[\partial_x\partial_\mu \varphi(s,\mathbb{P}^{W^0}_{X_s^{t_0,\xi,\alpha^{\varepsilon}}})(X_s^{t_0,\xi,\alpha^{\varepsilon}})\sigma(s,X_s^{t_0,\xi,\alpha^{\varepsilon}},\alpha_s^{\varepsilon})\sigma^\top(s,X_s^{t_0,\xi,\alpha^{\varepsilon}},\alpha_s^{\varepsilon})\pig]\Big\}\\
   &\h{30pt}+ \frac{1}{2}\operatorname{tr}\pig[\mathcal{H}\varphi(s,\mathbb{P}^{W^0}_{X_s^{t_0,\xi,\alpha^{\varepsilon}}})\sigma^0(s)\sigma^{0;\top}(s)\pig]ds\Bigg\}+\varepsilon\\
			\leq &\frac{1}{h}\mathbb{E}\Bigg\{\int_{t_0}^{t_0+h}\sup_{a\in A}\bigg(f(s,X_s^{t_0,\xi,\alpha^\varepsilon},\mathbb{P}^{W^0}_{X_s^{t_0,\xi,\alpha^\varepsilon}},a)   
   + \partial_t\varphi(s,\mathbb{P}_{X_s^{t_0,\xi,\alpha^{\varepsilon}}}^{W^0}) \\
   &\h{80pt}+ \mathbb{E}^1\big[\partial_\mu \varphi(s,\mathbb{P}^{W^0}_{X_s^{t_0,\xi,\alpha^\varepsilon}})(X_s^{t_0,\xi,\alpha^\varepsilon})\cdot b(s,X_s^{t_0,\xi,\alpha^\varepsilon},\mathbb{P}^{W^0}_{X_s^{t_0,\xi,\alpha^\varepsilon}},a)\big]\\
			&\h{80pt}+ \frac{1}{2}\mathbb{E}^1\pig\{\operatorname{tr}\pig[\partial_x\partial_\mu \varphi(s,\mathbb{P}^{W^0}_{X_s^{t_0,\xi,\alpha^\varepsilon}})(X_s^{t_0,\xi,\alpha^\varepsilon})\sigma(s,X_s^{t_0,\xi,\alpha^\varepsilon},a)\sigma^\top(s,X_s^{t_0,\xi,\alpha^\varepsilon},a)\pig]\Big\}\\
   &\h{80pt}+ \frac{1}{2}\operatorname{tr}\pig[\mathcal{H}\varphi(s,\mathbb{P}^{W^0}_{X_s^{t_0,\xi,\alpha^{\varepsilon}}})\sigma^0(s)\sigma^{0;\top}(s)\pig]\bigg)ds\Bigg\}+\varepsilon.
   \end{align*}
  Arguing as in \eqref{not_important_term},  we have
   \begin{align*}
       &\frac{1}{h}\mathbb{E}\left\{\int_{t_0}^{t_0+h} \partial_t \varphi(s,\mathbb{P}_{X_s^{t_0,\xi,\alpha^{\varepsilon}}}^{W^0}) + \frac{1}{2}\operatorname{tr}\Big[\mathcal{H}\varphi(s,\mathbb{P}_{X_s^{t_0,\xi,\alpha^{\varepsilon}}}^{W^0})
       \sigma^0(s)\sigma^{0;\top}(s)\Big]ds\right\}\\
       &\longrightarrow \partial_t\varphi(t_0,\mu_0)+\frac{1}{2}\operatorname{tr}\pig[\mathcal{H}\varphi(t_0,\mu_0)
       \sigma^0(t_0)\sigma^{0;\top}(t_0)\pig]\text{ as }h\to 0.
   \end{align*}
We estimate the following terms:
  \begin{align*}
      \RN{1}:=&\frac{1}{h}\mathbb{E}\int_{t_0}^{t_0+h}\sup_{a\in A}\Big\{f(s,X_s^{t_0,\xi,\alpha^\varepsilon},\mathbb{P}^{W^0}_{X_s^{t_0,\xi,\alpha^\varepsilon}},a)-f(t_0,\xi,\mu_0,a)\Big\}ds;\\
      \RN{2}:=&\frac{1}{h}\mathbb{E}\int_{t_0}^{t_0+h}\sup_{a\in A}\Big\{\mathbb{E}^1\pig[\partial_\mu \varphi(s,\mathbb{P}^{W^0}_{X_s^{t_0,\xi,\alpha^\varepsilon}})(X_s^{t_0,\xi,\alpha^\varepsilon})\cdot b(s,X_s^{t_0,\xi,\alpha^\varepsilon},\mathbb{P}^{W^0}_{X_s^{t_0,\xi,\alpha^\varepsilon}},a)\\
      &\h{85pt}-\partial_\mu \varphi(t_0,\mu_0)(\xi)\cdot b(t_0,\xi,\mu_0,a)\pig]\Big\}ds;\\
      \RN{3}:=&\frac{1}{h}\mathbb{E}\int_{t_0}^{t_0+h}\sup_{a\in A}\bigg\{\frac{1}{2}\mathbb{E}^1\Big\{\operatorname{tr}\pig[\partial_x\partial_\mu \varphi(s,\mathbb{P}^{W^0}_{X_s^{t_0,\xi,\alpha^\varepsilon}})(X_s^{t_0,\xi,\alpha^\varepsilon})\sigma(s,X_s^{t_0,\xi,\alpha^\varepsilon},a)\sigma^\top(s,X_s^{t_0,\xi,\alpha^\varepsilon},a)\pig]\Big\}\\
  &\h{85pt}- \frac{1}{2}\mathbb{E}^1\Big\{\operatorname{tr}\pig[\partial_x\partial_\mu \varphi(t_0,\mu_0)(\xi)\sigma(t_0,\xi,a)\sigma^\top(t_0,\xi,a)\pig]\Big\}\bigg\}ds.
  \end{align*}
  Then the above gives 
  \begin{align*}
  &\frac{1}{h}\mathbb{E}\Bigg\{\int_{t_0}^{t_0+h}\sup_{a\in A}\bigg(f(s,X_s^{t_0,\xi,\alpha^\varepsilon},\mathbb{P}^{W^0}_{X_s^{t_0,\xi,\alpha^\varepsilon}},a)
  +\mathbb{E}^1\Big[\partial_\mu \varphi(s,\mathbb{P}^{W^0}_{X_s^{t_0,\xi,\alpha^\varepsilon}})(X_s^{t_0,\xi,\alpha^\varepsilon})\cdot b(s,X_s^{t_0,\xi,\alpha^\varepsilon},\mathbb{P}^{W^0}_{X_s^{t_0,\xi,\alpha^\varepsilon}},a)\Big]\\
			&\h{80pt}+ \frac{1}{2}\mathbb{E}^1\Big\{\operatorname{tr}\pig[\partial_x\partial_\mu \varphi(s,\mathbb{P}^{W^0}_{X_s^{t_0,\xi,\alpha^\varepsilon}})(X_s^{t_0,\xi,\alpha^\varepsilon})\sigma(s,X_s^{t_0,\xi,\alpha^\varepsilon},a)\sigma^\top(s,X_s^{t_0,\xi,\alpha^\varepsilon},a)\pig]\Big\}\bigg)ds\Bigg\}\\ 
  &\leq\RN{1}+\RN{2}+\RN{3}+\mathbb{E}\sup_{a\in A}\Big\{f(t_0,\xi,\mu_0,a)+\partial_\mu \varphi(t_0,\mu_0)(\xi)\cdot b(t_0,\xi,\mu_0,a)\\
  &\h{110pt}+\frac{1}{2}\operatorname{tr}\pig[\partial_x\partial_\mu \varphi(t_0,\mu_0)(\xi)\sigma(t_0,\xi,a)\sigma^\top(t_0,\xi,a)\pig]\Big\}.
  \end{align*}
  We now show that all these three terms $\RN{1}$, $\RN{2}$, $\RN{3} \to 0$ as $h\to 0$. We first investigate $\RN{1}$:
  \begin{align*}
   \RN{1}=\,&\frac{1}{h}\mathbb{E}\int_{t_0}^{t_0+h}\sup_{a\in A}\Big\{f(s,X_s^{t_0,\xi,\alpha^\varepsilon},\mathbb{P}^{W^0}_{X_s^{t_0,\xi,\alpha^\varepsilon}},a)-f(t_0,\xi,\mu_0,a)\Big\}ds\\ 
  \leq\,& \frac{1}{h}\mathbb{E}\int_{t_0}^{t_0+h}K\Big\{|X_s^{t_0,\xi,\alpha^\varepsilon}-\xi|+\mathcal{W}_2(\mathbb{P}_{X_s^{t_0,\xi,\alpha^\varepsilon}}^{W^0},\mu_0)+|s-t_0|^{\beta}\Big\}ds\\
  &\h{-10pt}\longrightarrow  0,
\end{align*}
as $h\to 0$ by Proposition \ref{prop. property of X} and Assumption (A). For $\RN{2}$, we have
\begin{align}
    \RN{2}=\,&\frac{1}{h}\mathbb{E}\int_{t_0}^{t_0+h}\sup_{a\in A}\Big\{\mathbb{E}^1\pig[\partial_\mu \varphi(s,\mathbb{P}^{W^0}_{X_s^{t_0,\xi,\alpha^\varepsilon}})(X_s^{t_0,\xi,\alpha^\varepsilon})\cdot b(s,X_s^{t_0,\xi,\alpha^\varepsilon},\mathbb{P}^{W^0}_{X_s^{t_0,\xi,\alpha^\varepsilon}},a)\nonumber\\
    &\h{70pt}-\partial_\mu \varphi(t_0,\mu_0)(\xi)\cdot b(t_0,\xi,\mu_0,a)\pig]\Big\}ds\nonumber\\
    \leq\, &\frac{1}{h}\mathbb{E}\int_{t_0}^{t_0+h} \sup_{a\in  A}\bigg\{\mathbb{E}^1\Big[\partial_\mu \varphi(s,\mathbb{P}^{W^0}_{X_s^{t_0,\xi,\alpha^\varepsilon}})(X_s^{t_0,\xi,\alpha^\varepsilon})\cdot b(s,X_s^{t_0,\xi,\alpha^\varepsilon},\mathbb{P}^{W^0}_{X_s^{t_0,\xi,\alpha^\varepsilon}},a)\Big]\nonumber\\
    &\,\h{70pt}-\mathbb{E}^1\Big[\partial_\mu \varphi(t_0,\mu_0)(\xi)\cdot b(s,X_s^{t_0,\xi,\alpha^\varepsilon},\mathbb{P}^{W^0}_{X_s^{t_0,\xi,\alpha^\varepsilon}},a)\Big]\bigg\}ds\nonumber\\
    &+\frac{1}{h}\mathbb{E}\int_{t_0}^{t_0+h} \sup_{a\in  A}\mathbb{E}^1\Big[\partial_\mu \varphi(t_0,\mu_0)(\xi)\cdot b(s,X_s^{t_0,\xi,\alpha^\varepsilon},\mathbb{P}^{W^0}_{X_s^{t_0,\xi,\alpha^\varepsilon}},a)-\partial_\mu \varphi(t_0,\mu_0)(\xi)\cdot b(t_0,\xi,\mu_0,a)\Big]ds\nonumber\\
    \leq\, &\frac{K}{h}\mathbb{E}\int_{t_0}^{t_0+h} \mathbb{E}^1\Big[\pig|\partial_\mu \varphi(s,\mathbb{P}^{W^0}_{X_s^{t_0,\xi,\alpha^\varepsilon}})(X_s^{t_0,\xi,\alpha^\varepsilon})-\partial_\mu \varphi(t_0,\mu_0)(\xi)\pig|\Big]ds\nonumber\\
    &+\frac{K}{h}\mathbb{E}\int_{t_0}^{t_0+h} \mathbb{E}^1\Big[\pig|\partial_\mu \varphi(t_0,\mu_0)(\xi)\pig|\pig(|s-t_0|^\beta + |X_s^{t_0,\xi,\alpha^\varepsilon}-\xi|+\mathcal{W}_2(\mathbb{P}^{W^0}_{X_s^{t_0,\xi,\alpha^\varepsilon}},\mu_0)\pig)\Big]ds\nonumber\\
    \leq\,&\frac{K}{h}\mathbb{E}\int_{t_0}^{t_0+h} \mathbb{E}^1\Big[\pig|\partial_\mu \varphi(s,\mathbb{P}^{W^0}_{X_s^{t_0,\xi,\alpha^\varepsilon}})(X_s^{t_0,\xi,\alpha^\varepsilon})-\partial_\mu \varphi(t_0,\mu_0)(\xi)\pig|\Big]ds\label{1314}\\
    &+\frac{\sqrt{3}K}{h}\sqrt{\mathbb{E}\mathbb{E}^1|\partial_\mu \varphi(t_0,\mu_0)(\xi)|^2}\int_{t_0}^{t_0+h} \sqrt{\mathbb{E}\mathbb{E}^1\Big(|s-t_0|^{2\beta} + |X_s^{t_0,\xi,\alpha^\varepsilon}-\xi|^2+\pig[\mathcal{W}_2(\mathbb{P}^{W^0}_{X_s^{t_0,\xi,\alpha^\varepsilon}},\mu_0)\pigr]^2\Big)}ds.\nonumber
\end{align}
The convergence of the term in \eqref{1314} is due to the continuity as proved in \eqref{cts_vitali_existence}. Thus, the term II converges to zero as $h \to 0$ with the aid of Proposition \ref{prop. property of X}. Similar estimate holds for the term $\RN{3}$, thus as $h\to 0$,
\begin{align*}
			0\leq \,&\partial_t\varphi(t_0,\mu_0)+\mathbb{E}\sup_{a\in A}\bigg\{f(t_0,\xi,\mu_0,a) + \big[\partial_\mu \varphi(t_0,\mu_0)(\xi)\cdot b(t_0,\xi,a,\mu_0)\big]\\
			&+\frac{1}{2}\operatorname{tr}\pig[\partial_x\partial_\mu \varphi(t_0,\mu_0)(\xi)\sigma(t_0,\xi,a)\sigma^\top(t_0,\xi,a)\pig]\bigg\}+\frac{1}{2} \operatorname{tr}\pig[\mathcal{H} \varphi(t_0,\mu_0)\sigma^0(t_0)\sigma^{0;\top}(t_0)\pig]+\varepsilon. 
		\end{align*}
	\end{proof}
	\subsection{Comparison Theorem and Uniqueness}
	\begin{theorem}
		Suppose that {\color{black}Assumptions (A*) and (B)} hold. Let  $u_1$, $u_2 : [0,T]\times \mathcal{P}_2(\mathbb{R}^d) \to \mathbb{R}$ be bounded functions such that they are the viscosity subsolution and supersolution (in the sense of Definition \ref{def. of vis sol}) of equation \eqref{HJB_section_5} respectively. Then it holds that $u_1 \leq u_2$ on $[0,T]\times \mathcal{P}_2(\mathbb{R}^d)$. Hence, the viscosity solution of equation \eqref{HJB_section_5} is unique.
		\label{thm compar}
	\end{theorem}
	\begin{proof}
		The following proof is inspired by \cite[Theorem 5.1]{cosso_master_2022}. Recalling the function defined in \eqref{def. v_e=V_e mfc} with $\varepsilon=0$, we shall prove that $u_1\leq v_{0}$ and $v_{0} \leq u_2$ on $[0,T]\times \mathcal{P}_2(\mathbb{R}^d)$.\\
		\hfill\\
		\noindent {\bf Part 1. Proof of  $u_1\leq v_{0}$:}\\
		We prove by contradiction and suppose that there exists $(t_0,\widetilde{\mu}_0)\in [0,T] \times \mathcal{P}_2(\mathbb{R}^d)$ such that
		\begin{align*}
			(u_1-v_{0})(t_0,\widetilde{\mu}_0) >0.
		\end{align*}
		Let $\xi\in L^2(\Omega,\mathcal{F},\mathbb{P};\mathbb{R}^d)$ such that $\mathcal{L}(\xi)=\widetilde{\mu}_0$. For any $k\in \mathbb{N}$, we let $\mu_0^k \in \mathcal{P}_2(\mathbb{R}^d)$ be the law of $\xi\mathbbm{1}_{\{|\xi|\leq k\}}$. We see that $\mu_0^k \in \mathcal{P}_q(\mathbb{R}^d)$ for any $q\geq1$ and
		$$ \pig[\mathcal{W}_2(\mu_0^k,\widetilde{\mu}_0)\pigr]^2\leq \mathbb{E}\pig[|\xi\mathbbm{1}_{\{|\xi|\leq k\}}-\xi|^2\pig] =\int_{|x|>k}|x|^2\widetilde{\mu}_0(dx) \longrightarrow 0$$
		as $k \to \infty$. Therefore, as both $u_1$ and $v_{0}$ are continuous on $[0,T] \times \mathcal{P}_2(\mathbb{R}^d)$, we can find a $k \in \mathbb{N}$ large enough such that $\mu_0 := \mu_0^k\in \mathcal{P}_q(\mathbb{R}^d)$ for any $q\geq1$ and 
		\begin{align}
			(u_1-v_{0})(t_0, \mu_0) >0.
			\label{ineq. u_1-v_0(t_0,mu_0)>0}
		\end{align}
		{\color{black}The purpose of adopting this $\mu_0 := \mu_0^k$ instead of just working with $\widetilde{\mu}_0$ is to apply the approximation theorem in Lemma \ref{lem. conv. of v_e,n,m to v_e} which requires the point $(t,\mu)$ with $\mu$ having a higher moment $q>2$.}\\
  \hfill\\
		\noindent {\bf Step 1A: Choice of the comparison function:}\\
		For any $\varepsilon>0$, $n,m \in \mathbb{N}$, we define $\widecheck{u}_1(t,\mu):=e^{t-t_0}u_1(t,\mu)$ for any $(t,\mu)\in [0,T]\times \mathcal{P}_2(\mathbb{R}^d)$. Recalling the definition in \eqref{def. v_e=V_e mfg}, \eqref{def. approx of f}, \eqref{def. approx of g}, we define similarly for $\widecheck{v}_{\varepsilon,n,m}$, $\widecheck{f}^i_{n,m}$, $\widecheck{f}$ from $v_{\varepsilon,n,m}$, $f^i_{n,m}$, $f$ respectively; we also define $\widecheck{g}:=e^{T-t_0}g$ and $\widecheck{g}^i_{n,m}:=e^{T-t_0}g^i_{n,m}$. By direct computation, we see that $\widecheck{u}_1$ is a viscosity subsolution of the equation
		\begin{equation}
			\left\{
			\begin{aligned}
				&\partial_t u (t,\mu)
				+\int_{\mathbb{R}^d}\sup_{a \in A}\Bigg\{
				\widecheck{f}(t,x,\mu,a)
				+\dfrac{1}{2} \textup{tr}\Big\{\pig[(\sigma\sigma^\top)(t,x,a)\pig]
				\p_x\p_\mu u (t,\mu)(x)\Big\}\\
				&\h{90pt} + \Big\langle b (t,x,\mu,a),\p_\mu u (t,\mu)(x)\Big\rangle\Bigg\}\mu(dx)
    	+\dfrac{1}{2} \textup{tr}\Big\{\pig[(\sigma^0\sigma^{0;\top})(t)\pig]\mathcal{H} u (t,\mu)\Big\}-u (t,\mu)\\
				&=0\h{5pt} \text{for any $(t,\mu) \in [0,T) \times\mathcal{P}_2(\mathbb{R}^d)$};\\
				& u (T,\mu)
				=\int_{\mathbb{R}^d}\widecheck{g}(x,\mu)\mu(dx)
				\h{5pt} \text{for any $\mu \in \mathcal{P}_2(\mathbb{R}^d)$}.
			\end{aligned}
			\right.    
			\label{eq. vis subsol of check u1}
		\end{equation}
		Besides, recalling the definition in \eqref{def. bar of v = tilde of v}, we define $\widecheck{\overline{v}}_{\varepsilon,n,m}(t,\overline{x}):=e^{t-t_0}\overline{v}_{\varepsilon,n,m}(t,\overline{x})$ where $\overline{x}=(x^1,\ldots,x^n)$ for each $x^i \in \mathbb{R}^d$ with $i=1,2,\ldots,n$. By Theorem \ref{thm v_e,n,m}, we obtain that $\widecheck{v}_{\varepsilon,n,m}$ solves the following equation in the classical sense:
		\begin{equation}
			\left\{
			\begin{aligned}
				&\partial_t u(t,\mu)\\
				&+\int_{\mathbb{R}^{dn}}
				\sup_{\overline{a} \in A^n}\Bigg\{\dfrac{1}{n}\sum^n_{i=1}\widecheck{f}^i_{n,m}(t,\overline{x},a^i)
				+\dfrac{1}{2}
				\sum^n_{i=1}\textup{tr}\left[\Big((\sigma\sigma^\top)(t,x^i,a^i)+(\sigma^0\sigma^{0;\top})(t)+\varepsilon^2I_d\Big)\p_{x^ix^i}^2
				\widecheck{\overline{v}}_{\varepsilon,n,m}
				(t,\overline{x})\right]\\
				&\h{65pt}+\dfrac{1}{2}
				\sum^n_{i,j=1,i\neq j}\textup{tr}\left[(\sigma^0\sigma^{0;\top})(t)\p_{x^ix^j}^2
				\widecheck{\overline{v}}_{\varepsilon,n,m}
				(t,\overline{x})\right]\\
				&\h{65pt}+\sum^n_{i=1}\Big\langle b^i_{n,m}(t,\overline{x},a^i)
				,\p_{x^i}\widecheck{\overline{v}}_{\varepsilon,n,m}(t,\overline{x})\Big\rangle\Bigg\}\bbotimes_{k=1}^n\mu(dx^k)
				-u (t,\mu)\\
				&=0\h{5pt} \text{for any $(t,\mu) \in [0,T) \times\mathcal{P}_2(\mathbb{R}^{d})$};\\
				& u(T,\mu)
				=\dfrac{1}{n}\sum^n_{i=1}\int_{\mathbb{R}^{dn}}\widecheck{g}^i_{n,m}(\overline{x})\bbotimes_{k=1}^n\mu(dx^k)
				\h{5pt} \text{for any $\mu \in \mathcal{P}_2(\mathbb{R}^{d})$},
			\end{aligned}
			\right.    
			\label{eq. vis sol of check v_e,n,m}
		\end{equation}
		where we write 	$\displaystyle\bbotimes_{k=1}^n\mu(dx^k):=\mu(dx^1)\otimes\ldots\otimes\mu(dx^n)$ and $\overline{a}=(a^1,\ldots,a^n)$ for each $a^i \in A$ with $i=1,2,\ldots,n$. As $v_{\varepsilon,n,m}$, $u_1$ and $G:=\widecheck{u}_1 - \widecheck{v}_{\varepsilon,n,m}$ are bounded by a constant independent of $\varepsilon,m,n$ and continuous due to Theorem \ref{thm v_e,n,m} and the assumption of this theorem, there is a $\lambda_0>0$ (depending only on $K$, $u_1$, $T$) such that
		\begin{align}
			\sup_{(t,\mu) \in [0,T) \times\mathcal{P}_2(\mathbb{R}^{d})}G(t,\mu)
			=\sup_{(t,\mu) \in [0,T) \times\mathcal{P}_2(\mathbb{R}^{d})}\pig[\widecheck{u}_1(t,\mu) - \widecheck{v}_{\varepsilon,n,m}(t,\mu)
			\pig]
			\leq\widecheck{u}_1(t_0,\mu_0) - \widecheck{v}_{\varepsilon,n,m}(t_0,\mu_0)
			+\lambda_0.
			\label{ineq. bdd of check of u-check of v_e,n,m}
		\end{align}
  Recalling the definitions of $\rho_{1/\delta}$ and $\varphi_\delta$ in \eqref{def. rho_1/d} and \eqref{varphi_and_rho} for $\delta>0$ respectively, we apply Theorem \ref{thm. bdd of varphi_d} to show that there exists $(\widetilde{t}, \widetilde{\mu}) \in[0, T] \times \mathcal{P}_2(\mathbb{R}^d)$ and a sequence $\left\{\left(t_n, \mu_n\right)\right\}_{n \in \mathbb{N}} \subset[0, T] \times \mathcal{P}_2(\mathbb{R}^d)$ converging to $\left(\widetilde{t}, \widetilde{\mu}\right)$ such that 
  \begin{align*}
  \left[SW_2^{1/\delta}(\widetilde{\mu},\mu_0)\right]^2\leq \rho_{1/\delta}\left((\widetilde{t},\widetilde{u}),(t_0,\mu_0)\right)\leq \lambda_0/\delta^2.
\end{align*}
  Therefore, by the stability result stated in \cite[Lemma 1]{pmlr-v139-nietert21a}, there is a constant $C_d>0$ depending only on $d$ such that
  \begin{align*}
      SW_2(\widetilde{\mu},\mu_0) \leq \sqrt{2}SW_2^{1/\delta}(\widetilde{\mu},\mu_0)
      +\dfrac{C_d}{\delta}
      \leq \dfrac{\sqrt{2\lambda_0}+C_d}{\delta}.
  \end{align*}
  The values of $C_d$ may change from line to line in this proof, but it still depends only on $d$. Furthermore, the first equality in \eqref{878} yields that $\kappa_d\mathcal{W}_2(\mu,\delta_0) 
   = SW_2(\mu,\delta_0)$ for all $\mu\in\mathcal{P}_2(\mathbb{R}^d)$, which implies
   \begin{align}
    \mathcal{W}_2(\widetilde{\mu},\mu_0)
    &\leq \mathcal{W}_2(\widetilde{\mu},\delta_0)
    +\mathcal{W}_2(\delta_0,\mu_0)\nonumber\\
    &\leq C_d SW_2(\widetilde{\mu},\mu_0)
    +C_d SW_2(\mu_0,\delta_0)
    +\mathcal{W}_2(\delta_0,\mu_0)\nonumber\\
    &\leq C_d \dfrac{\sqrt{\lambda_0}+1}{\delta}
    +C_d SW_2(\mu_0,\delta_0)
    +\mathcal{W}_2(\delta_0,\mu_0)\nonumber\\
    &\leq C_d \left(\dfrac{\sqrt{\lambda_0}+1}{\delta}
    +SW_2(\mu_0,\delta_0)\right).
    \label{ineq. W(tilde mu,mu_0)}
   \end{align}
Besides, item (2) in Theorem \ref{thm. bdd of varphi_d} implies that		
\begin{align}
			G(t_0,\mu_0)
			=(\widecheck{u}_1- \widecheck{v}_{\varepsilon,n,m})(t_0,\mu_0)
			=(u_1- v_{\varepsilon,n,m})(t_0,\mu_0)
			&\leq \widecheck{u}_1(\widetilde{t},\widetilde{\mu}) - \widecheck{v}_{\varepsilon,n,m}(\widetilde{t},\widetilde{\mu})
			-\delta^2 \varphi_\delta(\widetilde{t},\widetilde{\mu})\nonumber\\
			&\leq \widecheck{u}_1(\widetilde{t},\widetilde{\mu}) - \widecheck{v}_{\varepsilon,n,m}(\widetilde{t},\widetilde{\mu}).
			\label{check u_1 - check v_e,n,m at t_0< at tilde t }
		\end{align}		
		\noindent {\bf Step 1B. Proof of $\widetilde{t}<T$:}\\
		In this step, we shall prove that $\widetilde{t}<T$. Suppose not, we have $\widetilde{t}=T$ and hence \eqref{check u_1 - check v_e,n,m at t_0< at tilde t } implies that $u_1(t_0,\mu_0) - v_{\varepsilon,n,m}(t_0,\mu_0)\leq  \widecheck{u}_1(T,\widetilde{\mu}) - \widecheck{v}_{\varepsilon,n,m}(T,\widetilde{\mu})$. Thus, for $\widehat{\mu}^{n,\overline{x}}:=\dfrac{1}{n}\displaystyle\sum^n_{j=1} \delta_{x^j}$, we can use equations \eqref{eq. vis subsol of check u1}, \eqref{eq. vis sol of check v_e,n,m} to obtain that 
		
		\begin{align*}
			u_1(t_0,\mu_0) - v_{\varepsilon,n,m}(t_0,\mu_0)
			\leq\,& \dfrac{e^{T-t_0}}{n}
			\sum^n_{i=1}\left[
			\int_{\mathbb{R}^{dn}}
			\left(g(x^i,\widetilde{\mu})
			-g^i_{n,m}(x^1,\ldots,x^n)\right)
			\bbotimes_{k=1}^n\widetilde{\mu}(dx^k)
			\right]\\
			=\,& \dfrac{e^{T-t_0}}{n}
			\sum^n_{i=1}\left[
			\int_{\mathbb{R}^{dn}}
			\left(g(x^i,\widetilde{\mu})
			-g(x^i,\widehat{\mu}^{n,\overline{x}})\right)
			\bbotimes_{k=1}^n\widetilde{\mu}(dx^k)
			\right]\\
			&+\dfrac{e^{T-t_0}}{n}
			\sum^n_{i=1}\left[
			\int_{\mathbb{R}^{dn}}
			\left(g(x^i,\widehat{\mu}^{n,\overline{x}})
			-g^i_{n,m}(x^1,\ldots,x^n)\right)
			\bbotimes_{k=1}^n\widetilde{\mu}(dx^k)
			\right].
		\end{align*}
		Using the Lipschitz property of $g$ and (3) of Lemma \ref{lem estimate of b^i_n,m...}, we further have
		\begin{align}
			u_1(t_0,\mu_0) - v_{\varepsilon,n,m}(t_0,\mu_0)
			\leq\,&Ke^{T-t_0}
			\left[
			\int_{\mathbb{R}^{dn}}
			{\color{black}\mathcal{W}_1(\widetilde{\mu},\widehat{\mu}^{n,\overline{x}})}
			\bbotimes_{k=1}^n\widetilde{\mu}(dx^k)
			\right]\nonumber\\
			&+\dfrac{2Ke^{T-t_0}}{n}
			\left[
			m^{dn}\int_{\mathbb{R}^{dn}}
			\left(\sum^n_{i=1}|y^i|\right) \prod^n_{j=1}\Phi(my^j)dy^j
			\right].
			\label{1139}
		\end{align}
{\color{black}From \cite[Theorem 1]{FG15}, there is a constant $C_d>0$ depending on $d$ only and a sequence $\{h_n\}_{n \in \mathbb{N}} \subset \mathbb{R}$ such that 
  \begin{align}
			\int_{\mathbb{R}^{dn}}
			\mathcal{W}_1(\widetilde{\mu},\widehat{\mu}^{n,\overline{x}})
			\bbotimes_{k=1}^n\widetilde{\mu}(dx^k)
			\leq C_{d} \left[\int_{\mathbb{R}^{d}}|x|^{q_0}\widetilde{\mu}(dx)\right]^{1/q_0}h_n.
   \label{def. h_n}
		\end{align}
The sequence $h_n$ and the number $q_0$ are given by
  \[
h_n=
\left\{
\begin{array}{ll}
n^{-1/2} + n^{-(q_0-1)/q_0} & \text{if } d=1; \\
n^{-1/2} \log(1+n) + n^{-(q_0-1)/q_0} & \text{if }  d=2; \\
n^{-1/d} + n^{-(q_0-1)/q_0} & \text{if } d>2,
\end{array}
\right.
\quad \text{with}\quad
q_0=
\left\{
\begin{array}{ll}
3/2 & \text{if } d=1,2; \\
5/3 & \text{if } d>2,
\end{array}
\right.
\]
where $h_n$ satisfies $\lim_{n\to \infty}h_n=0$. Together with \eqref{ineq. W(tilde mu,mu_0)}, we further have
		\begin{align}
			\dfrac{1}{C_dh_n}\int_{\mathbb{R}^{dn}}
			\mathcal{W}_1(\widetilde{\mu},\widehat{\mu}^{n,\overline{x}})
			\bbotimes_{k=1}^n\widetilde{\mu}(dx^k)
			\leq
			\left[\int_{\mathbb{R}^{d}}
			|x|^2
			\widetilde{\mu}(dx)\right]^{1/2}
			&=\mathcal{W}_2(\widetilde{\mu},\delta_0)\leq C_d \left[\dfrac{\sqrt{\lambda_0}+1}{\delta}
    + SW_2(\mu_0,\delta_0)\right].
			\label{ineq. int W_2(tilde mu- hat mu^n,x)}
		\end{align}}
		Hence, from \eqref{1139}, we have,
		\begin{align*}
			u_1(t_0,\mu_0) - v_{\varepsilon,n,m}(t_0,\mu_0)
			\leq\,&Ke^{T-t_0} C_dh_n
		\left[\dfrac{\sqrt{\lambda_0}+1}{\delta}
    + SW_2(\mu_0,\delta_0)\right]\nonumber\\
			&+\dfrac{2Km^{dn}e^{T-t_0}}{n}
			\left[
			\int_{\mathbb{R}^{dn}}
			\left(\sum^n_{i=1}|y^i|\right) \prod^n_{j=1}\Phi(my^j)dy^j
			\right].
		\end{align*}
		Passing $m \to \infty$ and then $n \to \infty$ subsequently, we use the fact that $h_n \to 0$ to yield that 
		\begin{align*}
			u_1(t_0,\mu_0) - \lim_{n \to \infty}\lim_{m \to \infty} v_{\varepsilon,n,m}(t_0,\mu_0)
			\leq 0.
		\end{align*}
		Finally, using Lemmas \ref{lem. |v_e-v_0|<C_5 e} and \ref{lem. conv. of v_e,n,m to v_e}, we pass $\varepsilon \to 0$ to conclude that $u_1(t_0,\mu_0)-v_0(t_0,\mu_0)\leq 0$ which contradicts \eqref{ineq. u_1-v_0(t_0,mu_0)>0}, and thus $\widetilde{t}<T$.\\
  \hfill\\	
		\noindent {\bf Step 1C. Estimate of  $u_1-v_{\varepsilon,n,m}$:}\\
		Recalling from Theorems \ref{thm. bdd of varphi_d}, \ref{thm v_e,n,m}, it implies that $\widecheck{v}_{\varepsilon,n,m}+\delta^2\varphi_\delta$ is in $PC_{1}^{1,2}([0,T]\times\mathcal{P}_2(\mathbb{R}^d))$ (see Definition \ref{PC_121}). From item (3) of Theorem \ref{thm. bdd of varphi_d} with $G:=\widecheck{u}_1 - \widecheck{v}_{\varepsilon,n,m}$ and \eqref{ineq. bdd of check of u-check of v_e,n,m}, we observe that $\widecheck{u}_1 - \widecheck{v}_{\varepsilon,n,m}-\delta^2\varphi_\delta$ attains the  maximum at $(\widetilde{t},\widetilde{\mu})$. {\color{black} Suppose the maximum value is $M^*\in\mathbb{R}$, then $\widecheck{u}_1 - \widecheck{v}_{\varepsilon,n,m}-\delta^2\varphi_\delta-M^*$ attains the maximum with value zero at $(\widetilde{t},\widetilde{\mu})$.}  We use the fact that $\widecheck{u}_1$ is the viscosity subsolution of \eqref{eq. vis subsol of check u1} to see that
\begin{align*}
			0\leq\,&\partial_t (\widecheck{v}_{\varepsilon,n,m}+\delta^2\varphi_\delta) (\widetilde{t},\widetilde{\mu})
			-(\widecheck{v}_{\varepsilon,n,m}+\delta^2\varphi_\delta) (\widetilde{t},\widetilde{\mu})
   {\color{black}-M^*}\\
			&+\int_{\mathbb{R}^d}\sup_{a \in A}\Bigg\{
			\widecheck{f}(\widetilde{t},x,\widetilde{\mu},a)
			+\dfrac{1}{2} \textup{tr}\Big\{\pig[(\sigma\sigma^\top)(\widetilde{t},x,a)\pig]
			\p_x\p_\mu(\widecheck{v}_{\varepsilon,n,m}+\delta^2\varphi_\delta) (\widetilde{t},\widetilde{\mu})(x)\Big\}\\
			&\h{55pt}+\dfrac{1}{2} \textup{tr}\pig[(\sigma^0\sigma^{0;\top})(\widetilde{t}\,)
			\mathcal{H}(\widecheck{v}_{\varepsilon,n,m}+\delta^2\varphi_\delta) (\widetilde{t},\widetilde{\mu})\pig]\\
			&\h{55pt}+ \Big\langle b (\widetilde{t},x,\widetilde{\mu},a),\p_\mu (\widecheck{v}_{\varepsilon,n,m}+\delta^2\varphi_\delta) (\widetilde{t},\widetilde{\mu})(x)\Big\rangle
			\Bigg\}\widetilde{\mu}(dx).
		\end{align*}
		Therefore, as $\widecheck{u}_1(\widetilde{t},\widetilde{\mu}) - \widecheck{v}_{\varepsilon,n,m}(\widetilde{t},\widetilde{\mu})-\delta^2\varphi_\delta(\widetilde{t},\widetilde{\mu}){\color{black}-M^*}=0$ and $\widecheck{v}_{\varepsilon,n,m}$ solves \eqref{eq. vis sol of check v_e,n,m}, we further have
  \begin{align}
			&\h{-10pt}(\widecheck{u}_1-\widecheck{v}_{\varepsilon,n,m})(\widetilde{t},\widetilde{\mu})\nonumber\\
			\leq\,&
			\delta^2\partial_t \varphi_\delta (\widetilde{t},\widetilde{\mu})
			+\int_{\mathbb{R}^d}\sup_{a \in A}\Bigg\{
			\widecheck{f}(\widetilde{t},x,\widetilde{\mu},a)
			+\dfrac{1}{2} \textup{tr}\Big\{\pig[(\sigma\sigma^\top)(\widetilde{t},x,a)\pig]
			\p_x\p_\mu(\widecheck{v}_{\varepsilon,n,m}+\delta^2\varphi_\delta) (\widetilde{t},\widetilde{\mu})(x)\Big\}\nonumber\\
			&\h{105pt}+\dfrac{1}{2}\textup{tr}\Big\{\pig[(\sigma^0\sigma^{0;\top})(\widetilde{t}\,)\pig]
			\mathcal{H}(\widecheck{v}_{\varepsilon,n,m}+\delta^2\varphi_\delta) (\widetilde{t},\widetilde{\mu})\Big\}\nonumber\\
			&\h{105pt}+ \Big\langle b (\widetilde{t},x,\widetilde{\mu},a),\p_\mu (\widecheck{v}_{\varepsilon,n,m}+\delta^2\varphi_\delta) (\widetilde{t},\widetilde{\mu})(x)\Big\rangle
			\Bigg\}\widetilde{\mu}(dx)\nonumber\\
			&-\int_{\mathbb{R}^{dn}}
			\sup_{\overline{a} \in A^n}\Bigg\{\dfrac{1}{n}\sum^n_{i=1}\widecheck{f}^i_{n,m}(\widetilde{t},\overline{x},a^i)
			+\dfrac{1}{2}
			\sum^n_{i=1}\textup{tr}\left[\Big((\sigma\sigma^\top)(\widetilde{t},x^i,a^i)+(\sigma^0\sigma^{0;\top})(\widetilde{t}\,)+\varepsilon^2I_d\Big)\p_{x^ix^i}^2
			\widecheck{\overline{v}}_{\varepsilon,n,m}
			(\widetilde{t},\overline{x})\right]\nonumber\\
			&\h{60pt}+\dfrac{1}{2}
			\sum^n_{i,j=1,i\neq j}\textup{tr}\left[(\sigma^0\sigma^{0;\top})(\widetilde{t}\,)\p_{x^ix^j}^2
			\widecheck{\overline{v}}_{\varepsilon,n,m}
			(\widetilde{t},\overline{x})\right]\nonumber\\
			&\h{60pt}+\sum^n_{i=1}\Big\langle b^i_{n,m}(\widetilde{t},\overline{x},a^i),\p_{x^i}\widecheck{\overline{v}}_{\varepsilon,n,m}(\widetilde{t},\overline{x})\Big\rangle
			\Bigg\}\bbotimes_{k=1}^n\widetilde{\mu}(dx^k).
			\label{ineq. check u1-check v_e,n,m}
		\end{align}
		We estimate term by term. First, Theorem \ref{thm. bdd of varphi_d} and {\color{black}Assumption (A*)} tell us that
\begin{align}
			&\partial_t \varphi_\delta (\widetilde{t},\widetilde{\mu})
			+\int_{\mathbb{R}^d}\sup_{a \in A}\bigg\{
			\left\langle b (\widetilde{t},x,\widetilde{\mu},a),\p_\mu \varphi_\delta(\widetilde{t},\widetilde{\mu})(x)\right\rangle
   +\dfrac{1}{2} \textup{tr}\Big\{\pig[(\sigma\sigma^\top)(\widetilde{t},x,a)
			\pig]
			\p_x\p_\mu\varphi_\delta (\widetilde{t},\widetilde{\mu})(x)\Big\}\bigg\}\widetilde{\mu}(dx)\nonumber\\
&+\dfrac{1}{2} \textup{tr}\Big\{\pig[(\sigma^0\sigma^{0;\top})(\widetilde{t}\,)\pig]
			\mathcal{H} \varphi_\delta(\widetilde{t},\widetilde{\mu})\Big\}\nonumber\\
			&\leq 4 T
			+  \int_{\mathbb{R}^d}K\pig|\p_\mu \varphi_\delta(\widetilde{t},\widetilde{\mu})(x)\pigr|
			+\frac{K^2}{2}\pig|\p_x\p_\mu\varphi_\delta (\widetilde{t},\widetilde{\mu})(x)\pigr|\widetilde{\mu}(dx)
			+\dfrac{K^2}{2}\left|\mathcal{H}\varphi_\delta(\widetilde{t},\widetilde{\mu})
			\right|
			\nonumber\\
			&\leq 4 T
			+  K\sqrt{C_d}
			\left(\int_{\mathbb{R}^d}|x|^2\widetilde{\mu}(dx)
			+\int_{\mathbb{R}^d}|x|^2\mu_0(dx)
			+\dfrac{1+\lambda_0}{\delta^2}\right)^{1/2}\nonumber\\
			&\h{10pt}+\dfrac{K^2}{2}\sqrt{C_d}\delta
			\left(\int_{\mathbb{R}^d}
			|x|^2\mu_0(dx)+\dfrac{1+\lambda_0}{\delta^2}\right)^{1/2}+\frac{K^2}{2}C_d.
			\label{1887}
		\end{align}
Second, we recall the representation of $\widecheck{v}_{\varepsilon,n,m}(t,\mu)
		=e^{t-t_0}v_{\varepsilon,n,m}(t,\mu)$ in \eqref{eq. v_e,n,m=int bar of v_e,n,m} and also (1), (2) of Theorem \ref{thm v_e,n,m}. Inequality in \eqref{ineq.|D v_e,n,m| and |D^2 v_e,n,m|} shows that we can directly compute that
		\begin{align*}
			\p_\mu\widecheck{v}_{\varepsilon,n,m}(t,\mu)(x)
			=\sum^n_{i=1}
			\int_{\mathbb{R}^{d(n-1)}}
			\p_{x^i}\widecheck{\overline{v}}_{\varepsilon,n,m}
			(t,\overline{x})\Big|_{x^i=x}
			\bbotimes^n_{k=1,k\neq i}\mu(dx^k);
		\end{align*}
		and (2) in Theorem \ref{thm v_e,n,m}, Lemma \ref{regular_H} yield that
\begin{align*}
	\mathcal{H}\widecheck{v}_{\varepsilon,n,m}(t,\mu)
   =\sum^n_{i=1}
			\sum^n_{j=1}
			\int_{\mathbb{R}^{dn}}
			\p_{x^ix^j}^2\widecheck{\overline{v}}_{\varepsilon,n,m}
			(t,\overline{x})
			\bbotimes^n_{k=1}\mu(dx^k).
\end{align*}
		Hence, we estimate the term 
	\begin{align}
			&\int_{\mathbb{R}^d}\sup_{a \in A}\bigg\{\widecheck{f}(\widetilde{t},y,\widetilde{\mu},a)
			+\dfrac{1}{2} \textup{tr}\Big\{\pig[(\sigma\sigma^\top)(\widetilde{t},y,a)\pig]
			\p_x\p_\mu\widecheck{v}_{\varepsilon,n,m}(\widetilde{t},\widetilde{\mu})(y)\Big\}\nonumber\\
			&\h{40pt}+\dfrac{1}{2} \textup{tr}\Big\{\pig[(\sigma^0\sigma^{0;\top})(\widetilde{t}\,)\pig]
			\mathcal{H}\widecheck{v}_{\varepsilon,n,m}(\widetilde{t},\widetilde{\mu})\Big\}
			+ \Big\langle b (\widetilde{t},y,\widetilde{\mu},a),\p_\mu \widecheck{v}_{\varepsilon,n,m} (\widetilde{t},\widetilde{\mu})(y)\Big\rangle\Bigg\}\widetilde{\mu}(dy)\nonumber\\
      &=\int_{\mathbb{R}^d}\sup_{a \in A}\Bigg\{\sum^n_{j=1}\int_{\mathbb{R}^{d(n-1)}}\dfrac{1}{n}\widecheck{f}(\widetilde{t},y,\widetilde{\mu},a)\bbotimes^n_{k=1,k\neq j}\widetilde{\mu}(dx^k)\nonumber\\
			&\h{80pt}+\dfrac{1}{2} \textup{tr}\bigg\{\pig[(\sigma\sigma^\top)(\widetilde{t},y,a)\pig]
			\sum^n_{j=1}
			\int_{\mathbb{R}^{d(n-1)}}
			\p_{x^jx^j}^2\widecheck{\overline{v}}_{\varepsilon,n,m}
			(t,\overline{x})\Big|_{x^j=y}
			\bbotimes^n_{k=1,k\neq j}\widetilde{\mu}(dx^k)\bigg\}\nonumber\\
			&\h{80pt}+\dfrac{1}{2} \textup{tr}\Big\{\pig[(\sigma^0\sigma^{0;\top})(\widetilde{t}\,)\pig]
			\sum^n_{j=1}
			\sum^n_{l=1}
			\int_{\mathbb{R}^{dn}}
			\p_{x^jx^l}^2\widecheck{\overline{v}}_{\varepsilon,n,m}
			(t,\overline{x})
			\bbotimes^n_{k=1}\widetilde{\mu}(dx^k)\Big\}\nonumber\\
			&\h{80pt}+ \Big\langle b (\widetilde{t},y,\widetilde{\mu},a),\sum^n_{j=1}
			\int_{\mathbb{R}^{d(n-1)}}
			\p_{x^j}\widecheck{\overline{v}}_{\varepsilon,n,m}
			(t,\overline{x})\Big|_{x^j=y}
			\bbotimes^n_{k=1,k\neq j}\widetilde{\mu}(dx^k)\Big\rangle\Bigg\}\widetilde{\mu}(dy)\nonumber\\
         &\leq\sum^n_{j=1}\int_{\mathbb{R}^{dn}}\sup_{a^j \in A}\Bigg\{\dfrac{1}{n}\widecheck{f}(\widetilde{t},x^j,\widetilde{\mu},a^j)
			+\dfrac{1}{2} \textup{tr}\bigg\{\pig[(\sigma\sigma^\top)(\widetilde{t},x^j,a^j)\pig]
			\p_{x^jx^j}^2\widecheck{\overline{v}}_{\varepsilon,n,m}
			(t,\overline{x})\bigg\}\nonumber\\
			&\h{80pt}+ \Big\langle b (\widetilde{t},x^j,\widetilde{\mu},a^j),
			\p_{x^j}\widecheck{\overline{v}}_{\varepsilon,n,m}
			(t,\overline{x})
			\Big\rangle\Bigg\}\bbotimes^n_{k=1}\widetilde{\mu}(dx^k)\nonumber\\
   			&\h{10pt}+\dfrac{1}{2} \textup{tr}\Big\{\pig[(\sigma^0\sigma^{0;\top})(\widetilde{t}\,)\pig]
			\sum^n_{j=1}
			\sum^n_{l=1}
			\int_{\mathbb{R}^{dn}}
			\p_{x^jx^l}^2\widecheck{\overline{v}}_{\varepsilon,n,m}
			(t,\overline{x})
			\bbotimes^n_{k=1}\widetilde{\mu}(dx^k)\Big\}.
			\label{1941}
		\end{align}
Putting \eqref{1887} and \eqref{1941} into \eqref{ineq. check u1-check v_e,n,m}, we use \eqref{ineq.|D v_e,n,m| and |D^2 v_e,n,m|} and Theorem \ref{thm. bdd of varphi_d} to deduce that
\begin{align}
			&\h{-10pt}(\widecheck{u}_1-\widecheck{v}_{\varepsilon,n,m})(\widetilde{t},\widetilde{\mu})\nonumber\\
			\leq\,&
			4 T \delta^2
			+  K\delta^2\sqrt{C_d}
			\left(\int_{\mathbb{R}^d}|x|^2\widetilde{\mu}(dx)
			+\int_{\mathbb{R}^d}|x|^2\mu_0(dx)
			+\dfrac{1+\lambda_0}{\delta^2}\right)^{1/2}
			\nonumber\\
			&
			+\dfrac{K^2\delta^2}{2}\sqrt{C_d}
			\left(1+\lambda_0+\delta^2\int_{\mathbb{R}^d}
			|x|^2\mu_0(dx)\right)^{1/2}
   +\frac{K^2\delta^2 C_d}{2}
   \nonumber\\
&+\int_{\mathbb{R}^{dn}}
			\sum^n_{i=1}\sup_{a^i \in A}
			\Bigg\{\dfrac{1}{n}
			\widecheck{f}(\widetilde{t},x^i,\widetilde{\mu},a^i)
			-\dfrac{1}{n}\widecheck{f}^i_{n,m}(\widetilde{t},\overline{x},a^i)
			-\dfrac{\varepsilon^2}{2}\textup{tr}\p_{x^ix^i}^2
			\widecheck{\overline{v}}_{\varepsilon,n,m}
			(\widetilde{t},\overline{x})\nonumber\\
			&\h{130pt}+\left\langle b (\widetilde{t},x^i,\widetilde{\mu},a^i)-b^i_{n,m}(\widetilde{t},\overline{x},a^i),\p_{x^i}\widecheck{\overline{v}}_{\varepsilon,n,m}(\widetilde{t},\overline{x})\right\rangle
			\Bigg\}\bbotimes^n_{k=1}\widetilde{\mu}(dx^k)\nonumber\\
			\leq\,&
			4 T \delta^2
			+  K\delta^2\sqrt{C_d}
			\left(\int_{\mathbb{R}^d}|x|^2\widetilde{\mu}(dx)
			+\int_{\mathbb{R}^d}|x|^2\mu_0(dx)
			+\dfrac{1+\lambda_0}{\delta^2}\right)^{1/2}
			\nonumber\\
			&
			+\dfrac{K^2\delta^2}{2}\sqrt{C_d}
			\left(1+\lambda_0+\delta^2\int_{\mathbb{R}^d}
			|x|^2\mu_0(dx)\right)^{1/2}
			+\frac{K^2\delta^2 C_d}{2}
			\nonumber\\
			&
			+\int_{\mathbb{R}^{dn}}
			\sum^n_{i=1}\sup_{a^i \in A}
			\Bigg\{\dfrac{e^{\widetilde{t}-t_0}}{n}
			\pig|f(\widetilde{t},x^i,\widetilde{\mu},a^i)
			-f^i_{n,m}(\widetilde{t},\overline{x},a^i)\pig|
			-\dfrac{\varepsilon^2}{2}\textup{tr}\p_{x^ix^i}^2
			\widecheck{\overline{v}}_{\varepsilon,n,m}
			(\widetilde{t},\overline{x})\nonumber\\
			&\h{130pt}+\dfrac{C_4e^{\widetilde{t}-t_0}}{n}\pig|b (\widetilde{t},x^i,\widetilde{\mu},a^i)-b^i_{n,m}(\widetilde{t},\overline{x},a^i)\pig|
			\Bigg\}\bbotimes^n_{k=1}\widetilde{\mu}(dx^k).
			\label{1271}
		\end{align}
		We use {\color{black}Assumption (A*)} and Lemma \ref{lem estimate of b^i_n,m...} to estimate the following term:
		\begin{align}
			&\h{-10pt}\pig|f(\widetilde{t},x^i,\widetilde{\mu},a^i)
			-f^i_{n,m}(\widetilde{t},\overline{x},a^i)\pig|
			+C_4\pig|b (\widetilde{t},x^i,\widetilde{\mu},a^i)-b^i_{n,m}(\widetilde{t},\overline{x},a^i)\pig|\nonumber\\
			\leq\,&
			\pig|f(\widetilde{t},x^i,\widetilde{\mu},a^i)
			-f (\widetilde{t},x^i,\widehat{\mu}^{n,\overline{x}},a^i)\pig|+\pig|f (\widetilde{t},x^i,\widehat{\mu}^{n,\overline{x}},a^i)
			-f^i_{n,m}(\widetilde{t},\overline{x},a^i)\pig|\nonumber\\
			&+C_4\pig|b(\widetilde{t},x^i,\widetilde{\mu},a^i)
			-b (\widetilde{t},x^i,\widehat{\mu}^{n,\overline{x}},a^i)\pig|
			+C_4\pig|b (\widetilde{t},x^i,\widehat{\mu}^{n,\overline{x}},a^i)
			-b^i_{n,m}(\widetilde{t},\overline{x},a^i)\pig|\nonumber\\
			\leq\,&
			K(1+C_4){\color{black}\mathcal{W}_1(\widetilde{\mu},\widehat{\mu}^{n,\overline{x}})}
   +K(1+C_4)m\int_{\mathbb{R}}\left|\widetilde{t}-\pig[T\wedge(\widetilde{t}-s)^+\pig]\right|^\beta \phi(ms)ds
			\nonumber\\
			&+K(1+C_4)m^{dn}\int_{\mathbb{R}^{dn}}
			\left(|y^i|+\dfrac{1}{n}\sum^n_{j=1}|y^j|\right) \prod^n_{k=1}\Phi(my^k)dy^k.
			\label{1291}
		\end{align}
		Putting \eqref{1291}, \eqref{check u_1 - check v_e,n,m at t_0< at tilde t } and \eqref{ineq. int W_2(tilde mu- hat mu^n,x)} into \eqref{1271}, we see that
\begin{align*}
			&\h{-10pt}(u_1-v_{\varepsilon,n,m})(t_0,\mu_0)\nonumber\\
			\leq\,&
			4 T \delta^2
			+  K\delta^2\sqrt{C_d}
			\left(\left[\dfrac{\sqrt{\lambda_0}+1}{\delta}
    + SW_2(\mu_0,\delta_0)\right]^2
			+\int_{\mathbb{R}^d}|x|^2\mu_0(dx)
			+\dfrac{1+\lambda_0}{\delta^2}\right)^{1/2}\nonumber\\
   &+\dfrac{K^2\delta^2}{2}\sqrt{C_d}
			\left[1+\lambda_0+\delta^2\int_{\mathbb{R}^d}
			|x|^2\mu_0(dx)\right]^{1/2}+\frac{K^2\delta^2 C_d}{2}\noindent\\
   &-\int_{\mathbb{R}^{dn}}
			\dfrac{\varepsilon^2}{2}
			\sum^n_{i=1}\textup{tr}\p_{x^ix^i}^2
			\widecheck{\overline{v}}_{\varepsilon,n,m}
			(\widetilde{t},\overline{x})\bbotimes^n_{k=1}\widetilde{\mu}(dx^k)
			+e^{T-t_0}C_dh_nK(1+C_4)\left[\dfrac{\sqrt{\lambda_0}+1}{\delta}
    + SW_2(\mu_0,\delta_0)\right]\nonumber\\
			&+e^{T-t_0}K(1+C_4)m\int_{\mathbb{R}}\left|\widetilde{t}-\pig[T\wedge(\widetilde{t}-s)^+\pig]\right|^\beta \phi(ms)ds\\
			&+2e^{T-t_0}K(1+C_4)\frac{m^{dn}}{n}\int_{\mathbb{R}^{dn}}
			\left(\sum^n_{j=1}|y^j|\right) \prod^n_{k=1}\Phi(my^k)dy^k.
		\end{align*}
		Using Lemma \ref{lem. |v_e,n,m-v_0,n,m| < C_6e}, \ref{lem. classical sol. of smooth approx.} and \eqref{ineq.|D v_e,n,m| and |D^2 v_e,n,m|}, we first take $\varepsilon \to 0^+$ and then $m \to \infty$ to obtain that
\begin{align*}
			&\h{-10pt}(u_1-\lim_{m\to \infty}v_{0,n,m})(t_0,\mu_0)\nonumber\\
			\leq\,&
			4 T \delta^2
			+  K\delta^2\sqrt{C_d}
			\left(\left[\dfrac{\sqrt{\lambda_0}+1}{\delta}
    + SW_2(\mu_0,\delta_0)\right]^2
			+\int_{\mathbb{R}^d}|x|^2\mu_0(dx)
			+\dfrac{1+\lambda_0}{\delta^2}\right)^{1/2}\\
			&+\dfrac{K^2\delta^2}{2}\sqrt{C_d}
			\left[1+\lambda_0+\delta^2\int_{\mathbb{R}^d}
			|x|^2\mu_0(dx)\right]^{1/2}
			+\frac{K^2\delta^2 C_d}{2}\\
			&+e^{T-t_0}C_dh_nK(1+C_4)\left[\dfrac{\sqrt{\lambda_0}+1}{\delta}
    + SW_2(\mu_0,\delta_0)\right].
		\end{align*}
		By \eqref{def. h_n}, \eqref{ineq. int W_2(tilde mu- hat mu^n,x)}, Lemmas \ref{lem. conv. of v_e,n,m to v_e} and \ref{lem. |v_e-v_0|<C_5 e}, we take $n \to \infty$ and then $\delta \to 0^+$ to obtain that
		\begin{align*}
			(u_1-v_{0})(t_0,\mu_0)
			=\left(u_1-\lim_{n\to \infty}\lim_{m\to \infty}v_{0,n,m}\right)(t_0,\mu_0)
			\leq 0,
		\end{align*}
		which contradicts \eqref{ineq. u_1-v_0(t_0,mu_0)>0}.\hfill\\
		
		\noindent {\bf Part 2. Proof of $u_2 \geq v_0$:}\\
{\color{black}
This step mainly follows the same pattern as in 
		Step II of \cite[Theorem 5.1]{cosso_master_2022} as the presence of the common noise does not alter the arguments of the approximated control problems there. However, due to the technical gap presented in Remark \ref{techincal_gap_def}, we give a sketch of the new proof here. Following the arguments of the first part of Step II of the proof of \cite[Theorem 5.1]{cosso_master_2022}, showing $u_2 \geq v_0$ is equivalent to showing 
\begin{equation*}
    u_2(t, \mu) \geq v^s(t, \nu):= \mathbb{E} \left[ \int_t^s f \left( r, X_r^{t, \xi, \mathfrak{a}}, \mathbb{P}_{X_r^{t, \xi, \mathfrak{a}}}^{W^0}, \mathfrak{a} \right) dr \right] + \mathbb{E}u_2 \left( s, \mathbb{P}_{X_s^{t, \xi, \mathfrak{a}}}^{W^0} \right),
\end{equation*}
for every $(t, \mu) \in [0, T] \times \mathcal{P}_2(\mathbb{R}^d), s \in (t, T], \xi \in L^2(\Omega, \mathcal{F}_t, \mathbb{P}; \mathbb{R}^d)$, with $\mathcal{L}(\xi) = \mu$, and $\mathfrak{a} \in \mathcal{M}_t$, where $\mathcal{M}_t$ denotes the set of $\mathcal{F}_t^t$-measurable random variables $\alpha: \Omega \to A$ and $\nu:=\mathcal{L}(\xi,\mathfrak{a})$. Moreover, it is without loss of generality that $u_2(s,\cdot)$ is Lipschitz continuous for every $s\in [0,T]$ by the arguments of the first part of Step II of the proof of \cite[Theorem 5.1]{cosso_master_2022}. Suppose that there exist $t_0 \in [0, T)$, $s_0 \in (t_0, T]$ and $\mu_0 \in \mathcal{P}_2(\mathbb{R}^d)$ and $\nu_0 \in \mathcal{P}_2(\mathbb{R}^d \times A)$, with $\mu_0$ being the marginal of $\nu_0$ on $\mathbb{R}^d$, such that
\[
    v^{s_0}(t_0, \nu_0) > u_2(t_0, \mu_0).
\]
As in the beginning of Part 1 of this proof, we can suppose that there exists some $q > 2$ such that $\nu_0 \in \mathcal{P}_q(\mathbb{R}^d\times A)$.\\
\hfill\\
The function $\widecheck{u}_2(t, \mu) := e^{t-t_0} u_2(t, \mu)$ is a viscosity supersolution of the following equation 
\begin{align*}
			\left\{\begin{aligned}
				&\partial_t u(t,\mu)
				+\int_{\mathbb{R}^d}
				\sup_{a\in A} \Bigg\{\widecheck{f}(t,x,\mu,a)+b(t,x,\mu,a)\cdot \partial_\mu u(t,\mu)(x) \\
				&+ \dfrac{1}{2}\textup{tr}\Big(\sigma(t,x,a)\big[\sigma(t,x,a)\big]^\top\partial_x\partial_\mu u(t,\mu)(x)\Big)\Bigg\}\mu(dx)
				+ \dfrac{1}{2}\textup{tr}\Big[\sigma^0(t)[\sigma^0(t)]^\top\mathcal{H}u(t,\mu)\Big]-u(t,\mu)\\
			&=0\h{5pt} \text{for $(t,\mu) \in [0,T) \times\mathcal{P}_2(\mathbb{R}^{d})$};\\	&u(T,\mu)=\int_{\mathbb{R}^d}\widecheck{g}(x,\mu)\mu(dx) \h{10pt}\text{for $\mu\in \mathcal{P}_2(\mathbb{R}^d)$}.
			\end{aligned}\right.
		\end{align*}
where $\widecheck{f}(t,x,\mu,a):=e^{t-t_0}f(t,x,\mu,a)$ and $\widecheck{g}(x,\mu):=e^{t-t_0}g(x,\mu)$. That is, for any $s_0\in [0,T]$, for any $\varphi \in PC_{1}^{1,2}([0,s_0]\times\mathcal{P}_2(\mathbb{R}^d\times A))$ such that $\widecheck{u}_2-\varphi$ attains a minimum with a value of $0$ at $(t^*,\nu^*) \in [0,s_0)\times \mathcal{P}_2(\mathbb{R}^d\times A)$, then the following inequality holds:
\begin{align}
    0\geq\,&\partial_t \varphi(t^*,\nu^*)+\int_{\mathbb{R}^d\times A}
				 \Bigg\{\widecheck{f}(t^*,x,\mu^*,a)+b(t^*,x,\mu^*,a)\cdot \partial_\mu \varphi(t^*,\nu^*)(x,a) \nonumber\\
				&+ \dfrac{1}{2}\textup{tr}\Big[\sigma(t^*,x,a)\big[\sigma(t^*,x,a)\big]^\top\partial_x\partial_\mu \varphi(t^*,\nu^*)(x,a)\Big]\Bigg\}\nu^*(dx,da)\nonumber\\
    &+ \dfrac{1}{2}\textup{tr}\Big[\sigma^0(t^*)[\sigma^0(t^*)]^\top\mathcal{H}_{d\times d}\varphi(t^*,\nu^*)\Big] 
    -\varphi(t^*,\nu^*),
    \label{2214}
		\end{align}
  where $\mu^*$ is the marginal of $\nu^*$ on $\mathbb{R}^d$. We regularize the coefficients for the control variable. Let $\Psi:\mathbb{R}^d \to \mathbb{R}^+$ be a compactly supported smooth function satisfying $\int_{\mathbb{R}^{d}}\Psi(y)dy=1$. We extend \(b\) and \(f\) trivially (for example, $f(t,x,\mu,a)=0$ when $a\in\mathbb{R}^d$ is not in $A$) on the space \([0, T] \times \mathbb{R}^{d}\times\mathcal{P}_2(\mathbb{R}^d)\times \mathbb{R}^d\), these extensions are still denoted by \(b\) and \(f\). We further define the functions \(\widetilde{b}^i_{n,m}\) and \(\widetilde{f}^i_{n,m}\) by
\[
\widetilde{b}^i_{n,m}(t, \overline{x}, a) = m^d \int_{\mathbb{R}^d} b^i_{n,m}(t, \overline{x}, a - a') \Psi(m a') \, da',
\]
\[
\widetilde{f}^i_{n,m}(t, \overline{x}, a) = m^d \int_{\mathbb{R}^d} f^i_{n,m}(t, \overline{x}, a - a') \Psi(m a') \, da',
\]
for any $n, m \in \mathbb{N}$, $i = 1,2, \ldots, n$, $\overline{x} = (x^1,x^2, \ldots, x^n) \in \mathbb{R}^{dn}$ and $(t, a) \in [0, T] \times A$. Recalling the compactly supported smooth function $\Phi$ defined in Section \ref{sec. fin-d app}, we also define 
\begin{equation*}
u_{n,m}(t, \overline{x}) = m^{nd} \int_{\mathbb{R}^{dn}} u_2 \left(t, \frac{1}{n} \sum_{j=1}^{n} \delta_{x^j - y^j} \right) \prod_{j=1}^{n} \Phi(my^j) \, dy^j,
\end{equation*}
and
\begin{align*}
v^{s_0}_{n,m}(t, \nu) := \frac{1}{n} \sum_{i=1}^{n} \mathbb{E} \bigg[& \int_{t}^{s_0} \widetilde{f}^i_{n,m} \left( r, \overline{\widetilde{X}}^{1,m,t,\overline{\xi},\overline{\mathfrak{a}}_0}_{r}, \ldots, \overline{\widetilde{X}}^{n,m,t,\overline{\xi},\overline{\mathfrak{a}}_0}_{r}, \mathfrak{a}_0^i \right) dr\\
&+ u_{n,m} \left( s_0, \overline{\widetilde{X}}^{1,m,t,\overline{\xi},\overline{\mathfrak{a}}_0}_{s_0}, \ldots, \overline{\widetilde{X}}^{n,m,t,\overline{\xi},\overline{\mathfrak{a}}_0}_{s_0} \right) \bigg],
\end{align*}
for any $t \in [0,s_0]$ and $\nu \in \mathcal{P}_2(\mathbb{R}^d \times A)$, where $\overline{\xi}=(\overline{\xi}^1,\overline{\xi}^2,\ldots,\overline{\xi}^n) \in L^2(\Omega, \mathcal{F}_t, \mathbb{P}; \mathbb{R}^{dn})$, $\overline{\mathfrak{a}}_0=(\mathfrak{a}_0^1,\mathfrak{a}_0^2,\ldots,\mathfrak{a}_0^n) \in (\mathcal{M}_t)^n$ such that $\mathcal{L}(\overline{\xi},\overline{\mathfrak{a}}_0)=\nu \otimes \ldots \otimes \nu$ and $\overline{\widetilde{X}}^{m,t,\overline{\xi},\overline{\mathfrak{a}}_0}_{s}=\Big(\overline{\widetilde{X}}^{1,m,t,\overline{\xi},\overline{\mathfrak{a}}_0}_{s},\ldots,\overline{\widetilde{X}}^{n,m,t,\overline{\xi},\overline{\mathfrak{a}}_0}_{s}\Big)$ is the solution to \eqref{eq. state perturbed by e BM} with $\overline{\alpha}=\overline{\mathfrak{a}}_0$, $\varepsilon=0$ and $b$ replacing by $\widetilde{b}^i_{n,m}$ for $s\in[t,s_0]$.\\
\hfill\\
For every $n, m \in \mathbb{N}$, $(t, \mu) \in [0, s_0] \times \mathcal{P}_2(\mathbb{R}^d)$, we define $\widecheck{v}_{n,m}^{s_0}:=e^{t-t_0} v_{n,m}^{s_0}$ and similarly for $\widecheck{\widetilde{f}^i}_{n,m}$ and $\widecheck{u}_{n,m}$ from $\widetilde{f}_{n,m}^i$ and $u_{n,m}$. Let
\begin{align*}
\overline{v}_{n,m}^{s_0}(t,\overline{x},\overline{a})
:=\widetilde{v}_{n,m}^{s_0}
(t,\delta_{(x^1,a^1)}\otimes\ldots\otimes\delta_{(x^n,a^n)})
		\end{align*}
 with $\overline{x}=\left(x^1, \ldots, x^n\right) \in \mathbb{R}^{dn}$ and $\overline{a}=\left(a^1, \ldots, a^n\right) \in A^n$, where
\begin{align}
		\widetilde{v}_{n,m}^{s_0}(t,\overline{\nu})
		:=& \dfrac{1}{n}\sum^n_{i=1}\mathbb{E}\Bigg[\int_t^{s_0} \widetilde{f}^i_{n,m}\left(s,\overline{\widetilde{X}}^{1,m,t,\overline{\xi},\overline{\mathfrak{a}}_0}_s,\ldots,
		\overline{\widetilde{X}}^{n,m,t,\overline{\xi},\overline{\mathfrak{a}}_0}_s,
		\overline{\mathfrak{a}}_0^i\right)ds\nonumber \\
		&\h{110pt}+ u_{n,m}\left(\overline{\widetilde{X}}^{1,m,t,\overline{\xi},\overline{\mathfrak{a}}_0}_{s_0},\ldots,
		\overline{\widetilde{X}}^{n,m,t,\overline{\xi},\overline{\mathfrak{a}}_0}_{s_0}\right)\Bigg],
		\label{def. tilde v_e,n,m, with fixed control}
	\end{align}
 with $\overline{\nu}=\mathcal{L}(\overline{\xi},\overline{\mathfrak{a}}_0)$. Moreover, by \cite[Theorem A.8]{cosso_master_2022}, we deduce that $\widecheck{v}_{n,m}^{s_0}$ 
can be represented by 
\begin{align*}
\widecheck{v}_{n,m}^{s_0}(t,\nu)
=e^{t-t_0}\int_{\mathbb{R}^{dn}\times A^n}
\overline{v}_{n,m}^{s_0}(t,\overline{x},\overline{a})\nu(dx^1,da^1)\ldots
\nu(dx^n,da^n),
\end{align*}
and it could be shown by following the proofs of Lemma \ref{lem. classical sol. of smooth approx.}, Theorem \ref{thm v_e,n,m} and \cite[Theorem A.8]{cosso_master_2022} that 
\begin{enumerate}[(1).]
    \item $\overline{v}_{n, m}^{s_0} \in C^{1,2}\left(\left[0, s_0\right] \times \mathbb{R}^{dn} \times A^n \right)$ and $v_{n, m}^{s_0} \in C^{1,2}\left(\left[0, s_0\right] \times \mathcal{P}_2\left(\mathbb{R}^d \times A\right)\right)$;
    \item for any $i=1, \ldots, n$ and $(t,\overline{x},\overline{a}) \in\left[0, s_0\right] \times\mathbb{R}^{dn} \times A^n$, it holds that
$$
\left|\partial_{x^i} \overline{v}_{n, m}^{s_0}(t,\overline{x},\overline{a})\right| \leq \frac{C_K}{n}
$$
where the constant $C_K \geq 0$ depends on $K$, but independent of $n, m$;
\item if $t \in [0,s_0]$ and $\nu \in \mathcal{P}_q\left(\mathbb{R}^d \times A\right)$ for some $q>2$,
then
$$
\lim _{n \rightarrow+\infty} \lim _{m \rightarrow+\infty} v_{n, m}^{s_0}(t, \nu)=v^{s_0}(t, \nu);
$$
\item the function $\widecheck{v}_{n,m}^{s_0}(t,\nu)$ solves the following equation classically:
\[
\left\{
\begin{aligned}
    &\partial_t u(t, \nu) + \overline{\mathbb{E}} \Bigg[ \sum_{i=1}^n \bigg\{ \frac{1}{n} \widecheck{\widetilde{f}^i}_{n,m}(t, \overline{\xi}, \mathfrak{a}_0^i) 
    + \Big\langle\widetilde{b}_{n,m}^i (t, \overline{\xi}, \mathfrak{a}_0^i), \partial_{x^i} \overline{v}_{n,m}^{s_0}(t, \overline{\xi}, \overline{\mathfrak{a}}_0)\Big\rangle  \\
    &\h{90pt} + \frac{1}{2} \text{tr} \left(\left[ (\sigma \sigma^\top) (t, \overline{\xi}^i, \mathfrak{a}_0^i)+(\sigma^0\sigma^{0;\top})(t)\right]  \partial_{x^ix^i}^2 \overline{v}_{n,m}^{s_0}(t, \overline{\xi}, \overline{\mathfrak{a}}_0)\right) \bigg\}\\
    &\h{65pt}+\dfrac{1}{2}
				\sum^n_{i,j=1,i\neq j}\textup{tr}\Big[(\sigma^0\sigma^{0;\top})(t)\partial_{x^ix^j}^2 \overline{v}_{n,m}^{s_0}(t, \overline{\xi}, \overline{\mathfrak{a}}_0)\Big]\Bigg]\\
    &- u(t, \nu)=0; \\
    &u(s_0, \nu) = \overline{\mathbb{E}}  [\widecheck{u}_{n,m}(s_0, \overline{\xi})],
\end{aligned}
\right.
\]
for any $t \in [0,s_0)$ and $\nu \in \mathcal{P}_2(\mathbb{R}^d \times A)$, where $\overline{\xi}=(\overline{\xi}^1,\overline{\xi}^2,\ldots,\overline{\xi}^n) \in L^2(\Omega, \mathcal{F}_t, \mathbb{P}; \mathbb{R}^{dn})$, $\overline{\mathfrak{a}}_0=(\mathfrak{a}_0^1,\mathfrak{a}_0^2,\ldots,\mathfrak{a}_0^n) \in (\mathcal{M}_t)^n$ such that $\mathcal{L}(\overline{\xi},\overline{\mathfrak{a}}_0)=\nu \otimes \ldots \otimes \nu$.
\end{enumerate}
\noindent Now, notice that $v_{n,m}^{s_0}$ is bounded by a constant independent of $n, m$. As a consequence, there exists $\lambda \ge 0$, independent of $n, m$, satisfying
\begin{align}
\label{supersolution_end_eq}
    \sup_{[0, T] \times \mathcal{P}_2(\mathbb{R}^d \times A)} (\widecheck{v}_{n,m}^{s_0} - \widecheck{u}_2) \leq (\widecheck{v}_{n,m}^{s_0} - \widecheck{u}_2)(t_0, \nu_0) + \lambda.
\end{align}
Since $\widecheck{v}_{n,m}^{s_0} - \widecheck{u}_2$ is bounded and continuous, by \eqref{supersolution_end_eq} and Theorem \ref{regularity_metric} applied on $[0, s_0] \times \mathcal{P}_2(\mathbb{R}^d \times A)$ with $G = \widecheck{v}_{n,m}^{s_0} - \widecheck{u}_2$, we obtain that for every $\delta > 0$ there exist $\{(t_k, \nu_k)\}_{k \ge 1} \subset [0, s_0] \times \mathcal{P}_2(\mathbb{R}^d \times A)$ converging to some $(\widetilde{t}, \widetilde{\nu}) \in [0, s_0] \times \mathcal{P}_2(\mathbb{R}^d \times A)$ and a function $\varphi_\delta$ such that items (1)-(4) of Theorem \ref{regularity_metric} hold. It deduces that there is $M_* \in \mathbb{R}$ such that $ \widecheck{v}_{n,m}^{s_0}-\widecheck{u}_2 -\delta^2\varphi_\delta-M_*$ attains the maximum with a value of $0$ at $(\widetilde{t},\widetilde{\mu})$. If $\widetilde{t} = s_0=T$, then we proceed as in Step 1B to get a contradiction. If $\widetilde{t} = s_0<T$, then it is straightforward to draw the contradiction by following \eqref{check u_1 - check v_e,n,m at t_0< at tilde t } and using the definition of $\widetilde{v}_{n,m}^{s_0}(t,\overline{\nu})$ in \eqref{def. tilde v_e,n,m, with fixed control}. If $\widetilde{t}<s_0$, we apply Definition \ref{def. of vis sol} of supersolution and put $\varphi=-\widecheck{v}_{n,m}^{s_0}+\delta^2\varphi_\delta+M_*$ in \eqref{2214}. Then we proceed as in Step 1C and utilize items (1)-(4) in the above to find a contradiction and conclude the proof.}
\end{proof}

\noindent\textbf{Acknowledgement:} The first and second authors would like to thank the helpful discussion with Professor Iosif Pinelis. J.Q.~is partially supported by the National Science and Engineering Research Council of Canada (NSERC). The authors express their gratitude to the two anonymous reviewers for their valuable comments, which significantly enhanced the article. In particular, the authors thank one of the anonymous reviewers for pointing out Remark \ref{pham gap remark}.

\mycomment{
\noindent\textbf{Funding.}
J.Q.~is partially supported by the National Science and Engineering Research Council of Canada (NSERC).\vspace{5pt}

\noindent {\bfseries\Large Declarations}\vspace{5pt}

\noindent \textbf{Competing Interests.}		
The authors have no competing interests to declare.
				
}

\bibliography{Bib2}
\bibliographystyle{abbrv} 
\end{document}